\documentclass[11pt]{amsart}

\usepackage{mathrsfs}
\usepackage{amsfonts}
\usepackage{amssymb}
\usepackage{amsxtra}
\usepackage{dsfont}
\usepackage{color}
\usepackage[compress, sort]{cite}
\usepackage{enumitem}
\usepackage{graphicx}

\usepackage[type1]{newtxtext}
\usepackage{newtxmath}

\usepackage[english,polish]{babel}
\usepackage[T1]{fontenc}

\usepackage[margin=2.5cm, centering]{geometry}
\usepackage[colorlinks,citecolor=blue,urlcolor=blue,bookmarks=true]{hyperref}
\hypersetup{
pdfpagemode=UseNone,
pdfstartview=FitH,
pdfdisplaydoctitle=true,
pdfborder={0 0 0}, 
pdftitle={Stationary states for stable processes with resetting},
pdfauthor={Tomasz Grzywny and Zbigniew Palmowski and Karol Szczypkowski and Bartosz Trojan},
pdflang=en-US
}

\newcommand{\eqdistr}{\stackrel{D}{=}}
\newcommand{\CC}{\mathbb{C}}
\newcommand{\ZZ}{\mathbb{Z}}

\newcommand{\NN}{\mathbb{N}}
\newcommand{\RR}{\mathbb{R}}
\newcommand{\PP}{\mathbb{P}}
\newcommand{\EE}{\mathbb{E}}

\newcommand{\calC}{\mathcal{C}}

\newcommand{\calF}{\mathcal{F}}

\newcommand{\calO}{\mathcal{O}}

\newcommand{\calB}{\mathcal{B}}

\newcommand{\calD}{\mathcal{D}}

\newcommand{\pl}[1]{\foreignlanguage{polish}{#1}}

\newcommand{\norm}[1]{\lvert {#1} \rvert}
\newcommand{\abs}[1]{\lvert {#1} \rvert}
\newcommand{\sprod}[2]{\langle {#1}, {#2} \rangle}

\newcommand{\GL}{\operatorname{GL}}

\newcommand{\ind}[1]{{\mathds{1}_{{#1}}}}
\newcommand{\vphi}{\vartheta}

\renewcommand{\atop}[2]{\genfrac{}{}{0pt}2{#1}{#2}}

\newcommand{\qbinom}[3]{\genfrac{[}{]}{0pt}{}{{#1}}{{#2}}_{{#3}}}

\newcounter{thm}

\newtheorem{main_theorem}[thm]{Theorem}

\newtheorem{theorem}{Theorem}[section]
\newtheorem{proposition}[theorem]{Proposition}
\newtheorem{lemma}[theorem]{Lemma}
\newtheorem{corollary}[theorem]{Corollary}

\numberwithin{equation}{section}
\theoremstyle{definition}
\newtheorem{example}[theorem]{Example}
\newtheorem{remark}[theorem]{Remark}

\title{
Stationary states for stable processes with partial resetting}
\date{\today}

\author{Tomasz Grzywny}
\address{
        \pl{
        Tomasz Grzywny\\
        Wydzia\l{} Matematyki,
        Politechnika Wroc\l{}awska\\
        Wyb. Wyspia\'{n}skiego 27\\
        50-370 Wroc\l{}aw\\
        Poland}
}
\email{tomasz.grzywny@pwr.edu.pl}

\author{Zbigniew Palmowski}
\address{
        \pl{
		Zbigniew Palmowski\\
        Wydzia{\l{}} Matematyki,
        Politechnika Wroc\l{}awska\\
        Wyb. Wyspia\'{n}skiego 27\\
        50-370 Wroc\l{}aw\\
        Poland}
}
\email{zbigniew.palmowski@pwr.edu.pl}

\author{Karol Szczypkowski}
\address{
        \pl{
        Karol Szczypkowski\\
        Wydzia\l{} Matematyki,
        Politechnika Wroc\l{}awska\\
        Wyb. Wyspia\'{n}skiego 27\\
        50-370 Wroc\l{}aw\\
        Poland}
}
\email{karol.szczypkowski@pwr.edu.pl}

\author{Bartosz Trojan}
\address{
	\pl{
		Bartosz Trojan\\
		Wydzia\l{} Matematyki,
        Politechnika Wroc\l{}awska\\
        Wyb. Wyspia\'{n}skiego 27\\
        50-370 Wroc\l{}aw\\
        Poland}
}
\email{bartosz.trojan@gmail.com}

\subjclass[2020]{60G10, 60J35, 60K40, 82C05, 82C31, 35K08,60J65, 60G51, 60G52}
\keywords{asymptotic behavior, Brownian motion, ergodic measure, Fokker--Planck equation, heat kernel, non-equilibrium stationary state, transition density}
\begin{document}
\selectlanguage{english}

\begin{abstract}
We study a $d$-dimensional stochastic process $\mathbf{X}$ which arises from a L\'evy process
$\mathbf{Y}$ by partial resetting, that is the position of the process $\mathbf{X}$ at a Poisson moment equals
$c$ times its position right before the moment, and it develops as $\mathbf{Y}$ between these two consecutive moments,
$c \in (0, 1)$.

We focus on $\mathbf{Y}$ being a strictly $\alpha$-stable process with $\alpha\in (0,2]$ having a transition density: We
analyze properties of the transition density $p$ of the process $\mathbf{X}$. We establish a series representation of
$p$. We prove its convergence as time goes to infinity (ergodicity), and we show that the limit $\rho_{\mathbf{Y}}$
(density of the ergodic measure) can be expressed by means of the transition density of the process $\mathbf{Y}$
starting from zero, which results in closed concise formulae for its moments. We show that the process $\mathbf{X}$
reaches a non-equilibrium stationary state. Furthermore, we check that $p$ satisfies the Fokker--Planck equation,
and we confirm the harmonicity of $\rho_{\mathbf{Y}}$ with respect to the adjoint generator.

In detail, we discuss the following cases: Brownian motion, isotropic and $d$-cylindrical $\alpha$-stable
processes for $\alpha \in (0,2)$, and $\alpha$-stable subordinator for $\alpha\in (0,1)$. We find the
asymptotic behavior of $p(t;x,y)$ as $t\to +\infty$ while $(t,y)$ stays in a certain space-time region.
For Brownian motion, we discover a phase transition, that is a change of the asymptotic behavior
of $p(t;0,y)$ with respect to $\rho_{\mathbf{Y}}(y)$.
\end{abstract}

\maketitle

\section{Introduction}
\label{sec:Intro}

 We consider a semigroup density $p(t;x,y)$ corresponding to a $d$-dimensional L\'evy process with partial resetting, that is, a L\'evy process with additional
proportional jumps realized at independent Poisson epochs.
The process solves the following stochastic differential equation
\[{\mathrm d} X_t=(c-1)X_{t-}{\mathrm d} N_t +{\mathrm d} Y_t\]
where $\mathbf{Y}=(Y_t : t \geq 0)$ is a L\'evy process, $\mathbf{N}=(N_t : t \geq 0)$ is an independent Poisson process and $c\in (0,1)$ is a 
constant.
 Focusing $\mathbf{Y}$ being  a strictly $\alpha$-stable process with $\alpha\in (0,2]$, 
 we give a representation of $p$ in terms of splines satisfying certain recursion. With the help of
this representation we prove the convergence of $p(t;x,y)$ as $t\to +\infty$
to a density $\rho_{\mathbf{Y}}$.
We describe $\rho_{\mathbf{Y}}$, in particular, we provide formulas for its moments. 
Later, we show that the process under considerations
has non-equilibrium stationary state, that is, we prove that
the infinitesimal generator related to $p$ on $L^2(\RR^d, \rho_{\mathbf{Y}}(y) {\rm d} y)$ is not self-adjoint.
Let us recall that the classical ergodic theory concerns the convergence of $p(t;x,y)$ as $t\to +\infty$ for fixed $x,y\in \mathbb{R}^d$. 
Moreover, one of our main results gives the space-time regions
where the uniform asymptototic behavior of $p(t;0,y)$ as $t\to +\infty$ is precisely described.
In particular, we find the regions where
$p(t;0,y)$ is weakly equivalent to $\rho_{\mathbf{Y}}$. Additionally, in the case of Brownian motion we show
that there is a phase transition in behavior along the curve $|y|=2t$.

Let us motivate the study of the process with partial resetting. In the past decade, due to various applications, models that accommodate the resetting mechanism have
been extensively studied. One of them appears in simulating results of procedures dealing
with missing packets in the transmission control protocol (TCP), see \cite{MR1895332, MR2023017}. In the
ideal TCP congestion avoidance algorithm, when a congestion signal is received, e.g. missing packets are
detected, the window transferring size is proportionally decreased and the retransmission starts. Otherwise,
it grows at constant speed. In \cite{Kemperman} it was shown that the evolution of the window size may
be approximated by a continuous time process: a linear drift with partial resetting. More precisely, the
process grows linearly in time and at Poisson epochs experiences downward jumps proportional to its position
right before the epoch. This one-dimensional process is also known as the additive-increase and
multiplicative-decrease process (called AIMD), or the growth-collapse process. For these processes, the main questions
addressed in the literature concerned: stability conditions, the form of the steady-state laws, and identification
of first-passage times, see \cite{MR4546112, MR2840300, MR2576022}. Due to possible perturbations during
data transmission, instead of the constant drift process, it is reasonable to consider models based on  $\alpha$-stable
subordinators which, among other things, motivates our studies.

Another important application where resetting occurs is related to searching for a static target by a method
based on two mechanisms: slow local movements and a relocation procedure. This strategy is widely used
in nature, for example, by foraging animals, biomolecules searching for proteins on DNA, or people looking for an object in
a crowd. The corresponding model consists of a stochastic process representing the first phase, and partial resetting that
mimics the relocation, see \cite{19} and \cite{Bel, Ben, Evans, White} for an extensive list of references.
This motivates us to study multi-dimensional L\'evy processes that are subject to resetting.

Let us explain the resetting procedure in detail. Given a $d$-dimensional L\'evy process $\mathbf{Y}$ a stochastic process
$\mathbf{X}$ is obtained from $\mathbf{Y}$ by partial resetting if at each Poisson moment the position of the process
$\mathbf{X}$ equals a point obtained by multiplying the position of the process right before that moment by a factor
$c\in(0,1)$, and that it develops according to the process $\mathbf{Y}$ between these two consecutive moments.
To be more precise, let $\mathbf{N}$ be a Poisson process with intensity $1$ independent of $\mathbf{Y}$.
Let us denote by $(T_j : j \in \NN)$ the Poisson arrival moments (Poisson epochs) of $\mathbf{N}$. We define $\mathbf{X}$
as
\begin{equation}
	\label{eq:18}
	X_t =
	\begin{cases}
		Y_t, & \text{if } t<T_1 , \\
		c X_{T_n^-} + Y_t - Y_{T_n}, & \text{for } t \in [T_n, T_{n+1}),\, n\in\NN.
	\end{cases}
\end{equation}
We say that $\mathbf{X}$ is obtained from $\mathbf{Y}$ by partial resetting with factor $c \in (0, 1)$. Throughout the paper
we use the following notation
\begin{equation}
	\label{def:m}
	m = c^\alpha.
\end{equation}

It has already been observed by a group of physicists that introducing the resetting to a one-dimensional diffusive movement
of a single particle turns it into a process with a stationary measure, see \cite{MR4525953, Gupta}. The existence of such a measure
is a desired feature, for example, in the context of thermodynamics of certain physical systems, in optimizing the efficiency
of stochastic heat engines, or in modeling search processes.

Before we state our first result, let us recall the $q$-Pochhammer symbol,
\begin{align*}
	(a; q)_0 = 1,\qquad
	(a; q)_n = \prod_{j = 0}^{n-1} (1-aq^j),\qquad
	(a; q)_\infty = \prod_{j = 0}^\infty (1 - a q^j),
\end{align*}
and $q$-Gamma function,
\[
	\Gamma_q(x)=(1-q)^{1-x}\frac{(q;q)_{\infty}}{(q^x;q)_{\infty}}\,, \qquad \qquad x\notin -\mathbb{N}.
\]
The following theorem concerns the ergodicity of the process $\mathbf{X}$.
\begin{main_theorem}
	\label{thm:B}
	Suppose that $\mathbf{Y}$ is a strictly $\alpha$-stable process in $\RR^d$, $\alpha\in(0,2]$, with a transition density $p_0$. Assume that $\mathbf{X}$ is obtained from $\mathbf{Y}$ by partial resetting with factor $c\in(0,1)$. Then
the process $\mathbf{X}$ has a transition density denoted by $p$, such that for each $x, y \in \RR^d$,
	\begin{equation}
		\label{eq:4}
		\rho_{\mathbf{Y}}(y)=\lim_{t\to+\infty} p(t;x,y)
	\end{equation}
	where
	\[
		\rho_{\mathbf{Y}}(y)=
		\frac{1}{(m; m)_\infty}\sum_{k=0}^\infty (-1)^k \frac{m^{\frac{1}{2}k(k-1)}}{(m; m)_k} \,
		\int_0^\infty e^{-m^{-k} u} p_0(u;0,y) {\: \rm d}u.
	\]
 Furthermore, $\rho_{\mathbf{Y}} \in \calC_0^\infty(\RR^d)$, and for every $\gamma \in \RR$,
	\begin{equation}
		\label{eq:3}
		\int_{\RR^d} |y|^{\gamma} \rho_{\mathbf{Y}}(y) {\: \rm d}y
		= \frac{\Gamma(\gamma/\alpha+1)}{\Gamma_m(\gamma/\alpha+1)} (1-m)^{-\gamma/\alpha}\, \mathbb{E}|Y_1|^\gamma.
	\end{equation}
\end{main_theorem}
For a proper interpretation of the quotient $\Gamma(\gamma+1)/\Gamma_m(\gamma+1)$ for $\gamma \in -\NN$, see \eqref{eq:G/G_m}.
The limit \eqref{eq:4} is a consequence of Theorem~\ref{thm:lim_p_t_infty}. The smoothness of $\rho_{\mathbf{Y}}$ as well
as its moments are studied in Proposition \ref{prop:6}. We also check that $p$ solves the \emph{Fokker--Planck equation}, and
$\rho_{\mathbf{Y}}$ is \emph{harmonic} with respect to the operator $L^2(\RR^d, {\rm d}y)$-adjoint to
the generator of the process $\mathbf{X}$, see Theorem~\ref{thm:H+F-P}.

To the best of our knowledge in this context the only rigorously studied process is a linear drift with partial resetting
\cite{14}. Since this process has values in the half-line, a natural tool to study its distribution is the Laplace
transform. For a one-dimensional Brownian motion with partial resetting in \cite{jaifizycy} some results are obtained using the Fourier
transform under the assumption that $\rho_{\mathbf{Y}}$ exists. In both cases the resulting formulas are obtained with the help
of inversion theorems. We tried to apply the same reasoning in the multidimensional case, but it led to expressions that are highly
nontrivial to analyze. In this paper, we develop another approach:
The derivation of Theorem \ref{thm:B} begins with establishing a series representation of $p$ valid for general L\'evy processes
having densities. To be more precise, if $p_0$ is the density of a L\'evy process $\mathbf{Y}$, then
\[
	p(t; x, y)
	=e^{-t} p_0(t; x, y)
	+
	\int_0^t \int_{\RR^d}
	e^{-s} p_0(s; x, z) p(t-s; cz, y) {\: \rm d} z {\: \rm d} s,
\]
and therefore
\[
	p(t; x, y) = e^{-t} \sum_{j = 0}^\infty p_j(t; x, y), \quad \text{for all } x,y \in \RR^d, t > 0
\]
where $(p_n : n \in \NN)$ satisfies the recursion
\[
	p_{n+1}(t; x, y) = \int_0^t \int_{\RR^d} p_0(s; x, z) p_n(t-s; cz, y) {\: \rm d}z {\: \rm d} s,
	\quad\text{for all }x, y \in \RR^d, t >0, n \in \NN_0.
\]
Assuming additionally that $\mathbf{Y}$ is a strictly stable process, we are able to simplify the representation and
we express it by means of an auxiliary family of one-dimensional splines $(P_j : j \in \NN)$. Namely, we get
\begin{equation}
	\label{eq:36}
	p(t; x, y)=e^{-t}p_0(t; 0, y-x)+e^{-t}\sum_{j=1}^\infty t^j \int_0^1 p_0(tu;0,y-c^jx)  P_j(u) {\: \rm d} u
\end{equation}
where $(P_j)$ are given by recursive formulas \eqref{eq:P1u} and \eqref{Pnu}. To simplify the exposition we restrict our
attention to $x=0$. In this case \eqref{eq:36} takes the form
\begin{equation}
	\label{eq:40}
	p(t;0,y)= \int_0^\infty p_0(u;0,y) \: \mu_t({\rm d} u), \quad\text{for all } y \in \RR^d, t > 0
\end{equation}
where $\mu_t$ is a probability measure constructed from splines $(P_j)$ as in \eqref{def:mu_t}.
Clearly,
\[
	p(t;0,0)=p_0(1;0,0)\int_0^\infty u^{-d/\alpha} \: \mu_t( {\rm d} u)
\]
which motivates the analysis of the moments of $\mu_t$. To do so, we first compute $\gamma$ moments for $P_j$ which satisfy a
two-parameter recursive equation, see \eqref{eq:19}. Namely, $\gamma$ moment of $P_j$ is expressed as a linear combination
of $\gamma$ moment of $P_{j+1}$ and $(\gamma-1)$ moment of $P_{j+1}$. Solving the equation for non-natural $\gamma$ is nontrivial
because it connects $\gamma+\ZZ$ moments, but there is no a priori known value in this collection. To solve this problem we
introduce scaled moments and we show that they do have a limit as $\gamma$ tends to minus infinity. It is not hard to compute
zero moments. Then to find negative integer moments with large absolute value we express them, with the help of the recurrence
relation, as a combination of moments of larger orders. However, the recurrence breaks down for $\gamma=0$ which makes it
impossible to use any initial condition. To overcome this
difficulty we use an epsilon trick to reach $\epsilon$ moment. Rough estimates on the moments together with continuity in
$\epsilon$ allow us to conclude. Having the negative integer moments computed we use them to evaluate the limit as $\gamma$
tends to minus infinity.
Next, we deal with non-integer moments. The previous steps permit us to iterate the scaled recursion infinitely many times which
reduces the problem to computing the value of a certain series. For this purpose we use the
$q$-binomial theorem. The missing integer moments are obtained by continuity. Having all moments of $P_j$'s we find
the corresponding moments of the measures $\mu_t$. This gives the tightness of the family $(\mu_t : t > 0)$
while the convergence of natural moments to explicit quantities allows us to deduce the weak convergence of $(\mu_t : t > 0)$
to certain absolutely continuous probability measure $\mu$. In fact, all the moments of $(\mu_t : t > 0)$ converge to
the corresponding moments of $\mu$ and are given explicitly, see Corollary \ref{cor:m-2} and Theorem \ref{thm:weak_conv}.
The weak convergence together with the convergence of moments and the absolute continuity lead to \eqref{eq:4} for $x=0$, that is,
\begin{equation}
	\label{eq:42}
	\rho_{\mathbf{Y}}(y) = \int_0^{\infty} p_0(u;0,y) \: \mu({\rm d} u).
\end{equation}
The general case requires additional work because we have to deal with \eqref{eq:36} in place of \eqref{eq:40}.
To prove the regularity of $\rho_{\mathbf{Y}}$ we use \eqref{eq:42} together with the finiteness of all moments of $\mu$
and the properties of the density $p_0$ of the stable process $\mathbf{Y}$.

Since $\mathbf{X}$ has the stationary measure, one may check its equilibrium. Let us recall that a stochastic process
reaches equilibrium stationary state if a time-reversed process has the same distribution as $\mathbf{X}$, see e.g.
\cite{e21090884, Floreani, Derrida}. Otherwise we say that it reaches the non-equilibrium stationary state (abbreviated as NESS).
One of commonly used tests to determine whether the process reaches NESS is to check if its generator is \emph{not}
self-adjoint in $L^2(\RR^d, \rho_{\mathbf{Y}}(x) {\rm d} x)$. In Theorem \ref{thm:NESS}, by this method
we prove that $\mathbf{X}$ reaches NESS.

The convergence \eqref{eq:4}, can also be written in the following form
\begin{equation}
	\label{eq:5}
	\lim_{t\to+\infty}\frac{p(t;x,y)}{\rho_{\mathbf{Y}}(y)}=1,
\end{equation}
for each $x,y \in \RR^d$, such that $\rho_{\mathbf{Y}}(y)>0$. To better understand the behavior of the transition density $p$
we seek for possibly largest space-time region $\calD \subset \RR_+ \times \RR^d$ such that \eqref{eq:5} holds true uniformly
with respect to $(t, y) \in \calD$ while $t$ tends to infinity (\footnote{$\RR_+ = (0, \infty)$}).
\begin{main_theorem}
	\label{thm:C}
	Suppose that $\mathbf{Y}$ is an isotropic $\alpha$-stable process in $\RR^d$, $\alpha\in(0,2)$.
	Assume that $\mathbf{X}$ is obtained from $\mathbf{Y}$ by partial resetting with factor $c\in(0,1)$. Then for each
	$\kappa \in (0, 1)$, the transition density of $\mathbf{X}$ satisfies
	\begin{equation}
		\label{eq:12}
		\lim_{\atop{t \to \infty}{\norm{y} \to \infty}}
		\sup_{\norm{x} \leq \kappa \norm{y}} \bigg| \frac{p(t; x, y)}{\rho_{\mathbf{Y}}(y)} - 1 \bigg| = 0.
	\end{equation}
\end{main_theorem}
Theorem \ref{thm:C} is a direct consequence of Theorem \ref{thm:ius} and Corollary \ref{cor:ius}. In fact, in Theorem
\ref{thm:ius} we also investigate uniform limits with respect to $c \in (0, 1)$. Similar theorems are obtained for
$\alpha$-stable subordinators $\alpha \in (0, 1)$, see Theorem \ref{thm:s-s}, and $d$-cylindrical $\alpha$-stable processes
$\alpha \in (0, 2)$, see Theorem \ref{thm:cylindrical}. To the best of our knowledge, the limit of the form as in
Theorem \ref{thm:C} has never been studied before in this context. The proof of \eqref{eq:12} proceeds as follows:
We first consider the quotient $(1-m)p(t;x,y)/\nu(y)$ where $\nu$ is the density of the L\'{e}vy measure of the isotropic
$\alpha$-stable process. For simplicity of the exposition, let us consider $x=0$ only. By \eqref{eq:40}, to
prove Theorem \ref{thm:C} we study the asymptotic behavior of the integral
\[
	\int_0^\infty \frac{p_0(u;0,y)}{\nu(y)} \: \mu_t({\rm d} u).
\]
To do so we use the well-known asymptotic behavior of $p_0(u;0,y)/(u \nu(y))$ as $u |y|^{-\alpha}$ tends to $0$, and
the splitting of the integral into two parts: the one that carries most of the mass, this is where the asymptotic is used,
and the remaining one which is negligible as $t$ goes to infinity. The explicit forms of the first and the second moments of
the measure $\mu_t$ are essential, especially to obtain results uniform in the parameter $c$.

Let us observe that Theorem \ref{thm:C} does not cover the Brownian motion case. In fact, the analysis for $\alpha = 2$
is more delicate. However, there is a large space-time region where uniform convergence occurs. We get the following result.
\begin{main_theorem}
	\label{thm:D}
	Suppose that $\mathbf{Y}$ is Brownian motion in $\RR^d$. Assume that $\mathbf{X}$ is obtained from $\mathbf{Y}$ by partial
	resetting with factor $c\in(0,1)$. For each $\delta > 0$, the transition density of $\mathbf{X}$ satisfies
	\begin{equation}
		\label{eq:16}
		p(t; 0, y)
		=
		\rho_{\mathbf{Y}}(y) \big(1 + \calO\big(t^{-1}\big)\big)
	\end{equation}
	as $t$ tends to infinity, uniformly in the region
	\begin{equation}
		\label{eq:14}
		\Big\{(t, y) \in \RR_+ \times \RR^d :
		m^2 +\delta \leq \frac{\norm{y}^2}{4t^2} \leq 1 - \delta \Big\}.
	\end{equation}
\end{main_theorem}
Theorem \ref{thm:D} is implied by Theorem \ref{thm:6} combined with Lemma \ref{lem:densities}. Currently, we do not know
how to get the asymptotic behavior of $p(t; 0, y)$ in the whole space-time region below $m^2 + \delta$, but we expect that
\eqref{eq:16} is uniform in the region
\[
	\Big\{(t, y) \in \RR_+ \times \RR^d : \frac{\norm{y}^2}{4t^2} \leq 1 - \delta \Big\}.
\]
We plan to return to this problem in the future. The following theorem shows that if $\norm{y}$ stays above $2t$, the asymptotic
behavior of $p(t; 0, y)$ is totally different.

\begin{main_theorem}
	\label{thm:F}
	Suppose that $\mathbf{Y}$ is a Brownian motion in $\RR^d$. Assume that $\mathbf{X}$ is obtained from $\mathbf{Y}$ by
	partial resetting with factor $c\in(0,1)$. For each $\delta > 0$, the transition density of $\mathbf{X}$ satisfies
	\[
		p(t; 0, y)
		= e^{-t}
		(4\pi t)^{-\frac{d}{2}} e^{-\frac{|y|^2}{4t}} \bigg\{1 + \bigg(\frac{4t^2}{\norm{y}^2}\bigg)
		\vphi\bigg(\frac{4t^2}{\norm{y}^2}\bigg)+
		\calO\bigg(\frac{t}{\norm{y}^2}\bigg)
		\bigg\}
	\]
	as $t$ tends to infinity, uniformly in the region
	\begin{equation}
		\label{eq:83}
		\Big\{(t, y) \in \RR_+ \times \RR^d : \frac{|y|^2}{4t^2} \geq 1 +\delta \Big\}
	\end{equation}
	where
	\[
		\vphi(x) = \sum_{j = 0}^\infty \frac{1}{(m; m)_{j+1}} x^j, \qquad \norm{x} < 1.
	\]
\end{main_theorem}
Theorem \ref{thm:F} is proved in Theorem \ref{thm:5}. Most of the existing papers focus on analyzing one-dimensional Brownian
motion subject to \emph{total resetting}, that is the process is put to zero at the Poisson moments. In this case one can explore the regenerative structure of Brownian motion with total resetting which is not available when $c \in (0, 1)$.
Let us also emphasize that for total resetting the transition density $p$ can be written explicitly which makes the asymptotic
analysis straightforward, for example by using the large deviation theory. In particular, in \cite{MR3476293} the authors
showed the asymptotic behavior of $p(t; 0, y)$ as $t$ goes to infinity while $|y|/t$ stays constant. Based on certain simulations
in dimensions $1$ and $2$, the change in the asymptotic behavior has been predicted by physicists, see e.g.
\cite{MR4093464, Tal}. An attempt to understand the case of multi-dimensional Brownian motion was done in \cite{MR3225982}
for total resetting.

To prove Theorems \ref{thm:D} and \ref{thm:F} we use the representation \eqref{eq:rep-p-0} of $p$, and the properties of
the splines $P_j$ to show that for $\norm{y} > 2 t m$,
\[
	p(t; 0, y) = e^{-t} (4\pi t)^{-\frac{d}{2}} \Big(
	e^{-\frac{|y|^2}{4t}} + I(t, y) + \text{negligible term}\Big)
\]
where
\[	
	I(t, y) = t \int_m^1 e^{\psi(t, y; u)} {\: \rm d} u
\]
for certain concave function $\psi(t, y; \cdot)$. If $(t, y)$ belongs to the region \eqref{eq:14}, the function
$\psi(t, y; \cdot)$ has the unique critical point in $[m, 1)$. To get the asymptotic behavior of $I(t, y)$ in
the uniform manner we use a variant of the steepest descent method keeping track of the interplay between $t$ and $\norm{y}$.
If $(t, y)$ belongs to the region \eqref{eq:83}, the function $\psi(t, y; \cdot)$ may have the critical point arbitrarily close
to or above $1$. In this case a careful study of the integral leads to a complete description of the asymptotic behavior of
$p(t; 0, y)$ in \eqref{eq:83}.

Our paper is organized as follows: In Section \ref{sec:2} we introduce the splines $(P_j : j \in \NN)$
and measures $(\mu_t : t > 0)$. We then computed their moments in Section \ref{sec:2.1} and Section \ref{sec:2.2}, respectively.
We show that the measures weakly converge to the probability measure $\mu$, see Section \ref{sec:mu_t}. Finally,
in Section \ref{sec:2.4} we define and study basic properties of the function $\rho_{\mathbf{Y}}$.
In Section \ref{sec:stationary} we provide a rigorous definition of the resetting. Then, with help of the splines $(P_j)$, we
construct the representation \eqref{eq:rep-p-0.1} for processes obtained by partial resetting from strictly $\alpha$-stable
processes with densities. Next, we prove that the function $\rho_{\mathbf{Y}}$ is the density of the ergodic measure for
the process $\mathbf{X}$. In the following Section
\ref{sec:3.3} we study the density of $\mathbf{X}$. In Section \ref{sec:3.4} we prove that the process $\mathbf{X}$ reaches NESS.
Section \ref{sec:4} is devoted to the study of the asymptotic behavior of the transition density of $\mathbf{X}$. Finally,
in Appendix \ref{appendix:A} we collect basic properties of strictly $\alpha$-stable processes. In Appendix \ref{appendix:B}
we put further comments about the resetting and connections with the existing literature.

\subsection*{Notation}
We denote by $\NN$ positive integers and $\NN_0 = \NN \cup \{0\}$. We write $f \approx g$ on $U$ or $f(x) \approx g(x)$
for $x \in U$, if there is a constant $C > 0$ such that $C^{-1} g \leq f \leq C g$ for all $x \in U$. As usual
$a \land b= \min\{a,b\}$, $a \vee b=\max\{a,b\}$. By $\lceil x\rceil$ and $\lfloor x \rfloor$ we denote the ceiling and
the floor of a real number $x$. An open ball of radius $r > 0$ centered at $x$ is denoted by $B_r(x)$, and abbreviated to $B_r$
if $x=0$.

\section{Splines $P_j$ and measures $\mu_t$}
\label{sec:2}
In this section we introduce a sequence of splines on $[0, 1]$ which is the building block for the representation of
the transition density of stable processes after resetting.

Given $c \in (0, 1)$ and $\alpha \in (0, 2]$, let us consider a sequence $(W_n : n \in \NN)$ of functions on $\RR_+ \times \RR$
defined as
\begin{align*}
	W_1(t, u) &= \frac{1}{1-m} \ind{(mt, t]}(u), \\
	W_{n+1}(t, u)
	&=
	\ind{(m^{n+1} t, t]}(u)
	\int^{\frac{t-u}{1- m^{n+1}}}_{\frac{m^{n+1} t - u}{m^n - m^{n+1}} \vee 0}
	W_n(t - s, u - m^{n+1} s) {\: \rm d} s, \quad \text{for } n \in \NN
\end{align*}
where $m = c^\alpha$. Observe that $W_n$ is a homogeneous function of degree $n-1$.
\begin{proposition}
	\label{prop:3}
	For every $n \in \NN$ and $\lambda \geq 0$,
	\[
		W_n(\lambda t, \lambda u) = \lambda^{n-1} W_n(t, u), \quad\text{for all } s, u \geq 0.
	\]
\end{proposition}
\begin{proof}
	We argue by induction. There is nothing to prove for $n = 1$. Next, by the change of variables, we obtain
	\begin{align*}
		W_{n+1}(\lambda t, \lambda u)
		&=
		\ind{[m^{n+1}\lambda t, \lambda t)}(\lambda u)
		\int^{\frac{\lambda t - \lambda u}{1-m^{n+1}}}_{\frac{m^n \lambda t - \lambda u}{m^n-m^{n+1}} \vee 0}
		W_n(\lambda t - s, \lambda u - m^{n+1} s) {\: \rm d} s \\
		&=
		\lambda
		\ind{[m^{n+1} t, t)}(u)
		\int^{\frac{t - u}{1-m^{n+1}}}_{\frac{m^n t - u}{m^n-m^{n+1}} \vee 0}
		W_n(\lambda t - \lambda s, \lambda u - m^{n+1} \lambda s) {\: \rm d} s.
	\end{align*}
	Now, by the inductive assumption
	\[
		W_{n+1}(\lambda t, \lambda u)
		=
		\lambda
		\ind{[m^{n+1} t, t)}(u)
		\int^{\frac{t - u}{1-m^{n+1}}}_{\frac{m^n t - u}{m^n-m^{n+1}} \vee 0}
		\lambda^{n-1} W_n(t - s, u - m^{n+1} s) {\: \rm d} s = \lambda^n W_{n+1}(t, u),
	\]
	and the proposition follows.
\end{proof}

For each $n \in \NN$, we set
\begin{equation}
	\label{eq:21}
	P_n(u) = W_n(1, u), \quad u \geq 0.
\end{equation}
\begin{proposition}
	\label{prop:1}
	The sequence $(P_n : n \in \NN)$ satisfies
	\begin{align}
		P_1(u) &= \frac{1}{1-m} \ind{(m, 1]}(u), \label{eq:P1u}\\
		P_{n+1}(u) &= \big(u-m^{n+1}\big)_+^n \int_u^1 \frac{P_n(v)}{(v-m^{n+1})^{n+1}} {\: \rm d}v,
		\quad \text{for } n \in \NN. \label{Pnu}
	\end{align}
	In particular, $P_n$ is supported on $[m^n, 1]$.
\end{proposition}
\begin{proof}
	For $u \in (m^{n+1} 1]$, we have
	\begin{align*}
		P_{n+1}(u)
		=
		W_{n+1}(1, u)
		&=
		\int_{\frac{m^n-u}{m^n-m^{n+1}} \vee 0}^{\frac{1-u}{1-m^{n+1}}}
		W_n(1-s, u - m^{n+1} s) {\: \rm d} s \\
		&= \int_{\frac{m^n-u}{m^n-m^{n+1}} \vee 0}^{\frac{1-u}{1-m^{n+1}}}
		(1-s)^{n-1} P_n\bigg(\frac{u-m^{n+1}s }{1 - s} \bigg) {\: \rm d} s.
	\end{align*}
	Setting
	\[
		w = \frac{u-m^{n+1} s }{1-s} = \frac{u-m^{n+1}}{1-s} + m^{n+1},
	\]
	we obtain
	\begin{align*}
		P_{n+1}(u)
		&=
		\int_{u \vee m^n}^1 \bigg(\frac{u-m^{n+1}}{w - m^{n+1}} \bigg)^{n-1} P_n(w) \frac{u-m^{n+1}}{(w-m^{n+1})^2}
		{\: \rm d} w,
	\end{align*}
	as claimed.
\end{proof}
Later we will need the following fact.
\begin{proposition}
	\label{prop:2}
	For each $n \in \NN$, $P_n$ is a spline supported on $[m^n, 1]$, such that
	\begin{equation}
		\label{eq:8}
		P_n(u) = \frac{1}{(n-1)!} \frac{1}{(m; m)_n} (1-u)^{n-1}, \quad \text{for all } u \in [m, 1],
	\end{equation}
	and
	\begin{equation}
		\label{eq:9}
		P_n(u) \leq \frac{1}{(n-1)!} \frac{1}{(m; m)_n} (1-u)^{n-1}, \quad \text{for all } u \in [0, 1].
	\end{equation}
\end{proposition}
\begin{proof}
	Let us recall that for $a<b$, $n\in \NN$ and $v>a$ we have
	\[
		\int \frac{(v-b)^{n-1}}{(v-a)^{n+1}}{\: \rm d} v = \frac1{n}\frac1{b-a} (v-b)^n(v-a)^{-n}.
	\]
	Hence, taking $a=m^{n+1}$ and $b=1$, for all $n \geq 1$ and $u \in [m, 1]$ we get
	\begin{align}
		\label{eq:integral_m}
		(u - m^{n+1})^n \int_u^1 \frac{(1-v)^{n-1}}{(v-m^{n+1})^{n+1}} {\: \rm d} v
		= \frac{1}{n} \frac{1}{1-m^{n+1}} (1-u)^n.
	\end{align}
	The proof of \eqref{eq:8} is by induction with respect to $n \in \NN$. For $n = 1$ the formula trivially holds true.
	Next, using the inductive hypothesis and Proposition \ref{prop:1} we can write
	\begin{align*}
		P_{n+1}(u)
		&=
		(u - m^{n+1})^n \int_u^1 \frac{P_n(v)}{(v-m^{n+1})^{n+1}} {\: \rm d} v \\
		&=
		\frac{1}{(n-1)!} \frac{1}{(m; m)_{n}}
		(u - m^{n+1})^n \int_u^1 \frac{(1-v)^{n-1}}{(v-m^{n+1})^{n+1}} {\: \rm d} v \\
		&=
		\frac{1}{n!} \frac{1}{(m; m)_{n+1}} (1-u)^n
	\end{align*}
	where the last equality is a consequence of \eqref{eq:integral_m}. Similarly, one can prove the estimates \eqref{eq:9}.
\end{proof}

In Section \ref{sec:repr}, we prove that the transition density of the process $\mathbf{X}$ obtained from strictly
$\alpha$-stable process in $\RR^d$, $\alpha \in (0, 2]$, by resetting with factor $c \in (0, 1)$, can be written in a closed form
with help of measures $(\mu_t : t > 0)$ where
\begin{align}
	\label{def:mu_t}
	\mu_t({\rm d} u)
	=e^{-t}\delta_{t}({\rm d} u)
	+
	e^{-t} \sum_{j=1}^\infty t^j P_j(u/t) \frac{{\rm d} u}{t}.
\end{align}
Note that $\mu_t$ is a probability measure supported on $[0, t]$. Our aim is to compute the moments of $\mu_t$.
To do so we start by computing the moments of $P_j$'s.

\subsection{Moments of $P_j$'s}
\label{sec:2.1}
In this section we compute moments of splines $P_j$'s. The main result of this section is Theorem \ref{thm:all-moments}.
For $\gamma \in \RR$ and $j \in \NN$, we set
\begin{equation}
	\label{eq:28b}
	\mathbb{A}(\gamma, j) = \int_0^1 u^{\gamma} P_j(u) {\: \rm d} u.
\end{equation}
We start by proving several auxiliary lemmas.
\begin{lemma}
	\label{lem:2}
	For all $\gamma \in \RR$ and $j \in \NN$,
	\begin{equation}
		\label{eq:19}
		(j+1+\gamma) \mathbb{A}(\gamma, j+1) =
		\mathbb{A}(\gamma, j) + \gamma m^{j+1} \mathbb{A}(\gamma-1, j+1).		
	\end{equation}
\end{lemma}
\begin{proof}
	For the proof, we write
	\begin{align*}
		\mathbb{A}(\gamma, j+1)
		&= \int_{m^{j+1}} ^1 u^{\gamma} \big(u - m^{j+1}\big)^j
		\int_u^1 \frac{P_j(v)}{(v-m^{j+1})^{j+1}} {\: \rm d} v {\: \rm d}u \\
		&=
		\int_{m^{j+1}}^1 \frac{P_j(v)}{(v-m^{j+1})^{j+1}}
		\int_{m^{j+1}}^v u^{\gamma} \big(u - m^{j+1}\big)^j {\: \rm d} u {\: \rm d} v.
	\end{align*}
	Next, by the integration by parts, we obtain the following
	\begin{align*}
		\int_{m^{j+1}}^v u^{\gamma} \big(u-m^{j+1}\big)^j {\: \rm d} u
		&=
		\frac{1}{j+1} v^{\gamma} \big(v-m^{j+1}\big)^{j+1}
		-
		\frac{\gamma}{j+1} \int_{m^{j+1}}^v u^{\gamma-1} \big(u-m^{j+1}\big)^{j+1} {\: \rm d} u \\
		&=
		\frac{1}{j+1} v^{\gamma} \big(v-m^{j+1}\big)^{j+1}
		-
		\frac{\gamma}{j+1} \int_{m^{j+1}}^v u^{\gamma} \big(u-m^{j+1}\big)^j {\: \rm d} u \\
		&\phantom{=\frac{1}{j+1} v^{-\gamma} \big(v-m^{j+1}\big)^{j+1}}
		+
		m^{j+1} \frac{\gamma}{j+1} \int_{m^{j+1}}^v u^{-\gamma-1} \big(u-m^{j+1}\big)^j {\: \rm d} u
	\end{align*}
	which leads to
	\begin{align*}
		(j+1 + \gamma)
		\int_{m^{j+1}}^v u^{\gamma} \big(u-m^{j+1}\big)^j {\: \rm d} u
		=
		v^{\gamma} \big(v-m^{j+1}\big)^{j+1}
		+
		\gamma
		m^{j+1} \int_{m^{j+1}}^v u^{\gamma-1} \big(u-m^{j+1}\big)^j {\: \rm d} u
	\end{align*}	
	and the proposition follows.
\end{proof}

\begin{corollary}
	\label{cor:A0}
	For each $n\in\NN$,
	\[
		\mathbb{A}(0, n)=\frac1{n!}.
	\]
\end{corollary}
We next introduce scaled moments. For $\gamma \in \RR$ and $n \in \NN$, we set
\begin{align}
	\label{defG}
	\mathbb{B}(\gamma, n)=
	\bigg(\prod_{k=1}^n \frac{k+\gamma}{1-m^{k+\gamma}}\bigg)
	\int_0^1 u^{\gamma} P_n(u)\: {\rm d}u.
\end{align}
If $\gamma$ is a negative integer the value of the product is understood in the limiting sense. Namely, if $\gamma \in -\NN$
and $n \geq \abs{\gamma}$, then
\begin{equation}
	\label{eq:43}
	\begin{aligned}
	\prod_{k = 1}^n \frac{k+\gamma}{1-m^{k+\gamma}}
	&=
	\lim_{\epsilon \to 0^+}
	\prod_{k = 1}^n \frac{k+\gamma+\epsilon}{1-m^{k+\gamma+\epsilon}} \\
	&=\frac{1}{-\log m} \prod_{\stackrel{k = 1}{k \neq \abs{\gamma}}}^n  \frac{k+\gamma}{1-m^{k+\gamma}}.
	\end{aligned}
\end{equation}
Clearly, for every $n\in\NN$ the function $\RR \ni \gamma \mapsto \mathbb{B}(\gamma, n)$ is continuous.
\begin{lemma}
	\label{lem:C_lim_-infty}
	For every $n\in\NN$,
	\[
		\lim_{\gamma \to -\infty} \mathbb{B}(\gamma,n+1)= m^{-\frac{n(n-1)}{2}} \frac{n!}{(1-m)^n} P_{n+1}(m^n).
	\]
\end{lemma}
\begin{proof}
	Given two real functions $f$, $g$ defined on $(-\infty, a)$, $a \in \RR$, we write $f \sim g$ as $x \to -\infty$, if
	\[
		\lim_{x \to -\infty} \frac{f(x)}{g(x)} = 1.
	\]
	Let us observe that
	\begin{equation}
		\label{eq:prod_beh}
		\prod_{k=1}^{n+1} \frac{k+\gamma}{1-m^{k+\gamma}}
		\sim
		(-\gamma)^{n+1} m^{-\gamma (n+1) -\frac{(n+2)(n+1)}{2}}
		\quad\text{as } \gamma \to -\infty.
	\end{equation}
	Since for $\gamma<0$,
	\[
		\int_{m^n}^1 u^{\gamma} P_{n+1}(u)\: {\rm d}u
		\leq (m^n)^{\gamma} \int_0^1 P_{n+1}(u)\: {\rm d}u=\frac{(m^n)^{\gamma}}{(n+1)!},
	\]
	we get
	\[
		\lim_{\gamma \to -\infty} \int_{m^n}^1 u^{\gamma} P_{n+1}(u)\: {\rm d}u = 0.
	\]
	Using now Proposition~\ref{prop:1} we obtain
	\begin{align}
		\label{eq:main_part}
		\int_{m^{n+1}}^{m^n}u^\gamma P_{n+1}(u) \: {\rm d}u
		&=
		\int_{m^{n+1}}^{m^n} u^\gamma \big(u-m^{n+1}\big)^n {\: \rm d}u
		\frac{P_{n+1}(m^n)}{(m^n-m^{n+1})^n}.
	\end{align}
	For $\gamma < -n -1$, we can write
	\begin{align*}
		\int_{m^{n+1}}^{m^n} u^\gamma \big(u-m^{n+1}\big)^n \: {\rm d}u
		&= (m^{n+1})^{\gamma+n+1} \int_m^1 u^{-\gamma-n-2}(1-u)^n  \: {\rm d}u\\
		&= (m^{n+1})^{\gamma+n+1}
		\bigg(\frac{\Gamma(-\gamma-n-1)\Gamma(n+1)}{\Gamma(-\gamma)} + \int_0^m u^{-\gamma-n-2}(1-u)^n  \: {\rm d}u \bigg)
	\end{align*}
	where in the last equality we expressed the beta function in terms of the gamma function. Since for $\gamma < -n -2$,
	\[
		\int_0^m u^{-\gamma-n-2}(1-u)^n  {\: \rm d}u \leq m^{-\gamma -n-1},
	\]
	and
	\[
		\frac{\Gamma(-\gamma-n-1)}{\Gamma(-\gamma)}
		=(-1)^{n+1}\bigg(\prod_{k=1}^{n+1} (k+\gamma)\bigg)^{-1}
		\sim
		(-\gamma)^{-n-1}
		\quad\text{as } \gamma \to -\infty,
	\]
	we conclude that
	\[
		\int_{m^{n+1}}^{m^n} u^\gamma \big(u-m^{n+1}\big)^n \: {\rm d}u
		\sim
		(m^{n+1})^{\gamma+n+1}
		(-\gamma)^{-n-1}
		\Gamma(n+1),
		\quad\text{as } \gamma \to -\infty,
	\]
	which together with \eqref{eq:prod_beh} and \eqref{eq:main_part} leads to
	\begin{align*}
		\mathbb{B}(\gamma, n+1)
		&\sim m^{-\gamma (n+1) -\frac{(n+2)(n+1)}{2}}
		(m^{n+1})^{\gamma+n+1} \Gamma(n+1) \frac{P_{n+1}(m^n)}{(m^n-m^{n+1})^n}
		\quad\text{as } \gamma \to -\infty.
	\end{align*}
	This completes the proof.
\end{proof}

Let us recall that for $q > 0$, the $q$-bracket  of $x \in \RR$ is defined as
\[
	[x]_q = \frac{1-q^x}{1-q}.
\]
For $1 \leq k \leq n$, the $q$-binomial coefficient is
\[
	\qbinom{n}{k}{q} = \frac{[n]_q!}{[k]_q! [n-k]_q!}
\]
where
\begin{align*}
	[n]_q! &= [1]_q [2]_q \ldots [n]_q, \quad n \in \NN,\\
	[0]_q! &= 1.
\end{align*}

\begin{lemma}
	\label{lem:C_neg_int_gamma}
	For all $n\in\NN$ and $\gamma\in-\NN$ satisfying $\gamma\leq -(n+1)$,
	\begin{equation}
		\label{eq:22}
		\mathbb{B}(\gamma,n)=\frac1{(m; m)_n}.
	\end{equation}
\end{lemma}
\begin{proof}
	Let $\gamma \in \RR \setminus \{-1\}$. By Lemma \ref{lem:2}, for all $n \in \NN$, we have
	\[
		(1-m^{n+1+\gamma+1})\mathbb{B}(\gamma+1,n+1)=\mathbb{B}(\gamma+1,n)+(1-m^{\gamma+1}) m^{n+1} \, \mathbb{B}(\gamma,n+1),
	\]
	or equivalently,
	\begin{align}
		\label{eq:C_rec}
		\mathbb{B}(\gamma,n+1)
		=- \frac{\mathbb{B}(\gamma+1,n)}{(1-m^{\gamma+1}) m^{n+1}}
		+ \frac{[n+1+\gamma+1]_m}{[\gamma+1]_m } \frac1{m^{n+1}} \mathbb{B}(\gamma+1,n+1).
	\end{align}
	Therefore, if $\gamma \in \RR \setminus \{-1, -2\}$,
	\begin{align*}
		\mathbb{B}(\gamma,n+1)
		&=
		- \frac{\mathbb{B}(\gamma+1,n)}{(1-m^{\gamma+1}) m^{n+1}}
		+ \frac{[n+1+\gamma+1]_m}{[\gamma+1]_m } \frac1{m^{n+1}} \mathbb{B}(\gamma+1,n+1) \\
		&=
		-\frac{\mathbb{B}(\gamma+1,n)}{(1-m^{\gamma+1})m^{n+1}}
		-\frac{[n+1+\gamma+1]_m}{[\gamma+1]_m } \frac1{m^{n+1}}
		\frac{\mathbb{B}(\gamma+2,n)}{(1-m^{\gamma+1})m^{n+1}}  \\
		&\phantom{=}
		+ \frac{[n+1+\gamma+1]_m}{[\gamma+1]_m }
		\frac{[n+1+\gamma+2]_m}{[\gamma+2]_m }
		\Big(\frac1{m^{n+1}}\Big)^2
		\mathbb{B}(\gamma+2,n+1).
	\end{align*}
	Hence, if $\gamma \in \RR \setminus \{-1, -2, \ldots, -r\}$, for $r \in \NN$, we can iterate \eqref{eq:C_rec}
	to get
	\begin{equation}
		\label{eq:23}
		\begin{aligned}
		\mathbb{B}(\gamma, n+1)
		&=- \sum_{k=0}^{r-1} \bigg\{\prod_{\ell=1}^k \frac{[n+1+\gamma+\ell]_m}{[\gamma+\ell]_m} \bigg\}
		\Big(\frac1{m^{n+1}}\Big)^k  \frac{\mathbb{B}(\gamma+k+1,n)}{(1-m^{\gamma+k+1})m^{n+1}}\\
		&\phantom{=}
		+ \bigg\{ \prod_{\ell=1}^r \frac{[n+1+\gamma+\ell]_m}{[\gamma+\ell]_m}\bigg\}\Big(\frac1{m^{n+1}}\Big)^r
		\mathbb{B}(\gamma+r,n+1).
		\end{aligned}
	\end{equation}
	Now, to prove \eqref{eq:22} we proceed by induction with respect to $n \in \NN$. Let $n = 1$ and $\gamma \leq -2$.
	By \eqref{eq:P1u}, we get
	\begin{align*}
		\mathbb{B}(\gamma, 1)
		&= \frac{1 + \gamma}{1 - m^{\gamma+1}} \int_0^1 u^\gamma P_1(u) {\: \rm d} u \\
		&= \frac{1 + \gamma}{1 - m^{\gamma+1}} \frac{1}{1-m} \int_m^1 u^{\gamma} {\: \rm d} u = \frac{1}{1-m}.
	\end{align*}
	Suppose that \eqref{eq:22} holds true for $n \in \NN$. Setting $\gamma_\epsilon = -(n+2) + \epsilon$ for
	$\epsilon \in (0,1)$, by continuity we have
	\[
		\mathbb{B}(-(n+2),n+1)
		=
		\lim_{\epsilon\to 0^+}
		\mathbb{B}(-(n+2)+\epsilon, n+1).
	\]
	Using \eqref{eq:23} with $r=n+2$ we can write
	\[
		 \mathbb{B}(-(n+2), n+1) = I_1+I_2+I_3+I_4
	\]
	where
	\begin{align*}
		I_1&=
		-\lim_{\epsilon\to 0^+} \frac{\mathbb{B}(-n-1+\epsilon,n)}{(1-m^{-n-1+\epsilon})m^{n+1}},\\
		I_2&=
		-\lim_{\epsilon\to 0^+} \sum_{k=1}^n
		\bigg\{\prod_{\ell=1}^k \frac{[-1+\epsilon+\ell]_m}{[\gamma_\epsilon+\ell]_m} \bigg\}
		\Big(\frac1{m^{n+1}}\Big)^k  \frac{\mathbb{B}(-n-1+\epsilon+k,n)}{(1-m^{-n-1+\epsilon+k})m^{n+1}},\\
		I_3
		&= -\lim_{\epsilon\to 0^+}
		\bigg\{\prod_{\ell=1}^{n+1} \frac{[n+1+\gamma_\epsilon+\ell]_m}{[\gamma_\epsilon+\ell]_m}\bigg\}
		\Big(\frac1{m^{n+1}}\Big)^{n+1}  \frac{\mathbb{B}(\epsilon,n)}{(1-m^{\epsilon})m^{n+1}},\\
		\intertext{and}
		I_4
		&=\lim_{\epsilon\to 0^+}
		\bigg\{ \prod_{\ell=1}^{n+2} \frac{[n+1+\gamma_\epsilon+\ell]_m}{[\gamma_\epsilon+\ell]_m}\bigg\}
		\Big(\frac1{m^{n+1}}\Big)^{n+2} \mathbb{B}(\epsilon,n+1).
	\end{align*}
	Thanks to the inductive hypothesis, we get
	\[
		I_1= - \frac{\mathbb{B}(-n-1,n)}{(1-m^{-n-1})m^{n+1}}=\frac1{(m;m)_{n+1}}.
	\]
	Since $\lim_{\epsilon \to 0^+} [\epsilon]_m = 0$, we also have $I_2 = 0$. Furthermore,
	\[
		I_3=- \bigg\{\prod_{\ell=2}^{n+1} \frac{[n+1+\gamma_0+\ell]_m}{[\gamma_0+\ell]_m}\bigg\}
		\Big(\frac1{m^{n+1}}\Big)^{n+2}  \frac{\mathbb{B}(0,n)}{1-m^{-n-1}},
	\]
	and
	\[
		I_4=
		\bigg\{ \prod_{\ell=2}^{n+1} \frac{[n+1+\gamma_0+\ell]_m}{[\gamma_0+\ell]_m}\bigg\}
		\Big(\frac1{m^{n+1}}\Big)^{n+2}\frac{1-m^{n+1}}{1-m^{-n-1}} \mathbb{B}(0,n+1).
	\]
	In view of Corollary~\ref{cor:A0} we have $-\mathbb{B}(0,n) + (1-m^{n+1}) \mathbb{B}(0,n+1) = 0$, thus $I_3 + I_4 = 0$.
	Summarizing, we obtain
	\[
		\mathbb{B}(-(n+2),n+1)=\frac1{(m;m)_{n+1}}.
	\]
	Next, we claim that for all $k \in \NN$,
	\[
		\mathbb{B}(-(n+1+k),n+1)=\frac1{(m;m)_{n+1}}.
	\]
	Indeed, if the formula holds true for $k \in \NN$, then by \eqref{eq:C_rec} we can write
	\begin{align*}
		\mathbb{B}(-(n+1+k+1),n+1)
		&=-\frac{\mathbb{B}(-(n+1+k),n)}{m^{n+1}-m^{-k}}+\frac{1-m^{-k}}{m^{n+1}-m^{-k}}\mathbb{B}(-(n+1+k),n+1)\\
		&=\frac1{m^{n+1}-m^{-k}}
		\bigg(\frac{-1}{(m;m)_n}+\frac{1-m^{-k}}{(m;m)_{n+1}} \Big) = \frac1{(m;m)_{n+1}},
	\end{align*}
	as claimed. This completes the proof of the lemma.
\end{proof}

Combining Lemmas~\ref{lem:C_lim_-infty} and \ref{lem:C_neg_int_gamma} one can compute the value of
$P_{n+1}(m^n)$ explicitly.
\begin{corollary}
	For $n\in\NN$,
	\[
		P_{n+1}(m^n)= m^{\frac{n(n-1)}{2}}  \frac1{n!} \frac{(1-m)^n}{(m;m)_{n+1}}.
	\]
\end{corollary}

We are now ready to compute moments of $P_n$.
\begin{theorem}
	\label{thm:all-moments}
	For all $n\in\NN$ and $\gamma\in \RR$,
	\begin{align*}
		\int_0^1 u^{\gamma} P_n(u)\: {\rm d}u
		= \frac1{(m;m)_n} \bigg\{\prod_{k=1}^n \frac{1-m^{k+\gamma}}{k+\gamma}\bigg\}.
	\end{align*}
	If $\gamma \in -\NN$ the value of the product is understood in the limiting sense, see \eqref{eq:43}.
\end{theorem}
\begin{proof}
	In view of \eqref{defG} our aim is to prove that
	\begin{equation}
		\label{eq:25}
		\mathbb{B}(\gamma,n)=\frac1{(m;m)_{n}}
	\end{equation}
	for all $n \in \NN$ and $\gamma \in \RR$. The reasoning is by induction with respect to $n \in \NN$. For $n = 1$,
	thanks to Proposition \ref{prop:1}, the formula holds true. Suppose that it holds for $n \geq 1$.
	By Lemma~\ref{lem:C_lim_-infty} the limit $\lim_{\gamma\to -\infty} \mathbb{B}(\gamma,n+1)$ exists. Furthermore, by
	Lemma~\ref{lem:C_neg_int_gamma} we have the equality
	\begin{align}
		\label{eq:C_lim_-infty_value}
		\lim_{\gamma\to -\infty} \mathbb{B}(\gamma,n+1)=\frac1{(m;m)_{n+1}}.
	\end{align}
	Let us first consider $\gamma \in \RR \setminus \ZZ$. By \eqref{eq:19}, we have
	\begin{equation}
		\label{eq:24}
		\mathbb{B}(\gamma,n+1)
		=\frac{\mathbb{B}(\gamma,n)}{(1-m^{n+1+\gamma})}+\frac{[\gamma]_m}{[n+1+\gamma]_m} m^{n+1}
		\mathbb{B}(\gamma-1,n+1).
	\end{equation}
	Hence, by repeated application of \eqref{eq:24} for $r \in \NN$ we get
	\begin{align*}
		\mathbb{B}(\gamma,n+1)
		&=
		\sum_{k=0}^{r-1} \bigg\{\prod_{\ell=0}^{k-1} \frac{[\gamma-\ell]_m}{[n+1+\gamma-\ell]_m} \bigg\}
		(m^{n+1})^k  \frac{\mathbb{B}(\gamma-k,n)}{(1-m^{n+1+\gamma-k})}\\
		&\phantom{=}
		+ \bigg\{\prod_{\ell=0}^{r-1} \frac{[\gamma-\ell]_m}{[n+1+\gamma-\ell]_m}\bigg\}
		(m^{n+1})^r \mathbb{B}(\gamma-r,n+1).
	\end{align*}
	Notice that
	\begin{align}
		\nonumber
		\prod_{\ell=0}^{r-1} \frac{[\gamma-\ell]_m}{[n+1+\gamma-\ell]_m}
		&=
		\frac{[n+1+\gamma-r]_m \ldots [1+\gamma-r]_m}{[n+1+\gamma]_m\ldots [1+\gamma]_m } \\
		\label{eq:prod_unified}
		&= \frac{(m^{1+\gamma-r};m)_{n+1}}{(m^{1+\gamma};m)_{n+1}}.
	\end{align}
	Therefore, by \eqref{eq:C_lim_-infty_value},
	\begin{align}
		\label{eq:C-remainder}
		\lim_{r\to +\infty}
		\bigg\{ \prod_{\ell=0}^{r-1} \frac{[\gamma-\ell]_m}{[n+1+\gamma-\ell]_m}\bigg\}
		(m^{n+1})^r \mathbb{B}(\gamma-r,n+1)
		=
		\frac{m^{\frac{(n+1)n}{2}} (- m^{1+\gamma})^{n+1}}{(m^{1+\gamma};m)_{n+1}} \frac1{(m;m)_{n+1}}.
	\end{align}
	Similarly, by \eqref{eq:prod_unified}, for $k\in \NN$,
	\[
		\bigg\{\prod_{\ell=0}^{k-1} \frac{[\gamma-\ell]_m}{[n+1+\gamma-\ell]_m} \bigg\}
		\frac1{(1-m^{n+1+\gamma-k})}=\frac{(m^{1+\gamma -k};m)_n}{(m^{1+\gamma};m)_{n+1}}.
	\]
	Hence, using the inductive hypothesis and the $q$-binomial theorem,
	\begin{align}
		\lim_{r\to \infty}
		&\sum_{k=0}^{r-1} \bigg\{\prod_{\ell=0}^{k-1} \frac{[\gamma-\ell]_m}{[n+1+\gamma-\ell]_m} \bigg\}
		(m^{n+1})^k  \frac{\mathbb{B}(\gamma-k,n)}{(1-m^{n+1+\gamma-k})} \nonumber  \\
		&= \frac1{(m^{1+\gamma};m)_{n+1}} \frac1{(m;m)_n}
		\sum_{k=0}^\infty  (m^{1+\gamma -k};m)_n (m^{n+1})^k \nonumber \\
		&= \frac1{(m^{1+\gamma};m)_{n+1}} \frac1{(m;m)_n}
		\sum_{k=0}^\infty \bigg(  \sum_{\ell=0}^n m^{\frac{\ell(\ell-1)}{2}}  \qbinom{n}{\ell}{m} (-m^{1+\gamma-k})^\ell \bigg)
		(m^{n+1})^k \nonumber \\
		&= \frac1{(m^{1+\gamma};m)_{n+1}} \frac1{(m;m)_n}
		\sum_{\ell=0}^n m^{\frac{\ell(\ell-1)}{2}}  \qbinom{n}{\ell}{m} (-m^{1+\gamma})^\ell
		\Big(\sum_{k=0}^\infty (m^{n+1-\ell})^k\Big) \nonumber \\
		&= \frac1{(m^{1+\gamma};m)_{n+1}} \frac1{(m;m)_{n+1}}
		\sum_{\ell=0}^n m^{\frac{\ell(\ell-1)}{2}}  \qbinom{n+1}{\ell}{m} (-m^{1+\gamma})^\ell. \label{eq:C-series}
	\end{align}
	Adding \eqref{eq:C-remainder} and \eqref{eq:C-series}, and using the $q$-binomial theorem we obtain \eqref{eq:25} for
	$\gamma \in \RR \setminus \ZZ$, which by continuity holds true for all $\gamma \in \RR$.
\end{proof}

We are going to derive alternative formulations of Theorem~\ref{thm:all-moments} that will be useful in Section~\ref{sec:2.2}.
For this purpose let us recall the generalized binomial coefficient and its $q$-version, $0<q<1$: For
$x,y\in\RR$ such that $x,x-y,y \notin -\NN$, we set
\[
	\binom{x}{y}=\frac{\Gamma(x+1)}{\Gamma(y+1)\Gamma(x-y+1)}, \qquad \mbox{and} \qquad\quad
	\qbinom{x}{y}{q}=\frac{\Gamma_q(x+1)}{\Gamma_q(y+1)\Gamma_q(x-y+1)}
\]
where
\[
	\Gamma_q(x)=(1-q)^{1-x}\frac{(q;q)_{\infty}}{(q^x;q)_{\infty}}.
\]
Notice that $x\Gamma(x)=\Gamma(x+1)$ and $[x]_q \Gamma_q(x)=\Gamma_q(x+1)$ for $x\notin -\NN$. Therefore,
for each $\gamma \in \RR \setminus (-\NN)$ and $N \in \NN_0$,
\[
	\frac{\Gamma(\gamma+1)}{\Gamma_m(\gamma+1)}
	=
	(1-m)^{-N} \bigg\{\prod_{k = 1}^{N} \frac{1 - m^{\gamma+k}}{\gamma + k} \bigg\}
	\frac{\Gamma(\gamma+N+1)}{\Gamma_m(\gamma+N+1)}.
\]
We can thus continuously extend $\Gamma(\gamma+1)/\Gamma_m(\gamma+1)$ to all $\gamma \in \RR$, by setting
\begin{align}
	\label{eq:G/G_m}
	\frac{\Gamma(\gamma+1)}{\Gamma_m(\gamma+1)} =
	(1-m)^{\gamma}
	\bigg\{\prod_{k=1}^{|\gamma|-1} \frac{1-m^{k+\gamma}}{k+\gamma}\bigg\}
	\log(1/m),
	\quad\text{for } \gamma \in -\NN.
\end{align}
In particular, one can extend the natural domain of
\[
	\frac{\qbinom{n+\gamma}{\gamma}{m}}{\binom{n+\gamma}{\gamma}}
\]
to all $\gamma \in \RR$.
\begin{corollary}
	\label{cor:m-1}
	For all $n\in\NN$ and $\gamma\in \RR$,
	\begin{align*}
		\int_0^1 u^\gamma P_n(u) {\: \rm d} u
		&= \frac{1}{n!} \frac{\qbinom{n+\gamma}{\gamma}{m}}{\binom{n+\gamma}{\gamma}}
		= \frac{\Gamma(\gamma+1)}{\Gamma_m(\gamma+1)}
		\frac{\Gamma_m(n+\gamma+1)}{\Gamma(n+\gamma+1)} \frac{1}{\Gamma_m(n+1)}.
	\end{align*}
	If $\gamma \in -\NN$, the value of the right-hand side is understood in the limiting sense, see \eqref{eq:G/G_m}.
	Furthermore, if $\gamma \in \RR \setminus (-\NN)$, then
	\[
		\int_0^1 u^\gamma P_n(u) {\: \rm d} u
		=\frac{\Gamma(\gamma+1)}{\Gamma_m(\gamma+1)}
		(1-m)^{-\gamma} \frac1{\Gamma(n+\gamma+1)} \frac{(m^{n+1};m)_{\infty}}{(m^{n+\gamma+1};m)_{\infty}}.
	\]
\end{corollary}

\subsection{Moments of $\mu_t$}
\label{sec:2.2}
In this section we compute the moments of $\mu_t$. For each $\gamma \in \RR$ and $t > 0$, by \eqref{def:mu_t}
\begin{align}
	\label{eq:moments-mu_t-P_j}
	\int_0^{\infty} u^\gamma \mu_t({\rm d} u)
	= e^{-t} t^\gamma + e^{-t} t^\gamma \sum_{j=1}^\infty t^j \int_0^1 u^\gamma P_j(u){\: \rm d} u.
\end{align}
Hence, by Corollary \ref{cor:m-1}, we immediately get the following statement.
\begin{corollary}
	\label{cor:m-2}
	For all $t>0$ and $\gamma\in \RR$,
	\[
		\int_0^\infty u^\gamma  \mu_t({\rm d} u)=
		e^{-t} t^\gamma \frac{\Gamma(\gamma+1)}{\Gamma_m(\gamma+1)}
		\sum_{j=0}^\infty \frac{t^{j}}{\Gamma_m(j+1)} \frac{\Gamma_m(j+\gamma+1)}{\Gamma(j+\gamma+1)}.
	\]
	If $\gamma \in -\NN$, the value of the right-hand side is understood in the limiting sense, see \eqref{eq:G/G_m}.
\end{corollary}

\begin{corollary}
	\label{cor:m-3}
	For all $t>0$ and $k\in \NN$,
	\begin{align*}
		\int_0^\infty u^k  \mu_t({\rm d} u)
		&=
		e^{-t} \frac{k!}{(m;m)_k} \int_{mt}^t \int_{m u_{k-1}}^{u_{k-1}} \ldots \int_{m u_1}^{u_1} e^{u_0} {\: \rm d}
		u_0 \ldots {\: \rm d} u_{k-1}\\
		&= k! \sum_{j=0}^k \bigg\{\prod_{\stackrel{i=0}{i\neq j}}^k \frac1{m^j-m^i} \bigg\} e^{-(1-m^j) t}.
	\end{align*}
\end{corollary}
\begin{proof}
	For $k\in\NN$, $\gamma \in \ZZ  \setminus\{-1,\ldots,-k\}$, and $t>0$,
	\begin{align}
		\label{eq:multi-integral}
		\int_{m t}^t \int_{m u_{k-1}}^{u_{k-1}} \ldots \int_{m u_1}^{u_1} u_0^{\gamma} {\: \rm d} u_0 \ldots {\: \rm d} u_{k-1}
		=
		t^{\gamma+k} \prod_{i=1}^k \frac{1-m^{\gamma+i}}{\gamma+i}.
	\end{align}
	Using Corollary~\ref{cor:m-2} and \eqref{eq:multi-integral} we get
	\begin{align*}
		\int_0^\infty u^k  \mu_t({\rm d} u)
		&= e^{-t} \frac{k!}{[k]_m!} \sum_{j = 0}^\infty \frac{t^{j+k}}{(j+k)!} \frac{[j+k]_m!}{[j]_m!}\\
		&= e^{-t} \frac{k!}{(m;m)_k} \sum_{j = 0}^\infty \frac{t^{j+k}}{j!}
		\bigg\{ \prod_{i=1}^k \frac{1-m^{j+i}}{j+i}\bigg\}\\
		&= e^{-t} \frac{k!}{(m;m)_k} \sum_{j = 0}^\infty \frac1{j!}
		\int_{m t}^t \int_{m u_{k-1}}^{u_{k-1}} \ldots \int_{m u_1}^{u_1}  u_0^j {\: \rm d} u_0 \ldots {\: \rm d} u_{k-1}.
	\end{align*}
	Now it suffices to show that
	\begin{align*}
		\int_{mt}^t \int_{m u_{k-1}}^{u_{k-1}} \ldots \int_{m u_1}^{u_1} e^{u_0} {\: \rm d} u_0 \ldots {\: \rm d} u_{k-1}
		=(m;m)_k \sum_{j=0}^k \bigg\{\prod_{\stackrel{i=0}{i \neq j}} ^k \frac1{m^j-m^i} \bigg\} e^{m^j t}
	\end{align*}
	which one can prove by a straightforward but tedious induction with respect to $k \in \NN$.
\end{proof}

Next, we compute the limits of moments.
\begin{proposition}
	\label{prop:m-1b}
	For all $\kappa \in (0, 1)$ and $\gamma\in \RR$,
	\begin{align}
		\label{lim:moments}
		\lim_{t\to +\infty}
		(1-m)^{\gamma}
		\frac{\Gamma_m(\gamma+1)}{\Gamma(\gamma+1)}
		\int_0^{\infty} u^\gamma \mu_t({\rm d} u) =1,
	\end{align}
	uniformly with respect to $m \in (0, \kappa]$ where for $\gamma \in -\NN$, the ratio is understood in the limiting
	sense, see \eqref{eq:G/G_m}. Moreover, for all $t_0 > 0$ and $\gamma \in \RR$,
	\begin{align}
		\label{ineq:sup-finite}
		\sup_{t\geq t_0} \int_0^{\infty} u^\gamma \mu_t({\rm d} u) < \infty.
	\end{align}
\end{proposition}
\begin{proof}
	Let $\gamma \in \RR$. If $\gamma>0$ we have
	\[
		1 \geq \frac{(m^{j+1};m)_{\infty}}{(m^{j+\gamma+1};m)_{\infty}}
		\geq (m^{j+1};m)_{\lceil \gamma \rceil} \geq (\kappa^{j+1};\kappa)_{\lceil \gamma \rceil}.
	\]
	Similarly, for $\gamma<0$, and $j \geq  \lfloor -\gamma \rfloor$ we get
	\[
		1 \leq
		\frac{(m^{j+1};m)_{\infty}}{(m^{j+\gamma+1};m)_{\infty}} \leq \frac1{(m^{j+\gamma+1};m)_{\lceil -\gamma \rceil}}
		\leq
		\frac1{(\kappa^{j+\gamma+1};\kappa)_{\lceil -\gamma \rceil}}.
	\]
	Therefore for a fixed $\epsilon \in (0, 1)$, there is $N \geq \lfloor -\gamma \rfloor$ which depends only on $\kappa$ and
	$\gamma$, such that for all $j \geq N$,
	\[
		\bigg| \frac{(m^{j+1};m)_{\infty}}{(m^{j+\gamma+1};m)_{\infty}} - 1 \bigg| \leq \epsilon.
	\]
	Using \eqref{eq:moments-mu_t-P_j}, we write
	\begin{align*}
		\int_0^{\infty} u^\gamma \mu_t({\rm d} u)
		=
		e^{-t} t^\gamma
		+
		e^{-t} t^{\gamma}
		\sum_{j = 1}^{N-1} t^j \int_0^1 u^\gamma P_j(u) {\: \rm d} u + I(t)
	\end{align*}
	where
	\[
		I(t) =
		e^{-t} t^{\gamma}
		\sum_{j = N+1}^{\infty} t^j \int_0^1 u^\gamma P_j(u) {\: \rm d} u.
	\]
	Therefore, by Corollary \ref{cor:m-1}
	\[
		\lim_{t \to +\infty} \frac{\Gamma_m(\gamma+1) (1-m)^{\gamma}}{\Gamma(\gamma+1)}
		\int_0^{\infty} u^\gamma \mu_t({\rm d} u)
		=
		\lim_{t \to +\infty}
		\frac{\Gamma_m(\gamma+1) (1-m)^{\gamma}}{\Gamma(\gamma+1)}
		I(t),
	\]
	uniformly with respect to $m \in (0, \kappa]$. Next, by Corollary \ref{cor:m-2} we have
	\begin{equation}
		\label{eq:88}
		\frac{\Gamma_m(\gamma+1) (1-m)^{\gamma}}{\Gamma(\gamma+1)} I(t)
		=
		e^{-t} t^\gamma\sum_{j=N+1}^\infty \frac{t^j}{\Gamma(j+\gamma+1)}
		\frac{(m^{j+1};m)_{\infty}}{(m^{j+\gamma+1};m)_{\infty}}.
	\end{equation}
	Let us recall the Mittag-Leffler function, that is
	\[
		E_{\alpha, \beta}(t) = \sum_{n = n_0}^\infty \frac{t^n}{\Gamma(\alpha n + \beta)},  \qquad t \in \RR
	\]
	where $n_0 \in \NN_0$ is any nonnegative integer such that $\alpha n_0 + \beta > 0$. Since
	\[
		E_{\alpha, \beta}(t) = \sum_{n = 0}^\infty \frac{t^{n+n_0}}{\Gamma(\alpha n + \beta + n_0\alpha)}
		=t^{n_0} E_{\alpha, \beta + n_0 \alpha}(t),
	\]
	by \cite[Theorem 4.3]{MR4179587}, for $\alpha \in (0, 2)$ we get
	\begin{equation}
		\label{eq:89}
		\lim_{t \to +\infty} t^{\beta-1} e^{-t} E_{\alpha, \beta}(t^\alpha)
		=
		\lim_{t \to +\infty} t^{\beta+ n_0 \alpha - 1} e^{-t} E_{\alpha, \beta+n_0\alpha}(t^{\alpha})
		=\frac1\alpha.
	\end{equation}
	Hence, by \eqref{eq:88},
	\begin{align*}
		\bigg| \frac{\Gamma_m(\gamma+1) (1-m)^{\gamma}}{\Gamma(\gamma+1)} I(t) -
		e^{-t} t^\gamma \mathit{E}_{1, \gamma+1}(t) \bigg|
		&\leq
		e^{-t}t^\gamma \sum_{j=N+1}^\infty \frac{t^j}{\Gamma(j+\gamma+1)}
		\bigg| \frac{(m^{j+1};m)_{\infty}}{(m^{j+\gamma+1};m)_{\infty}} -1\bigg|\\
		&\leq
		\epsilon e^{-t}t^\gamma \mathit{E}_{1,\gamma+1}(t),
	\end{align*}
	which by \eqref{eq:89} leads to \eqref{lim:moments}.
\end{proof}

\subsection{Weak convergence of $\mu_t$}
\label{sec:mu_t}
In this section we show that family of measures $(\mu_t : t > 0)$ converges weakly.
\begin{theorem}
	\label{thm:weak_conv}
	The family of probability measures $(\mu_t : t>0)$ on $[0,\infty)$ converges weakly as $t\to+\infty$ to a probability
	measure $\mu$ which is uniquely characterized by its moments:
	\[
		\int_0^{\infty} u^k \mu({\rm d}u)=\frac{k!}{(m;m)_k},\qquad k\in \NN_0.
	\]
	The measure $\mu$ has finite moments of all orders $\gamma\in \RR$, and
	\begin{equation}
		\label{eq:55}
		\lim_{t \to +\infty}
		\int_0^{\infty} u^\gamma \mu_t({\rm d} u)
		=
		\int_0^{\infty} u^\gamma \mu({\rm d} u)= \frac{\Gamma(\gamma+1)}{\Gamma_m(\gamma+1)} (1-m)^{-\gamma}.
	\end{equation}
	The value of the right-hand side for $\gamma \in -\NN$ is understood in the limiting sense, see \eqref{eq:G/G_m}.
\end{theorem}
\begin{proof}
	By Proposition \ref{prop:m-1b}, for each $k \in \NN$,
	\begin{equation}
		\label{eq:85}
		M_k = \lim_{t \to \infty} \int_0^\infty u^k \mu_t({\rm d} u) = \frac{k!}{(m;m)_k}.
	\end{equation}
	By Stirling's formula there is $C > 0$ such that
	\[
		\bigg(\frac{k!}{(m; m)_k} \bigg)^{\frac{1}{2k}} \leq C \sqrt{k},
	\]
	thus the Carleman's condition is satisfied, that is
	\begin{equation}	
		\label{eq:84}
		\sum_{k = 1}^\infty M_k^{-\frac{1}{2k}} = \infty.
	\end{equation}
	Consequently, the Stieltjes moment problem is determinate, i.e. the limit measure is unique if it exists.
	
	Next, by Corollary \ref{cor:m-2}, each measure $\mu_t$ is a probability measure on $[0, \infty)$. By Chebyshev's inequality,
	for all $\epsilon > 0$ and $t > 0$,
	\[
		1- \mu_t\big(\big\{ \abs{u} < \epsilon^{-1} \big\}\big)
		=
		\mu_t\big(\big\{ \abs{u} \geq \epsilon^{-1} \big\}\big)
		\leq
		\epsilon \int_0^\infty \abs{u} \: \mu_t({\rm d} u)
	\]
	which is uniformly bounded thanks to Proposition \ref{prop:m-1b}. Hence, the family $(\mu_t : t > 0)$ is tight.
	Since the moment problem is determinate, tightness implies that there is a measure $\mu$ such that
	$\mu_t$ weakly converge to $\mu$ as $t$ tends to infinity, see e.g. \cite[Theorem 25.10]{MR1324786}. Recall that a
	sequence of random variables which converges in distribution and has uniformly bounded $(p+\delta)$-moments, it has also
	convergent $p$-moments, see e.g. \cite[Theorem 25.12]{MR1324786}. Hence, all non-negative moments of $(\mu_t : t > 1)$
	converge to the moments of $\mu$ as $t$ tends to infinity. Lastly, notice that by the weak convergence, for each
	$\epsilon > 0$,
	\begin{align*}
		\mu(\{0\}) \leq \mu((-\infty,\epsilon))
		&\leq \liminf_{t\to +\infty} \mu_t((-\infty,\epsilon)) \\
		&\leq \sup_{t\geq 1} \int _0^\epsilon (u/\epsilon)^{-1}\mu_t({\rm d} u) \\
		&\leq \epsilon \sup_{t \geq 1} \int_0^\infty u^{-1} \mu_t({\rm d} u),
	\end{align*}
	hence by Proposition \ref{prop:m-1b} we obtain $\mu(\{0\})=0$. Consequently, we can use \cite[Theorem 25.7]{MR1324786}
	with $h(u)=|u|^{-1}$ to conclude that $\mu_t h^{-1}$ converges weakly to $\mu h^{-1}$. Hence, all positive real
	moments of $(\mu_t  h^{-1} : t \geq 1)$ converge to those of $\mu h^{-1}$ as $t \to +\infty$ which corresponds to negative
	real moments of $(\mu_t :  t\geq 1)$ and $\mu$, respectively. The exact values of the moments of $\mu$ follows by
	Proposition \ref{prop:m-1b}.
\end{proof}

We record the following corollary for later use.
\begin{corollary}
	For $t>0$ and $m\in (0,1)$,
	\begin{align}
		\label{eq:1-moment}
		\frac{(1-m)e^t}{e^t -e^{mt}} \int_0^\infty u\: \mu_t({\rm d} u) = 1,
	\end{align}
	and
	\begin{align}
		\label{ineq:2-moment}
		\frac{(1-m)^2 e^t}{e^t-e^{mt}} \int_0^\infty u^2 \: \mu_t({\rm d} u) \leq \frac2{1+m}.
	\end{align}
\end{corollary}
\begin{proof}
	The equality \eqref{eq:1-moment} directly follows from Corollary~\ref{cor:m-3} with $k=1$. To prove \eqref{ineq:2-moment},
	we observe that
	\[
		\int_{mt}^t \int_{m u_1}^{u_1} e^{u_0}  {\: \rm d} u_0 \: {\rm d} u_1 \leq e^t-e^{mt},
	\]
	thus the inequality is a consequence of Corollary~\ref{cor:m-3} with $k=2$.
\end{proof}

\begin{lemma}
	\label{lem:mu_t_erg}
	The family of probability measures $(\mu_t : t > 0)$ converge to $\mu$ in total variation distance, i.e.,
	\[
		\lim_{t \to +\infty}
		\|\mu_t-\mu\|_{TV}=
		\lim_{t \to +\infty} \sup_{B \in \mathcal{B}(\RR)} \left| \mu_t(B)-\mu(B) \right|
	\]
	where $\calB(\RR)$ denotes $\sigma$-field of Borel sets in $\RR$.
\end{lemma}
\begin{proof}
	Let us observe that by Corollary \ref{cor:m-3}, for each $t > 0$ and $k \in \NN_0$ we have
	\[
		M_k(t) = \int_0^\infty u^k \mu_t({\rm d} u) \leq M_k
	\]
	where $M_k$ is defined in \eqref{eq:85}. Then by \eqref{eq:84}, we obtain
	\[
		\sum_{k = 0}^\infty M_k(t)^{-\frac{1}{2k}} = \infty.
	\]
	Consequently, there is the unique measure on $(0, \infty)$ with moments $(M_k(t) : k \in \NN_0)$.

	Let $Y_t \equiv t$ for $t > 0$, and let $\mathbf{X}^{\rm TCP}$ be a process obtain from $\mathbf{Y}$ by partial resetting
	with factor $c$. The probability distribution of $X_t$ equals $\mu_t$ since they do have the same moments, see
	\cite[Theorem 3]{MR2426601}. Therefore proving the lemma is equivalent to proving ergodicity of the process
	$\mathbf{X}^{\rm TCP}$. The latter is a consequence of \cite[Theorem 1(3)]{MR1157423} together with
	\cite[Theorem 2.1 and Remark B]{MR1956829}. More precisely, the process  $\mathbf{X}^{\rm TCP}$ admits an embedded
	recursive chain $\mathbf{Z}^{TCP} = (Z_n^{\rm TCP} : n \in \NN_0)$ where
	\[
		Z^{\rm TCP}_n=X_{T_n}, \qquad  n \in \NN_0.
	\]
	Since $\mathbf{Z}^{TCP}$ satisfies
	\[
		Z^{\rm TCP}_{n+1}=c Z^{\rm TCP}_n + (T_{n+1}-T_n), \qquad n \in \NN_0,
	\]
	its driver $(T_{n+1}-T_n : n \in \mathbb{N}_0)$ consists of i.i.d. exponential random variables, see
	\cite[Definition 1]{MR1157423}. Hence, by \cite[Theorem 1(3)]{MR1157423}, it is enough to show sc-convergence of
	$\mathbf{Z}^{\rm TCP}$, see \cite[Def. 4, p. 23]{MR1157423} for the definition of sc-convergence.
	In view of \cite[Theorem 8]{MR1157423} it is enough to check conditions I--III given on page 18 of \cite{MR1157423}
	(notice that the assumption $\mathbb{E} \tau_V(x)<\infty$ is superfluous since it is a part of condition I). We notice
	that the conditions I--III hold true if $\mathbf{Z}^{\rm TCP}$ converge in total variation to its stationary law,
	see \cite[Theorem 2]{MR1157423}. Now, let us observe that the Markov chain $\mathbf{Z}^{\rm TCP}$ is an autoregressive
	process of order $1$, and it is positive Harris recurrent, see \cite[Theorem 2.1 and Remark B]{MR1956829}.
	Since $Z^{\rm TCP}_n$ has an absolutely continuous distribution, see e.g. \cite[Theorem 2.8]{MR1894253}, we immediately
	conclude that $\mathbf{Z}^{\rm TCP}$ convergence in total variation.
\end{proof}

\begin{remark}
	There is a purely analytic proof of Lemma \ref{lem:mu_t_erg}, however it would require preparatory steps which go
	beyond the scope of this article. The details will appear elsewhere. For the sake of completeness, here we provided a
	probabilistic arguments.
\end{remark}

\begin{remark}
	By Theorem~\ref{thm:weak_conv}, the measure $\mu$ is uniquely characterized by its moments. Since it has the same
	moments as the random variable $Z$ in \cite[(5.15)]{Kemperman}, the density of $\mu$ can be read of from
	\cite[(5.9)]{Kemperman}, that is
	\begin{equation}
		\label{eq:81}
		\mu(u)=
		\frac{1}{(m,m)_\infty}\sum_{k=0}^\infty (-1)^k \frac{m^{\frac{1}{2}k(k-1)}}{(m,m)_k}e^{-m^{-k}u},
		\quad u  \geq 0.
	\end{equation}
\end{remark}

The following lemma will be important in Section \ref{sec:cyl}.
\begin{lemma}
	\label{lem:inf_mu_t}
	For all $t_0 > 0$ and $\delta_2 > \delta_1 > 0$ such that $\delta_1 < t_0$ we have
	\[
		\inf_{t \in [t_0,\infty)} \mu_t\big([\delta_1,\delta_2]\big)  > 0.
	\]
\end{lemma}
\begin{proof}
	Since the function $\mu(u)$ is a density of the probability measure $\mu({\rm d} u)$ it is non-negative. By \eqref{eq:81},
	$\mu$ has holomorphic extension to $\big\{z \in \CC : \Re z > 0\big\}$. Therefore, it has finitely many zeros in
	$(\delta_1,\delta_2)$. In particular, $\mu([\delta_1,\delta_2]) > 0$. Hence, by Theorem~\ref{thm:weak_conv} there is $T > 0$
	such that for all $t > T$, $\mu_t([\delta_1,\delta_2]) > 0$. It remains to deal with $t \in [t_0, T]$. Let $j \geq 2$
	be such that $m^j < \delta_2 T^{-1}$. Using \eqref{def:mu_t}, for all $t \in [t_0, T]$ we get
	\begin{align}
		\nonumber
		\mu_t\big([\delta_1,\delta_2]\big)
		&\geq
		e^{-t}  t^j \int_{(\delta_1,\delta_2)} P_j(u/t) \frac{{\rm d} u}{t} \\
		\label{eq:82}
		&= e^{-t} t^j \int_{(\delta_1/t,\delta_2/t)} P_j(u) \: {\rm d} u.
	\end{align}
	Let us observe that \eqref{eq:82} defines a continuous function on $[t_0, T]$. Moreover, by Proposition \ref{prop:1}
	$P_j$ is positive on $[m_j, 1]$. Thus \eqref{eq:82} has positive infimum on $[t_0, T]$. This completes the proof.
\end{proof}

In the remaining part of this section we prove auxiliary lemmas which are helpful in studying ergodicity of the processes
with resetting.
\begin{lemma}
	\label{lem:unif_conv-1}
	Fix $C, \gamma > 0$ and let $\mathcal{G}$ be a family of functions $g : (0,\infty) \to \RR$, such that
	\[
		0 \leq g(u) \leq C\big(u+ u^{-1}\big)^\gamma, \quad\text{for all } u>0.
	\]
	Then
	\[
		\lim_{t\to+\infty}
		\sup_{g \in \mathcal{G}}
		\bigg| \int_0^{\infty}g(u) \: \mu_t({\rm d} u) - \int_0^{\infty}g(u)\: \mu({\rm d} u)\bigg|=0.
\]
\end{lemma}
\begin{proof}
	Let $\epsilon \in(0,1)$. By \eqref{ineq:sup-finite}, for sufficiently large $M > 1$, we have
	\[
		\int_{(0,1/M] \cup [M,\infty)}
		g(u) \: \mu_t({\rm d} u)
		\leq
		\frac{C}{ M+ M^{-1}}
		\sup_{t \geq 1} \int_0^{\infty} \big(u+u^{-1}\big)^{\gamma+1} \mu_t({\rm d} u) \leq \epsilon/3.
	\]
	A similar inequality holds true for $\mu$ in place of $\mu_t$. Notice that $M$ may be chosen so that the inequality holds
	true uniformly for all $g \in \mathcal{G}$ and $t \geq 1$. Next we write
	\begin{align*}
		&\bigg| \int_{(1/M, M)} g(u) \: \mu_t({\rm d} u) - \int_{(1/M,M)} g(u) \: \mu({\rm d} u)  \bigg| \\
		&\qquad
		\leq \int_{(1/M,M)} g(u) |\mu_t-\mu| ({\rm d} u) \\
		&\qquad= C \big(M+ M^{-1}\big)^{\gamma} 2 \|\mu_t-\mu\|_{TV}
	\end{align*}
	where $|\mu_t-\mu|({\rm d}u)$ denotes the total variation measure. Finally, by Lemma~\ref{lem:mu_t_erg}, there is $T \geq 1$
	such that for all $t\geq T$ we have
	\[
		C \big(M+ M^{-1}\big)^{\gamma} 2 \|\mu_t-\mu\|_{TV} \leq \epsilon/3
	\]
	which completes the proof.
\end{proof}

Given $c \in (0, 1)$, for $\delta > 0$ and $t > 0$, we set
\begin{equation}
	\label{eq:52}
	\mathscr{X}(t; \delta) = \bigg\{x \in \RR^d : -\log_c \norm{x} \leq \frac{t}{1+\delta} \bigg\}.
\end{equation}

\begin{lemma}
	\label{lem:unif_conv-2}
	Let $c \in (0, 1)$, $\alpha \in (0, 2]$ and $m = c^\alpha$. Fix $n \in \NN$ and $C, \gamma > 0$.
	Let $\mathcal{F}$ denote the family of functions $f : (0,\infty)\times \RR^n \to \RR$, such that for all $x \in \RR^n$
	and $u>0$,
	\begin{align}
		\label{eq:78}
		0 \leq f(u,x) \leq C \big(u+u^{-1} \big)^\gamma,\qquad
		| f(u, x)-f(u,0)|\leq C |x| \big(u+u^{-1} \big)^\gamma.
	\end{align}
	Then for each $\delta > 0$, we have
	\begin{align*}
		\lim_{t \to +\infty}
		\sup_{\stackrel{x \in \mathscr{X}(t; \delta)}{f \in \mathcal{F}}}
		\bigg| e^{-t}f(t,x)+e^{-t}\sum_{j=1}^\infty t^j \int_0^1 f(tu,c^jx) P_j(u) \: {\rm d}u
		-\int_0^{\infty} f(u,0) \: \mu({\rm d} u) \bigg|=0.
	\end{align*}
\end{lemma}
\begin{proof}
	In view of Lemma \ref{lem:unif_conv-1}, it is enough to show that
	\begin{equation}
		\lim_{t \to +\infty} \sup_{\stackrel{x \in \mathscr{X}(t; \delta)}{f \in \mathcal{F}}}
		\bigg|
		e^{-t} f(t, x) + e^{-t} \sum_{j = 1}^\infty t^j \int_0^1 f(tu, c^j x) P_j(u) {\: \rm d} u
		-
		\int_0^\infty f(u, 0) \mu_t({\rm d} u)
		\bigg|=0.
	\end{equation}
	Given $\epsilon > 0$ we set
	\[
		N = \max\big\{\gamma + 1, -\log_c (\norm{x}/\epsilon ) \big\}.
	\]
	By \eqref{def:mu_t}, we can write
	\begin{align*}
		&
		\bigg|
		e^{-t} f(t, x) + e^{-t} \sum_{j = 1}^\infty t^j \int_0^1 f(tu, c^j x) P_j(u) {\: \rm d} u
		- \int_0^\infty f(u, 0) {\: \rm d} \mu_t({\rm d} u) \bigg|\\
		&\qquad\qquad
		\leq
		e^{-t} \big( \abs{f(t, x)} + \abs{f(t, 0)}\big)
		+ e^{-t} \sum_{1 \leq j \leq N} t^j \int_0^1 \big(\abs{f(tu, c^j x)} + \abs{f(tu, 0)} \big) P_j(u) {\: \rm d} u \\
		&\qquad\qquad\phantom{\leq}
		+ e^{-t} \sum_{j > N} t^j  \int_0^1 \big| f(tu, c^j x) - f(tu, 0) \big| P_j(u) {\: \rm d} u.
	\end{align*}
	First, let us observe that if $x \in \mathscr{X}(t; \delta)$, then for $j > N$,
	$c^j \norm{x} \leq c^N \norm{x} \leq \epsilon$, and thus
	\begin{align*}
		e^{-t} \sum_{j > N} t^j  \int_0^1 \big| f(tu, c^j x) - f(tu, 0) \big| P_j(u) {\: \rm d} u
		&\leq
		C e^{-t}
		\sum_{j > N} t^j \int_0^1 c^j\norm{x} \Big( tu+\frac1{tu}\Big)^\gamma P_j(u) \: {\rm d}u \\
		&\leq
		\epsilon
		C \int_0^\infty \big(u + u^{-1}\big)^\gamma \mu_t({\rm d} u)
	\end{align*}
	which by Proposition \ref{prop:m-1b} is uniformly bounded by a constant multiply of $\epsilon$. Next, by
	Theorem \ref{thm:all-moments}, for each $j \in \NN$,
	\begin{align}
		\label{eq:56}
		\int_0^1 \abs{f(tu,c^jx)} P_j(u) \: {\rm d}u
		\leq
		C
		\frac{2^{\lceil \gamma \rceil - 1}}{(m;m)_\infty}
		&\bigg(
		t^{\lceil \gamma \rceil}
		\bigg\{\prod_{k=1}^j \frac{1-m^{k+\lceil \gamma \rceil}}{k+\lceil \gamma \rceil}\bigg\}
		+ t^{-\lceil \gamma \rceil}
		\bigg\{\prod_{k=1}^j \frac{1-m^{k-\lceil \gamma \rceil}}{k-\lceil \gamma \rceil}\bigg\}
		\bigg).
	\end{align}
	If $j \geq \lceil \gamma \rceil$, the second product is understood in the limiting sense, that is
	\[
		\prod_{k=1}^j \frac{1-m^{k-\lceil \gamma \rceil}}{k-\lceil \gamma \rceil}
		=
		(-\log m)
		\prod_{\stackrel{k=1}{k \neq \lceil \gamma \rceil}}^j \frac{1-m^{k-\lceil \gamma \rceil}}{k-\lceil \gamma \rceil}.
	\]
	Now, using \eqref{eq:56}, we get
	\[
		\sum_{1 \leq j \leq N} t^j \int_0^1 \abs{f(tu, c^j x)} P_j(u) {\: \rm d} u
		\leq
		C \frac{2^{\lceil \gamma \rceil-1}}{(m; m)_\infty} \big(I_1 + I_2\big)
	\]
	where
	\[
		I_1=\sum_{1 \leq j \leq N} t^{j+\lceil \gamma \rceil}
		\bigg\{\prod_{k=1}^j \frac{1-m^{k+\lceil \gamma \rceil}}{k+\lceil \gamma \rceil}\bigg\},
		\quad\text{and}\quad
		I_2 = \sum_{1 \leq j \leq N} t^{j-\lceil \gamma \rceil}
		\bigg\{\prod_{k=1}^j \frac{1-m^{k-\lceil \gamma \rceil}}{k-\lceil \gamma \rceil}\bigg\}.
	\]
	There is $c_1 > 0$ such that
	\[
		I_1 \leq
		c_1 \sum_{1 \leq j \leq N}
		\frac{t^{j+\lceil \gamma \rceil}}{(j+\lceil \gamma \rceil)!}
		\leq
		c_1 S(t, x)
	\]
	where
	\[
		S(t, x)= \sum_{0 \leq j \leq N+\lceil \gamma \rceil}\frac{t^j}{j!}.
	\]
	To bound $I_2$ for $t \geq 1$, we write
	\begin{align*}
		I_2
		&\leq \sum_{j = 1}^{\lceil \gamma \rceil-1} t^{j-\lceil \gamma \rceil}
		\bigg\{
		\prod_{k=1}^j \frac{1-m^{k-\lceil \gamma \rceil}}{k-\lceil \gamma \rceil}\bigg\}
		+
		(-\log m)
		\sum_{\lceil \gamma \rceil \leq j \leq N}
		t^{j-\lceil \gamma \rceil}
		\bigg\{
		\prod_{\stackrel{k=1}{k \neq \lceil \gamma \rceil}}^j
		\frac{1-m^{k-\lceil \gamma \rceil}}{k-\lceil \gamma \rceil}
		\bigg\} \\
		&\leq
		c_2 + c_3 \sum_{0 \leq j \leq N-\lceil \gamma \rceil}
		\frac{t^j}{j!}
	\end{align*}
	for certain $c_2, c_3 > 0$. Hence, there is $c_4 > 0$ such that
	\[
		\abs{f(t, x)} + \sum_{1 \leq j \leq N} t^j \int_0^1 \abs{f(tu, c^j x)} P_j(u) {\: \rm d} u
		\leq
		c_4 S(t, x).
	\]
	It remains to show that
	\begin{equation}
		\label{eq:50}
		\lim_{t \to +\infty} e^{-t} \sup_{x \in \mathscr{X}(t; \delta)} S(t, x) = 0.
	\end{equation}
	Let $T = \lceil N + \gamma \rceil$. If $-\log_c (\norm{x}/\epsilon) \leq \gamma+1$, then $N = \gamma+1$, and so for
	$t > T$,
	\[
		S(t, x) \leq e t^{2\gamma+1}.
	\]
	Otherwise, $N = - \log_c (\norm{x}/\epsilon)$. Since $x \in \mathscr{X}(t; \delta)$, if
	\[
		t \geq \frac{(1+\delta)^2}{\delta} (\log_c \epsilon + \gamma + 1),
	\]
	we have $t > (1+\delta)T$, and so
	\[
		S(t; x) = \sum_{j = 0}^T \frac{t^j}{j!} \leq (T+1) \frac{t^T}{T!}.
	\]
	The function $[1, \infty) \ni s \mapsto s^{-1} \log s$ is decreasing, hence for $t \geq (1+\delta) T$, we get
	we have
	\[
		T \log \frac{t}{T} \leq \frac{\log (1+\delta)}{1+\delta} t.
	\]
	Now, by the Stirling's formula
	\begin{align*}
		T \frac{t^T}{T!} e^{-t}
		&\leq
		c_5
		\sqrt{T}
		\exp\big\{ T \log t - T \log T + T - t\big\} \\
		&\leq
		c_5 \sqrt{\frac{t}{1+\delta}} \exp\Big\{t \frac{\log(1+\delta)-\delta}{1+\delta}\Big\},
	\end{align*}
	and because
	\[
		\frac{\log(1+\delta)-\delta}{1+\delta} < 0,
	\]
	we obtain \eqref{eq:50} and the theorem follows.
\end{proof}

\subsection{The function $\rho_{\mathbf{Y}}$}
\label{sec:2.4}
Let $(\mu_t : t > 0)$ be the family of probability measures defined by \eqref{def:mu_t} for certain $c \in (0, 1)$. In view
of Theorem \ref{thm:weak_conv}, $\mu_t$ weakly converge to the probability measure $\mu$ as $t$ tends to infinity. For a
given strictly $\alpha$-stable process $\mathbf{Y}$ in $\RR^d$, $\alpha \in (0, 2]$, with a transition density $p_0$, we
define a function
\begin{equation}
	\label{representationrho}
	\rho_{\mathbf{Y}} (y) = \int_0^{\infty} p_0(u;0,y) \: \mu({\rm d} u), \qquad y \in \RR^d.
\end{equation}
In this section we investigate the properties of $\rho_{\mathbf{Y}}$.
\begin{proposition}
	\label{prop:6}
	Suppose that $\mathbf{Y}$ is a strictly $\alpha$-stable process in $\RR^d$, $\alpha \in (0, 2]$, with a transition density.
	Then $\rho_{\mathbf{Y}} \in \calC_0^\infty(\RR^d)$. Moreover, for all $\gamma \in \RR$,
	\begin{equation}
		\label{eq:57}
		\int_{\RR^d} |y|^{\gamma} \rho_{\mathbf{Y}}(y) {\: \rm d}y
		= \frac{\Gamma(\gamma/\alpha+1)}{\Gamma_m(\gamma/\alpha+1)} (1-m)^{-\gamma/\alpha}\, \mathbb{E}|Y_1|^\gamma.
	\end{equation}
\end{proposition}
\begin{proof}
	The regularity of $\rho_{\mathbf{Y}}$ follows by Lemma~\ref{lem:A3}\eqref{en:3:2}, and the finiteness of all moments of
	$\mu$, see Theorem~\ref{thm:weak_conv}. Next, by the scaling Lemma \ref{lem:A3}\eqref{en:3:1} and Tonelli's theorem, we get
	\begin{align*}
		\int_{\RR^d} |y|^{\gamma} \rho_{\mathbf{Y}}(y) {\: \rm d}y
		&= \int_0^\infty u^{-d/\alpha} \int_{\RR^d}|y|^\gamma p_0(1;0, u^{-1/\alpha} y){\: \rm d} y \: \mu({\rm d}u)\\
		&=\int_0^\infty u^{\gamma/\alpha}\mu({\rm d}u) \int_{\RR^d}|z|^\gamma p_0(1;0, z){\: \rm d} z,
	\end{align*}
	which in view of \eqref{eq:55} completes the proof.
\end{proof}

\begin{remark}
	Let us recall that a function $f: \RR^d \to \RR$ is a homogeneous function of degree $\gamma \in \RR$, if
	$f(sx)=s^\gamma f(x)$ for all $x\in\RR^d$ and $s > 0$. Reasoning similar to the proof of \eqref{eq:57} leads to
	the following statement: Under the assumptions of Proposition \ref{prop:6}, let $f$ be a homogeneous function of degree $\gamma \in \RR$, such that
	\[
		\sup_{|\theta|=1} |f(\theta)|<\infty.
	\]
	If
	\[
		\gamma >
		\begin{cases}
			-\frac{\alpha}{d} & \text{if } \alpha \in (0, 2), \\
			-d & \text{if } \alpha = 2,
		\end{cases}
	\]
	then
	\[
		\int_{\RR^d} f(y) \rho_{\mathbf{Y}}(y){\: \rm d}y =
		\frac{\Gamma(\gamma/\alpha+1)}{\Gamma_m(\gamma/\alpha+1)}
		(1-m)^{-\gamma/\alpha} \: \mathbb{E} f(Y_1)
	\]
	and the integral converges absolutely.
\end{remark}

There are examples of strictly $\alpha$-stable processes in $\RR$ for which the explicit expression for all moments is
known, see e.g. \cite[Proposition 1.4]{Hardin}.
\begin{example}[{see \cite{Shanbhag} and \cite[Example 25.10]{MR1739520}}]
	If $d=1$ and $\mathbf{Y}$ is an $\alpha$-stable subordinator with $\alpha \in (0,1)$, then for $\gamma<\alpha$,
	\[
		\mathbb{E}|Y_1|^\gamma= \frac{\Gamma\left(1-\frac{\gamma}{\alpha}\right)}{\Gamma\left(1-\gamma\right)}.
	\]
\end{example}

\begin{example}[see \cite{Shanbhag}]
	If $d = 1$ and $\mathbf{Y}$ is a symmetric $\alpha$-stable process with $\alpha \in (0,2)$,
	then for $-1<\gamma<\alpha$,
	\[
		\mathbb{E}|Y_1|^\gamma
		=\frac{2^\gamma \Gamma (\tfrac{1+\gamma}{2})\Gamma (1-\tfrac{\gamma}{\alpha})}{\sqrt{\pi}\,\Gamma(1-\tfrac{\gamma}{2})}.
	\]
\end{example}

\begin{example}
	If $d = 1$ and $\mathbf{Y}$ is a Brownian motion, then for $\gamma>-1$ one can see that
	\[
		\mathbb{E}|Y_1|^\gamma=\frac{2^{\gamma}\Gamma \big(\tfrac{1+\gamma}{2}\big)}{\sqrt{\pi}}.
	\]
	In particular, by \eqref{eq:57} for $k \in \NN$,
	\[
		\int_{\RR} y^{2k-1} \rho_{\mathbf{Y}}(y){\: \rm d}y=0, \quad\text{ and }\quad
		\int_{\RR} y^{2k} \rho_{\mathbf{Y}}(y){\: \rm d}y= \frac{(2k)!}{(m;m)_k}.
	\]
\end{example}

\begin{lemma}
	\label{lem:densities}
	Suppose that $\mathbf{Y}$ is a strictly $\alpha$-stable process in $\RR^d$, $\alpha \in (0, 2]$, with a transition density.
	Then, for all $y \in \RR^d$,
	\begin{equation}
		\label{eq:66}
		\rho_{\mathbf{Y}}(y)
		=\frac{1}{(m; m)_\infty}
		\sum_{k=0}^\infty (-1)^k \frac{m^{\frac{1}{2}k(k-1)}}{(m; m)_k} \, U^{(m^{-k})}(y)
	\end{equation}
	where $U^{(\beta)}(y)=\int_0^\infty e^{-\beta u}  \,p_0(u;0,y) {\: \rm d}u$, $\beta > 0$. Moreover,
if $\mathbf{Y}$ is a Brownian motion in $\RR^d$, then
	\begin{equation}
		\label{eq:67}
		\lim_{\norm{y} \to +\infty}
		\frac{\rho_{\mathbf{Y}}(y)}{|y|^{-\frac{d-1}{2}}e^{-|y|}}
		= \frac12 \frac{1}{(m; m)_\infty}  (2\pi)^{-\frac{d-1}{2}}.
	\end{equation}
\end{lemma}
\begin{proof}
	The formula follows immediately from \eqref{representationrho} and \eqref{eq:81}. Indeed, we get
	\[
		\rho_{\mathbf{Y}}(y) = \int_0^\infty p_0(u; 0, y) \mu({\rm d} u)
		=
		\frac{1}{(m; m)_\infty} \sum_{k = 0}^\infty \frac{m^{\frac{1}{2}k(k-1)}}{(m; m)_k} \int_0^\infty
		e^{-m^k u} p_0(u; 0, y) \mu({\rm d} u).
	\]
	To show \eqref{eq:67}, let us recall that
	\begin{equation}
		\label{eq:U_Bessel}
		U^{(\beta)}(y)=(2\pi)^{-d/2}\left(\frac{|y|}{\sqrt{\beta}}\right)^{1-d/2} K_{d/2-1}(\sqrt{\beta}|y|)
	\end{equation}
	where $K_{d/2-1}$ is the modified Bessel function of the second type. Hence, by \eqref{eq:66} we get
	\[
		\frac{\rho_{\mathbf{Y}}(y)}{|y|^{-\frac{d-1}{2}}e^{-|y|}}
		=
		\frac{1}{(m; m)_\infty}
		\left( \frac{U^{(1)}(y)}{{|y|^{-\frac{d-1}{2}}e^{-|y|}}}\right)
		\sum_{k=0}^\infty (-1)^k \frac{m^{\frac{1}{2}k(k-1)}}{(m; m)_k} \frac{U^{(m^{-k})}(y)}{U^{(1)}(y)}.
	\]
	Because
	\[
		\lim_{|y|\to +\infty}
		\left(\frac{U^{(1)}(y)}{{|y|^{-\frac{d-1}{2}}e^{-|y|}}}\right)
		= \frac12 (2\pi)^{-\frac{d-1}{2}},
	\]
	and
	\[
		0\leq \frac{U^{(m^{-k})}(y)}{U^{(1)}(y)}\leq 1
	\]
	it is enough to treat sums of the form
	\[
		1+ \sum_{k=1}^N (-1)^k \frac{m^{\frac{1}{2}k(k-1)}}{(m; m)_k} \frac{U^{(m^{-k})}(y)}{U^{(1)}(y)}
	\]
	for a fixed $N$. For each $k \in \{1, \ldots, N\}$, by \eqref{eq:U_Bessel} and the asymptotic behavior of $K_{d/2-1}$,
	we obtain
	\[
		\lim_{|y|\to +\infty} \frac{U^{(m^{-k})}(y)}{U^{(1)}(y)}=0,
	\]
	which completes the proof.
\end{proof}

\begin{remark}
	Under the assumptions of Lemma \ref{lem:densities} for $d = 1$, we get the following formula
	\[
		\rho_{\mathbf{Y}} (y)= \frac12 \frac{1}{(m; m)_\infty}\sum_{k=0}^\infty (-1)^k
		\frac{m^{\frac{k^2}{2}}}{(m; m)_k}
		e^{-m^{-\frac{k}2 }|y|},
	\]
	and
	\[
		\lim_{|y|\to +\infty} \rho_{\mathbf{Y}}(y) \,e^{|y|} = \frac12 \frac{1}{(m; m)_\infty}.
	\]
\end{remark}

\section{Stable processes with resetting}
\label{sec:stationary}
In this section, we describe the resetting procedure for a large class of L{\'e}vy processes.
Let $\mathbf{Y}$ be a L{\'e}vy process in $\RR^d$ and let $c \in (0, 1)$.
We consider another L\'{e}vy process $\big((Y_t, N_t) : t \geq 0\big)$ with a natural filtration
$\big(\tilde{\calF}_t : t \geq 0\big)$ where $\mathbf{N}$ is a Poisson process independent of $\mathbf{Y}$ with intensity $1$.
Then the Poisson arrival moments $(T_k : k \geq 0)$ are Markov times relative to $(\tilde{\calF}_t : t\geq 0)$ where $T_0=0$.
Let $\mathbf{X}$ be a process with resetting defined in \eqref{eq:18}. Then for all $t > 0$ and every Borel set
$A \subset \RR^d$, we have
\begin{align*}
	&
	\PP^{(x,0)} (X_t \in A) \\
	&\qquad= \sum^\infty_{n=0} \PP^{(x, 0)}( X_t\in A, N_t=n) \\
	&\qquad= \sum^\infty_{n=0}
	\PP^{(x,0)} \bigg(\sum^n_{k=0} c^{n-k} \Big(Y_{t \wedge T_{k+1}}-Y_{T_k}\Big)+c^n x \in A, T_n \leq t <T_{n+1}\bigg)\\
	&\qquad= \PP^{(x,0)}(Y_t \in A) \PP^{(x, 0)}(N_t=0) +
	\sum^\infty_{n=1} \PP^{(x,0)}\bigg(\sum^n_{k=1}c^{n-k}
	\Big(Y_{t\wedge T_{k+1}}-Y_{T_k}\Big)+c^nY_{T_1}\in A,T_n\leq t <T_{n+1}\bigg).
\end{align*}
Hence, if for each $t > 0$, $Y_t$ has an absolutely continuous distribution, then for each $t > 0$ also $X_t$ has an absolutely
continuous distribution. Let us denote by $p_0$ and $p$ the transition densities of $\mathbf{Y}$ and $\mathbf{X}$, respectively.
Using the strong Markov property we get $\big( (\tilde{Y}_t, \tilde{N}_t) : t \geq 0\big)
= \big((Y_{t+T_1}-Y_{T_1}, N_{t+T_1}-1) : t \geq 0\big)$ is a L\'evy process with the same distribution as
$\big((Y_t,N_t) : t \geq 0\big)$ and independent of $Y_{T_1}$. Hence,
\begin{equation}
	\label{eq:pomoc}
	\begin{aligned}
	\PP^{(x,0)}(X_t\in A)
	&= \PP^{(x,0)}(Y_t\in A)e^{-t}\\
	&\phantom{=}
	+
	\sum^\infty_{n=1} \PP^{(x,0)}\bigg(\sum^{n-1}_{k=0}c^{n-(k+1)}
	\Big(\tilde{Y}_{(t-T_1) \wedge \tilde{T}_{k+1}}-\tilde{Y}_{\tilde{T}_k}\Big)
	+c^nY_{T_1}\in A,\tilde{T}_{n-1}\leq t-T_1 <\tilde{T}_{n}\bigg)\\
	&=\PP^{(x,0)}(Y_t\in A)e^{-t}+\PP^{(x,0)}\Big(T_1\leq t,\PP^{(cY_{T_1},0)}(\tilde{X}_{t-T_1}\in A)\Big),
\end{aligned}
\end{equation}
that is the transition densities satisfy
\begin{equation}
	\label{eq:6}
	p(t; x, y)
	=
	e^{-t} p_0(t; x, y)
	+
	\int_0^t \int_{\RR^d}
	e^{-s} p_0(s; x, z) p(t-s; cz, y) {\: \rm d} z {\: \rm d} s,
	\qquad x, y \in \RR^d, t > 0.
\end{equation}
Similarly, the natural filtration of $\textbf{X}$ is a subfiltration of $(\tilde{\calF}_t : t \geq 0)$, thus
$\PP^{(x,0)}(X_{t+u}\in A | \tilde{\calF_t})$ is a function of $X_t$, that is $\mathbf{X}$ is a Markov process.
\begin{proposition}
	\label{prop:5}
	Suppose that $\mathbf{Y}$ is a L{\'e}vy process in $\RR^d$ with a transition density $p_0$.
	Assume that $\mathbf{X}$ is obtained from $\mathbf{Y}$ by partial resetting with factor $c\in(0,1)$. Then the process
	$\mathbf{X}$ has the transition density $p$ satisfying
	\begin{equation}
		\label{eq:2}
		p(t; x, y) = e^{-t} \sum_{j = 0}^\infty p_j(t; x, y), \quad \text{for all } x,y \in \RR^d, t > 0
	\end{equation}
	where $(p_n : n \in \NN_0)$ is a sequence of functions that satisfy the recursion
	\begin{equation}
		\label{eq:15}
		p_{n+1}(t; x, y) = \int_0^t \int_{\RR^d} p_0(s; x, z) p_n(t-s; cz, y) {\: \rm d}z {\: \rm d} s,
		\quad\text{for all }x, y \in \RR^d, t >0.
	\end{equation}
\end{proposition}
\begin{proof}
	First, let us observe that for all $j \in \NN_0$, $x \in \RR^d$ and $t > 0$,
	\begin{equation}
		\label{eq:20}
		\int_{\RR^d} p_j(t; x, y) {\: \rm d} y = \frac{t^j}{j!}.
	\end{equation}
	To see this, we reason by induction over $j \in \NN_0$. For $j = 0$ the formula trivially holds true. Next,
	\begin{align*}
		\int_{\RR^d} p_{j+1}(t; x, y) {\: \rm d} y
		&=
		\int_{\RR^d} \int_0^t \int_{\RR^d} p_0(s; x, z) p_j(t-s; cz, y) {\: \rm d} z {\: \rm d} s {\: \rm d} y \\
		&=
		\frac{1}{j!}
		\int_0^t \int_{\RR^d} p_0(s; x, z) (t-s)^j {\: \rm d}z {\: \rm d} s \\
		&=
		\frac{1}{j!}
		\int_0^t (t-s)^j {\: \rm d} s = \frac{t^{j+1}}{(j+1)!},
	\end{align*}
	as claimed.

	Since all the terms of the series \eqref{eq:2} are positive, by Tonelli's theorem \eqref{eq:20} we get
	\begin{align*}
		\int_{\RR^d} e^{-t} \sum_{j = 0}^\infty p_j(t; x, y) {\: \rm d} y
		&=
		e^{-t} \sum_{j = 0}^\infty \int_{\RR^d} p_j(t; x, y) {\: \rm d} y \\
		&=
		e^{-t} \sum_{j = 0}^\infty \frac{t^j}{j!} = 1.
	\end{align*}
	Hence, \eqref{eq:2} defines actual solution of \eqref{eq:6}. Next, we observe that it is also the unique probabilistic
	solution of \eqref{eq:6}. Indeed, if $\tilde{p}(t; x , y)$ would be another solution, then
	\[
		e^{-t} p_0(t; x, y) \leq \tilde{p}(t; x, y), \quad \text{for all } x, y \in \RR^d, t > 0.
	\]
	Thus reasoning by induction one can prove that for all $N \in \NN$,
	\[
		e^{-t} \sum_{j = 0}^N p_j(t; x, y) \leq \tilde{p}(t; x, y), \quad \text{for all } x, y \in \RR^d, t > 0.
	\]
	Consequently,
	\[
		p(t; x, y) \leq \tilde{p}(t; x, y), \quad \text{for all } x, y \in \RR^d, t > 0.
	\]
	Since both $p(t; x, \cdot)$ and $\tilde{p}(t; x, \cdot)$ integrate to $1$, for each $\epsilon > 0$
	by Chebyshev's inequality we obtain
	\[
		\big|\big\{y \in \RR^d : p(t; x, y) \geq \tilde{p}(t; x, y) + \epsilon \big\}\big|
		\leq
		\epsilon^{-1} \int_{\RR^d} (p(t; x, y) - \tilde{p}(t; x, y)) {\: \rm d} y = 0,
	\]
	that is $p(t; x, y) = \tilde{p}(t; x, y)$. This completes the proof.
\end{proof}

Let us observe that certain properties of $p_0$ translate to $p_n$.
\begin{proposition}
	\label{prop:4}
	Suppose that $\mathbf{Y}$ is a L\'evy process on $\RR^d$ with a transition density $p_0$.
	\begin{enumerate}[label=\rm (\roman*), start=1, ref=\roman*]
		\item
		\label{en:2:1}
		Then for all $n \in \NN$,
		\[
			p_n(t; x, y) = p_n(t; 0, y - c^n x).
		\]
		\item
		\label{en:2:2}
		If there is $\Upsilon \in \GL(\RR, d)$ such that
		\[
			p_0(t; x, y) = p_0(t; \Upsilon x, \Upsilon y), \quad \text{for all } x, y \in \RR^d, t > 0,
		\]
		then for each $j \in \NN$,
		\[
			p_j(t; x, y) = p_j(t; \Upsilon x, \Upsilon y), \quad \text{for all } x, y \in \RR^d, t > 0.
		\]
	\end{enumerate}
\end{proposition}
\begin{proof}
	The proof of \eqref{en:2:1} is by induction with respect to $n \in \NN$. For $n = 0$ the formula trivially holds true.
	Using the inductive hypothesis and the change of variables we can write
	\begin{align*}
		p_{n+1}(t; x, y)
		&= \int_0^t \int_{\RR^d} p_0(s; x, z) p_n(t-s; cz, y) {\: \rm d}z {\: \rm d} s\\
		&= \int_0^t \int_{\RR^d} p_0(s; 0, z-x) p_n(t-s; cz, y) {\: \rm d} z {\: \rm d} s \\
		&= \int_0^t \int_{\RR^d} p_0(s; 0, z) p_n(t-s; c(z +x), y) {\: \rm d} z {\: \rm d} s \\
		&= \int_0^t \int_{\RR^d} p_0(s; 0, z) p_n(t-s; cz , y - c^{n+1} x) {\: \rm d} z {\: \rm d} s \\
		&= p_{n+1}(t; 0, y - c^{n+1} x),
	\end{align*}
	completing the proof of the first part of the proposition. To show \eqref{en:2:2} we again use induction to get
	\begin{align*}
		p_{j+1}(t; \Upsilon x, \Upsilon y)
		&= \int_0^t \int_{\RR^d} p_0(s; \Upsilon x, z) p_j(t-s; cz, \Upsilon y) {\: \rm d}z {\: \rm d} s \\
		&= \int_0^t \int_{\RR^d} p_0(s; x, \Upsilon^{-1} z) p_j(t-s; c\Upsilon^{-1} z, y) (\det \Upsilon)^{j}
		{\: \rm d}z {\: \rm d} s \\
		&= (\det \Upsilon)^{j+1} p_{j+1}(t; x, y).
	\end{align*}
	Note that $p(t; x, y)$ integrates to $1$, hence $\det \Upsilon =1$ and the proposition follows.
\end{proof}

\begin{remark}
	Due to Proposition \ref{prop:4}\eqref{en:2:2} we can easily see that $p_n$ is isotropic if $p_0$ is isotropic.
	Moreover, if $p_0$ is the transition density of isotropic unimodal L\'evy process, then
	$\RR^d \ni y \mapsto p_0(t; 0, y)$ is radial and non-increasing, thus $\RR^d \ni y \mapsto p_j(t; 0, y)$ is radial and
	non-increasing too. Indeed, by Proposition \ref{prop:4}\eqref{en:2:1},
	\begin{align*}
		p_{j+1}(t; 0, y)
		&=
		\int_0^t \int_{\RR^d} p_0(s; 0, z) p_j(t-s; cz, y) {\: \rm d} z {\: \rm d} s \\
		&
		=
		\int_0^t \int_{\RR^d} p_0(s; 0, z) p_j\big(t-s; 0, c^{j+1} (c^{-j-1}y - z) \big)  {\: \rm d} z {\: \rm d} s.
	\end{align*}
	Consequently, the process $\mathbf{X}$ is isotropic and unimodal if $\mathbf{X}$ is such.
\end{remark}

\subsection{Representation of the transition density}
\label{sec:repr}
We are now ready to prove the key representation of the density. From this point we assume that the original process is
$\alpha$-stable for $\alpha \in (0, 2]$.
\begin{theorem}
	\label{thm:rep-gen-st}
	Suppose that $\mathbf{Y}$ is a strictly $\alpha$-stable process in $\RR^d$, $\alpha \in (0, 2]$, with a transition density
	$p_0$. Suppose that $\mathbf{X}$ is obtained from $\mathbf{Y}$ by partial resetting with factor $c \in (0, 1)$. Then
	the transition density of $\mathbf{X}$ satisfies
	\begin{align}
		\label{eq:7}
		p_n(t; 0, y)
		= t^n \int_0^1 p_0(t u; 0, y) P_n(u) {\: \rm d} u, \quad \text{for all } n \in \NN, y \in \RR^d, t > 0.
	\end{align}
	Moreover, for all $x, y \in \RR^d$ and $t > 0$,
	\begin{align}
		\label{eq:rep-p-x}
		p(t; x, y)=e^{-t}p_0(t; 0, y-x)+e^{-t}\sum_{j=1}^\infty t^j \int_0^1 p_0(tu;0,y-c^jx)  P_j(u) {\: \rm d} u.
	\end{align}
\end{theorem}
\begin{proof}
	First, we link $W_n$ with $p_n$ which are defined by \eqref{eq:15}. Namely, we claim that
	\begin{equation}
		\label{eq:32}
		p_n(t; 0, y) = \int_{m^n t}^t p_0(u; 0, y) W_n(t, u) {\: \rm d} u,
		\quad \text{for all } y \in \RR^d,  t > 0.
	\end{equation}
	The proof is by induction with respect to $n \in \NN$. For $n = 1$, we have
	\begin{align*}
		p_1(t; 0, y)
		&= \int_0^t \int_{\RR^d} p_0(s; 0, z) p_0(t - s; m^{\frac{1}{\alpha}} z, y) {\: \rm d} z {\: \rm d} s \\
		&= \int_0^t \int_{\RR^d}
		p_0(s; 0, z m^{-\frac{1}{\alpha}}) p_0(t - s; z, y) m^{-\frac{d}{\alpha}} {\: \rm d} z {\: \rm d} s.
	\end{align*}
	Since $p_0$ is the density of $\alpha$-stable process, it satisfies the scaling property
	(see \cite[Theorems 14.2 and 14.7]{MR1739520}), hence
	\begin{align*}
		p_1(t; 0, y)
		&=
		\int_0^t \int_{\RR^d}
		p_0(m s; 0, z) p_0(t - s; z, y) {\: \rm d} z {\: \rm d} s \\
		&=
		\int_0^t p_0(t + ms - s; 0, y) {\: \rm d} s \\
		&=
		\int^t_{m t} p_0(s; 0, y) \frac{1}{1-m} {\: \rm d} s.
	\end{align*}
	Next, we write
	\begin{align*}
		p_{n+1}(t; 0, y)
		&=
		\int_0^t \int_{\RR^d} p_0(s; 0, z) p_n(t-s; 0, y- m^{\frac{n+1}{\alpha} z}) {\: \rm d} z {\: \rm d} s \\
		&=
		\int_0^t \int_{\RR^d} p_0(s; 0, z) \int_{m^n (t-s)}^{t-s} p_0(u; 0, y - m^{\frac{n+1}{\alpha}} z)
		W_n(t-s, u) {\: \rm d} u {\: \rm d} z {\: \rm d} s \\
		&=
		\int_0^t \int_{m^n (t-s)}^{t-s}
		\int_{\RR^d} p_0(s; 0, z) p_0(u; 0, y-m^{\frac{n+1}{\alpha}} z) {\: \rm d} z W_n(t-s, u)
		{\: \rm d}  u {\: \rm d} s \\
		&=
		\int_0^t \int_{m^n (t-s)}^{t-s}
		\int_{\RR^d} m^{-d\frac{n+1}{\alpha}} p_0(s; 0, m^{-\frac{n+1}{\alpha}} z)
		p_0(u; 0, y - z) {\: \rm d} z W_n(t-s, u) {\: \rm d} u {\: \rm d} s \\
		&=
		\int_0^t
		\int_{m^n (t-s)}^{t-s} \int_{\RR^d} p_0(m^{n+1} s; 0, z)
		p_0(u; z, y) {\: \rm d} z W_n(t-s, u) {\: \rm d} u {\: \rm d} s
	\end{align*}
	where in the last equality we have used properties of $\alpha$-stable density. Hence,
	\begin{align*}
		p_{n+1}(t; 0, y)
		&=
		\int_0^t
		\int_{m^n (t-s)}^{t-s} p_0(u + m^{n+1} s; 0, y) W_n(t - s, u) {\: \rm d} u {\: \rm d} s \\
		&=
		\int_0^t \int^{t-s + m^{n+1}s}_{m^n (t-s)+m^{n+1}s} p_0(u; 0, y) W_n(t-s, u - m^{n+1}s) {\: \rm d} u {\: \rm d} s.
	\end{align*}
	Now, we observe that
	\begin{align*}
		&
		\Big\{ (s, u) \in \RR^2 :
		0 \leq s \leq t \text{ and } m^n(t-s) + m^{n+1} s \leq u \leq t-s+m^{n+1}s \Big\} \\
		&\qquad=
		\bigg\{(s, u) \in \RR^2 :
		m^{n+1} t \leq u \leq t, \text{ and } \frac{m^n t - u}{m^n - m^{n+1}} \vee 0 \leq s \leq \frac{t - u}{1- m^{n+1}}
		\bigg\}.
	\end{align*}
	Thus changing the order of integration we obtain
	\begin{align*}
		p_{n+1}(t; 0, y)
		&=
		\int_{m^{n+1} t}^t
		\int_{\frac{m^n t- u}{m^n - m^{n+1}} \vee 0}^{\frac{t-u}{1-m^{n+1}}}
		p_0(u; 0, y) W_n(t-s, u - m^{n+1}s) {\: \rm d} s {\: \rm d} u
	\end{align*}
	as claimed.
	
	Having proved \eqref{eq:23}, by Proposition \ref{prop:3} and \eqref{eq:21}, we can write
	\begin{align*}
		p_n(t; 0, y)
		&= \int_{m^n t}^t p_0(u; 0, y) W_n(t, u) {\: \rm d} u \\
		&= \int_{m^n t}^t p_0(u; 0, y) t^{n-1} P_n(u t^{-1}) {\: \rm d} u
		= t^n \int_{m^n}^1 p_0(t u; 0, y) P_n(u) {\: \rm d} u.
	\end{align*}
	Now, \eqref{eq:rep-p-x} follows by Proposition \ref{prop:5}, Proposition~\ref{prop:4}\eqref{en:2:1} and \eqref{eq:7}.
\end{proof}

\begin{corollary}
	\label{cor:rep-1}
	Under the assumptions of Theorem \ref{thm:rep-gen-st}, the transition density of the process $\mathbf{X}$ satisfies
	\begin{align}
		\label{eq:rep-p-0.1}
		p(t;0,y)= \int_0^\infty p_0(u;0,y) \: \mu_t({\rm d} u), \quad\text{for all } y \in \RR^d, t > 0
	\end{align}
	where $\mu_t$ is the measure defined by the formula \eqref{def:mu_t}.
\end{corollary}

\begin{corollary}
	\label{cor:rep-2}
	Under the assumptions of Theorem \ref{thm:rep-gen-st}, the transition density of the process $\mathbf{X}$ satisfies
	\begin{align}
		\label{eq:rep-p-0}
		p(t;0,y)= e^{-t}p_0(t;0,y)+e^{-t} \int_0^1 p_0(tu;0,y) \Phi(t,u) {\: \rm d} u, \quad\text{for all } y \in \RR^d, t > 0
	\end{align}
	where
	\begin{align}
		\label{def:Phi}
		\Phi(t, u)= \sum_{j = 1}^\infty t^j P_j(u),\qquad  u \in [0, 1],\, t > 0.
	\end{align}
\end{corollary}

\subsection{Ergodicity of $\mathbf{X}$}
In this section we show that the process obtained by partial resetting of $\alpha$-stable process, $\alpha \in (0, 2]$,
is ergodic, see Theorem \ref{thm:lim_p_t_infty}. In fact, we are going to prove that $p(t; x, y)$ converges uniformly in a
certain space-time region in $(t, y)$ as $t$ tends to infinity.
\begin{theorem}
	\label{thm:lim_p_t_infty}
	Suppose that $\mathbf{Y}$ is a strictly $\alpha$-stable process in $\RR^d$, $\alpha \in (0, 2]$, with a transition density.
	Assume that $\mathbf{X}$ is obtained from $\mathbf{Y}$ by partial resetting with factor $c \in (0, 1)$.
	Then for each $\delta > 0$,
	\begin{align}
		\label{eq:lim_p_t_infty-unif}
		\lim_{t\to+\infty} \sup_{y \in \RR^d}\sup_{x \in \mathscr{X}(\delta; t)}
		| p(t;x,y) - \rho_{\mathbf{Y}}(y)|=0
	\end{align}
	where $\mathscr{X}(\delta; t)$ is defined in \eqref{eq:52} and $\rho_{\mathbf{Y}}$ is defined in \eqref{representationrho}.
	In particular, for all $x, y \in \RR^d$,
	\begin{equation}
		\label{eq:64}
		\lim_{t\to+\infty} p(t;x,y)=\rho_{\mathbf{Y}}(y).
	\end{equation}
\end{theorem}
\begin{proof}
	The convergence \eqref{eq:lim_p_t_infty-unif} is a consequence of Lemma \ref{lem:unif_conv-2} applied to the family of
	functions
	\[
		\mathcal{F}=\big\{p_0(u; 0, y-x) : y\in\RR^d \big\}. \qedhere
	\]
\end{proof}

By Theorem \ref{thm:lim_p_t_infty} together with Scheff\'e's lemma, see e.g. \cite{MR2728440}, we easily get the following
corollary.
\begin{corollary}
	\label{cor:1}
	Under the assumptions of Theorem \ref{thm:lim_p_t_infty}, for all $x \in \RR^d$,
	\begin{equation}
		\label{eq:65}
		\lim_{t\to+\infty} \int_{\RR^d}|p(t;x,y)-\rho_{\mathbf{Y}}(y)| {\: \rm d}y=0.
	\end{equation}
	In particular, the measure $\rho_{\mathbf{Y}}(y)$ is an invariant measure of the process $\mathbf{X}$.
\end{corollary}

\subsection{The transition density of $\mathbf{X}$}
\label{sec:3.3}
We study the regularity of the transition density of the process obtained by a partial resetting from $\alpha$-stable process.
\begin{lemma}
	\label{lem:p_reg}
	Let $\mathbf{Y}$ be a strictly $\alpha$-stable process in $\RR^d$, $\alpha \in (0, 2]$, with a transition density
	$p_0$. Suppose that $\mathbf{X}$ is obtained from $\mathbf{Y}$ by partial resetting with factor $c \in (0, 1)$.
	Then $p$ the transition density of $\mathbf{X}$ belongs to $\calC^{\infty}\big((0,\infty)\times\RR^d\times\RR^d\big)$.
	Moreover, for all $\ell \in \NN_0$, $\mathbf{a}, \mathbf{b} \in \NN_0^d$, $x, y \in \RR^d$ and $t>0$, both functions
	\[
		z \mapsto \partial_t^\ell \partial_{x}^{\mathbf{a}}\partial_{y}^\mathbf{b} p(t; z,y),
		\quad
		\text{and}
		\quad
		z \mapsto \partial_t^\ell \partial_{x}^\mathbf{a} \partial_{y}^\mathbf{b} \,p(t;x,z)
	\]
	belongs to $\calC_0^\infty(\RR^d)$.
\end{lemma}
\begin{proof}
	First, using \eqref{eq:rep-p-x} and the induction with respect to the order of the differentiation one can prove that
	$\partial_t^\ell \partial_{x}^{\mathbf{a}}\partial_{y}^{\mathbf{b}} p(t;x,y)$ equals to
	$\partial_t^\ell \partial_{x}^{\mathbf{a}}\partial_{y}^{\mathbf{b}} (e^{-t} p_0(t;0,y-x))$ plus a linear combination
	of expressions of the form
	\begin{align}
		\label{expr:aux}
		e^{-t}\sum_{j=1}^\infty  c_j  t^{j-\ell_1} \int_0^1 u^{\ell_2} f(tu,y-c^jx) P_j(u) \: {\rm d}u
	\end{align}
	where $\ell_1 + \ell_2 = \ell$,
	\[
		c_j= \begin{cases}
			\frac{j!}{(j-\ell_1)!}(-c^j)^{|\mathbf{a}|} & \text{if } j \geq \ell_1, \\
			0 & \text{if } j<\ell_1,
		\end{cases}
	\]
	and
	\[
		f(u,y)= \big(\partial_u^{\ell_2} \,\partial_{y}^{\mathbf{a}+\mathbf{b}}\, p_0\big)(u;0,y).
	\]
	The inductive step will be verified once we justify the differentiation under both the series and the integral sign.
	By Lemma~\ref{lem:A3}\eqref{en:3:2} we have
	\begin{align}
		\label{bound:aux}
		|\partial_u f(u,y)| +|\nabla_y f(u,y)|+|f(u,y)| \leq C\big(u+u^{-1}\big)^\gamma,
	\end{align}
	for certain $C,\gamma > 0$. Recall that each $P_j$ has support bounded away from zero, see Proposition~\ref{prop:1}.
	Thus, due to \eqref{bound:aux}, any partial sum of the series in \eqref{expr:aux} may be differentiated as
	many times as one needs.
	Let us consider $j \geq \ell_1+\lceil \gamma \rceil+1$. By \eqref{bound:aux} each difference quotient may be bounded
	terms of the form
	\begin{align*}
		\sum_{j=\ell_1+\lceil \gamma \rceil+1}^\infty
		\frac{j!}{(j-k)!}  (2t)^{j-k}
		\int_0^1 u^{\ell_2-\lceil\gamma\rceil}\left(t^{\lceil \gamma \rceil}+
		t^{-\lceil\gamma\rceil}\right)  P_j(u) \: {\rm d}u
	\end{align*}
	where $k = \ell_1$ or $k = \ell_1+1$, which by Theorem~\ref{thm:all-moments} is finite. Hence, to conclude it is enough to
	invoke the dominated convergence theorem.

	Lastly, the convergence to zero as $\norm{x} + \norm{y}$ gets large follows by reasoning similar to the proof of
	Lemma~\ref{lem:A3}\eqref{en:3:2}, together with \eqref{expr:aux} and \eqref{bound:aux}.
\end{proof}

\begin{remark}
	The Markov property of $\mathbf{X}$ implies that
	\[
		P_t f(x) = \int_{\RR^d} p(t; x, y) f(y) {\: \rm d} y, \qquad x \in \RR^d
	\]
	forms a semigroup on nonnegative bounded functions. In view of the regularity of $p$, the Chapmann--Kolmogorov equation
	holds true, that is
	\begin{equation}
		\label{eq:90}
		p(t + s; x, y) = \int_{\RR^d} p(t; x, z) p(s; z, y) {\: \rm d} z,
		\qquad\text{ for all } s, t > 0 \text{ and } x, y \in \RR^d.
	\end{equation}
\end{remark}

\subsection{Non-equilibrium stationary state of $\mathbf{X}$}
\label{sec:3.4}
In this section our aim is to show that $\mathbf{X}$ has non-equilibrium stationary state. To do so, we prove that
the infinitesimal generator of $\mathbf{X}$ on $L^2(\RR^d, \rho_{\mathbf{Y}}(y) {\rm d} y)$ is \emph{not} self-adjoint,
see Theorem \ref{thm:NESS}. We also show that $\rho_{\mathbf{Y}}$ is harmonic with respect to $L^2$-adjoint operator
to the infinitesimal generator of $\mathbf{X}$, see Theorem \ref{thm:H+F-P}.

For $f \in \calC_0^2(\RR^d)$ we set
\[
	\begin{aligned}
	\mathscr{L} f(x)=
	\sum_{j,k=1}^d a_{jk} \partial_{x_j} \partial_{x_k} f(x) &+ \sprod{\nabla f(x)}{\gamma} \\
	&+\int_{\RR^d} \big(f(x+z)-f(x)-\ind{\{|z|<1\}} \sprod{\nabla f(x)}{z} \big) \: \nu({\rm d} z)
	\end{aligned}
\]
where $(A, \gamma, \nu)$ with $A=[a_{jk}]_{j,k=1}^d$ is the L{\'e}vy triplet of $\mathbf{Y}$, see Appendix \ref{appendix:A}.
Furthermore, we set
\[
	\begin{aligned}
	\mathscr{L}^*f(x)
	=
	\sum_{j,k=1}^d a_{jk} \partial_{x_j} \partial_{x_k} f (x) &+ \sprod{\nabla f(x)}{\gamma^*} \\
	&+\int_{\RR^d} \big(f(x+z)-f(x)-\ind{\{|z|<1\}} \sprod{\nabla f(x)}{z}\big) \: \nu^*({\rm d} z)
	\end{aligned}
\]
where $\gamma^*=-\gamma$ and $\nu^*(B)=\nu(-B)$, for a Borel set $B \subset \RR^d$. Accordingly, we define
\[
	\mathscr{A} f(x) = \mathscr{L} f(x) + f(cx) - f(x),
	\qquad\text{and}\qquad
	\mathscr{A}^* f(x) = \mathscr{L}^* f(x) + \frac{1}{c} f\Big(\frac{x}{c}\Big) - f(x).
\]
Given $r \in [1, \infty)$, we let
\[
	\mathcal{B}_r =
		\begin{cases}
			L^r(\RR^d) & \text{if } r \in [1, \infty), \\
			\calC_0(\RR^d) & \text{if } r = \infty.
		\end{cases}
\]
In view of Proposition \ref{prop:9}, $\mathbf{Y}$ generates a strongly continuous contractive semigroup on $\mathcal{B}_r$.
Let $\mathscr{L}_r$ be the infinitesimal generator of this semigroup, and let $D(\mathscr{L}_r)$ be the domain of $\mathscr{L}_r$.
\begin{theorem}
	\label{thm:1}
	Suppose that $\mathbf{Y}$ is a strictly $\alpha$-stable process in $\RR^d$, $\alpha \in (0, 2]$, with a transition density.
	Assume that $\mathbf{X}$ is obtained from $\mathbf{Y}$ by partial resetting with factor $c \in (0, 1)$. Let
	$r \in [1, \infty]$. Then the process $\mathbf{X}$ generates strongly continuous semigroup of bounded operators on
	$\mathcal{B}_r$. Its infinitesimal generator is
	\begin{equation}
		\label{eq:87}
		\mathscr{A}_r f(x) = \mathscr{L}_r f(x) + f(cx) - f(x), \qquad f \in D(\mathscr{A}_r).
	\end{equation}
	Moreover, $D(\mathscr{A}_r) = D(\mathscr{L}_r)$.
\end{theorem}
\begin{proof}
	Fix $r \in [1, \infty]$. For $f \in \mathcal{B}_r$ and $t > 0$ we set
	\[
		P_t f(x) = \int_{\RR^d} p(t; x, y) f(y) {\: \rm d} y, \qquad\text{ for } x \in \RR^d.
	\]
	By Lemma \ref{lem:p_reg} and H\"older's inequality, the integral is well-defined. The proof that $(P_t : t > 0)$ is
	a strongly continuous semigroup of bounded operators on $\mathcal{B}_r$, is essentially
	the same as the proof that its infinitesimal generator is given by \eqref{eq:87}, which is proved as follows:
	First, using \eqref{eq:rep-p-x}, we write
	\begin{align*}
		\frac{P_t f(x) - f(x)}{t}
		&= e^{-t} \frac{1}{t} \int_{\RR^d} p_0(t; x, y) \big(f(y) - f(x)\big) {\: \rm d} y \\
		&\phantom{=}
		+ e^{-t} \sum_{j = 1}^\infty t^{j-1} \int_{\RR^d} \int_0^1 p_0(tu; 0, y- c^j x) \big(f(y) -f(x)\big) P_j(u) {\: \rm d} u.
	\end{align*}
	Let us observe that in the sum the terms for $j \geq 2$ tends to zero in $\mathcal{B}_r$. Indeed, we have
	\begin{align*}
		&
		\bigg\| \int_{\RR^d} \int_0^1 p_0(tu; 0, y - c^j x) \big(f(y) - f(x) \big) P_j(u) {\: \rm d} u {\: \rm d} y
		\bigg\|_{\mathcal{B}_r(x)} \\
		&\qquad=
		\bigg\| \int_{\RR^d} \int_0^1 p_0(tu; 0, y) \big(f(y+c^j x) - f(x) \big) P_j(u) {\: \rm d} u {\: \rm d} y
		\bigg\|_{\mathcal{B}_r(x)} \\
		&\qquad\leq
		\int_{\RR^d} \int_0^1 p_0(tu; 0, y) \big\|f(y+c^j x) - f(x) \big\|_{\mathcal{B}_r(x)} {\: \rm d} u {\: \rm d} y \\
		&\qquad\leq
		\frac{1 + c^{-d j/r}}{j!} \|f\|_{\mathcal{B}_r}
	\end{align*}
	where in the last inequality we have used Corollary \ref{cor:A0}. Hence,
	\begin{align*}
		\bigg\|
		\sum_{j = 2}^\infty t^{j-1} \int_{\RR^d} \int_0^1 p_0(tu; 0, y- c^j x) \big(f(y) -f(x)\big) P_j(u)
		{\: \rm d} u
		{\: \rm d} y
		\bigg\|_{\mathcal{B}_r(x)}
		\leq
		\|f\|_{\mathcal{B}_r} t
		\sum_{j = 0}^\infty
		\frac{1 + c^{-d (j+2)/r}}{(j+2)!} t^j.
	\end{align*}
	For $j = 1$, we write
	\begin{align*}
		&\bigg\|
		\int_{\RR^d} \int_0^1 p_0(tu; 0, y) f(y + c x) P_1(u) {\: \rm d} u {\: \rm d} y
		-f(cx)
		\bigg\|_{\mathcal{B}_r(x)} \\
		&\qquad\leq
		\int_0^1 \bigg\| \int_{\RR^d} p_0(tu; 0, y) \big( f(y + c x) - f(cx)\big) {\: \rm d} y \bigg\|_{\mathcal{B}_r(x)}
		P_1(u) {\: \rm d} u,
	\end{align*}
	thus by the dominated convergence together with the strong continuity on $\mathcal{B}_r$, we conclude that
	\[
		\lim_{t \to 0^+} \int_{\RR^d} \int_0^1 p_0(tu; 0, y) f(y + c x) P_1(u) {\: \rm d} u {\: \rm d} y
		=
		f(cx),
	\]
	and the theorem follows.
\end{proof}

\begin{corollary}
	\label{cor:2}
	Under the assumptions of Theorem \ref{thm:1}, for every $r \in [1, \infty]$, if $f \in \calC_0^2(\RR^d)$ is such that
	$\partial_x^{\mathbf{a}} f \in \mathcal{B}_r$ for each $\mathbf{a} \in \NN_0^d$, $\abs{\mathbf{a}} \leq 2$, then
	$f \in D(\mathscr{A}_r)$ and
	\[
		\mathscr{A}_r f = \mathscr{A} f.
	\]
\end{corollary}
Let $\mathscr{A}^*_2$ denote the $L^2(\RR^d)$-adjoint operator to $\mathscr{A}_2$. Recall that $g \in D(\mathscr{A}_2^*)$ if
and only if there is $h \in L^2(\RR^d)$ such that
\[
	\sprod{\mathscr{A}_2 f}{g} = \sprod{f}{h} \qquad\text{for all } f \in D(\mathscr{A}_2).
\]
Then $\mathscr{A}_2^* g = h$. We have the following lemma.
\begin{lemma}
	\label{lem:3}
	Suppose that $\mathbf{Y}$ is a strictly $\alpha$-stable process in $\RR^d$, $\alpha \in (0, 2]$, with a transition density.
	Assume that $\mathbf{X}$ is obtained from $\mathbf{Y}$ by partial resetting with factor $c \in (0, 1)$. If
	$g \in \calC_0^2(\RR^d)$ is such that $\partial^{\mathbf{a}} g \in L^2(\RR^d)$ for each $\mathbf{a} \in \NN_0^d$,
	$\abs{\mathbf{a}} \leq 2$, then $g \in D(\mathscr{A}_2^*)$ and $\mathscr{A}_2^* g = \mathscr{A}^* g$.
\end{lemma}
\begin{proof}
	Let $f \in D(\mathscr{A}_2)$. By Proposition \ref{prop:9}, we have
	\begin{align*}
		\sprod{\mathscr{L}_2 f}{g}
		&=
		\lim_{t\to 0^+}
		\int_{\RR^d}
		\frac{\EE\big[f(Y_t + x)\big]-f(x)}{t} g(x) \: {\rm d} x\\
		&=
		\lim_{t\to 0^+} \int_{\RR^d} f(x) \frac{\EE\big[g(-Y_t + x)\big] -g(x)}{t} \: {\rm d} x
		=
		\sprod{f}{\mathscr{L}^*g}.
	\end{align*}
	Hence, by Theorem \ref{thm:1},
	\begin{align*}
		\sprod{\mathscr{A}_2 f}{g}
		&= \sprod{\mathscr{L}_2 f}{g}+ \int_{\RR^d} \big(f(cx) - f(x)\big) g(x) \: {\rm d}x \\
		&= \sprod{f}{\mathscr{L}^*g}+ \int_{\RR^d} f(x) \bigg(\frac{1}{c} g(cx)-g(x)\bigg) \: {\rm d}x
		= \sprod{f}{\mathscr{A}^*g}.
	\end{align*}
	Since $\mathscr{A}^*g \in L^2(\RR^d)$, we conclude that $\mathscr{A}_2^* g = \mathscr{A}^* g$.
\end{proof}

\begin{theorem}
	\label{thm:H+F-P}
	Suppose that $\mathbf{Y}$ is a strictly $\alpha$-stable process in $\RR^d$, $\alpha \in (0, 2]$, with a transition density.
	Assume that $\mathbf{X}$ is obtained from $\mathbf{Y}$ by partial resetting with factor $c \in (0, 1)$. Then
	$\rho_{\mathbf{Y}}$ defined in \eqref{representationrho} belongs to $D(\mathscr{A}_2^*)$ and
	\begin{equation}
		\label{eq:69}
		\mathscr{A}_2^* \rho_{\mathbf{Y}} = \mathscr{A}^* \rho_{\mathbf{Y}}= 0.
	\end{equation}
	Moreover,
	\[
		\partial_t p = \mathscr{A}_x p =\mathscr{A}_y^* p
	\]
\end{theorem}
\begin{proof}
	In view of Proposition \ref{prop:6} and Lemma \ref{lem:p_reg}, $\rho_{\mathbf{Y}}$ and $p$ are regular, respectively.
	Let $(P_t : t > 0)$ be the semigroup on $\calC_0(\RR^d)$ generated by $\mathbf{X}$. By Corollary \ref{cor:2},
	for $f \in \calC_c^\infty(\RR^d)$,
	\[
		\partial_t P_t f(x) = P_t \mathscr{A} f(x).
	\]
	Hence, by Lemma \ref{lem:p_reg},
	\begin{align*}
		\int_{\RR^d} f(y) \partial_t p(t; x, y) {\: \rm d} y
		&=
		\partial_t P_t f(x)
		=
		P_t \mathscr{A} f(x) \\
		&= \int_{\RR^d} \mathscr{A} f(y) p(t; x, y) {\: \rm d} y
		= \int_{\RR^d} f(y) \mathscr{A}^*_y p(t; x, y) {\: \rm d} y
	\end{align*}
	proving $\partial_t p = \mathscr{A}_y^* p$. To show that $\partial_t p = \mathscr{A}_x p$, we observe that for
	$f \in \calC_0^\infty(\RR^d)$,
	\[
		\partial_t P_t f(x) = \mathscr{A} P_t f(x)
	\]
	provided that $P_t f \in \calC_0^\infty(\RR^d)$. Therefore, in view of the Chapmann--Kolmogorov equation \eqref{eq:90},
	it is enough to take $f(x) = p(s; x, y)$, since the required regularity is a consequence of Lemma \ref{lem:p_reg}.

	To prove \eqref{eq:69}, by Corollary \ref{cor:1}, for $f \in \calC_c^\infty(\RR^d)$ we have
	\begin{align*}
		0
		=
		\int_{\RR^d} \frac{P_t f(x) - f(x)}{t} \rho_{\mathbf{Y}}(x) {\: \rm d} x,
	\end{align*}
	thus,
	\[
		0=\sprod{\mathscr{A} f}{ \rho_{\mathbf{Y}}}
		=
		\sprod{f}{\mathscr{A}^* \rho_{\mathbf{Y}}}
		=
		\sprod{f}{\mathscr{A}_2^* \rho_{\mathbf{Y}}}
	\]
	where the last equation follows by Lemma \ref{lem:3}, which completes the proof.
\end{proof}
Notice that $\mathbf{X}$ generates a strongly continuous semigroup of contractive operators on
$L^2(\RR^d,  \rho_{\mathbf{Y}}(y) {\rm d} y)$. Let us denote by $\mathscr{A}_\rho$ its infinitesimal generator.
Since $\rho_{\mathbf{Y}}(y) {\rm d} y$ is a probabilistic measure, for
$f \in \calC_0^\infty(\RR^d) \subset D(\mathscr{A}_\rho)$, we have $\mathscr{A}_\rho f = \mathscr{A}_\infty f = \mathscr{A} f$.
Let $\mathscr{A}_\rho^*$ denote the operator adjoint to $\mathscr{A}_\rho$ on $L^2(\RR^d, \rho_{\mathbf{Y}}(y) {\rm d} y)$. If
$g \in \calC_c^\infty(\RR^d)$ belongs to the domain of $\mathscr{A}^*_\rho$, then
\begin{align*}
	\int_{\RR^d} \mathscr{A}_\rho^* g (y) f(y) \: \rho_{\mathbf{Y}}(y) {\rm d} y
	&=
	\int_{\RR^d} g(y) \mathscr{A}_\rho f(y) \: \rho_{\mathbf{Y}} (y) {\rm d} y \\
	&=
	\int_{\RR^d} \rho_{\mathbf{Y}} (y) g(y)  \mathscr{A} f(y) {\rm d} y \\
	&=
	\int_{\RR^d} \mathscr{A}^* \big(\rho_{\mathbf{Y}} g\big)(y)  f(y) {\rm d} y
\end{align*}
for all $f \in \calC_c^\infty(\RR^d)$. Hence,
\begin{equation}
	\label{eq:91}
	\rho_{\mathbf{Y}} \: \mathscr{A}_\rho^* g = \mathscr{A}^* \big(\rho_{\mathbf{Y}} \: g\big).
\end{equation}
\begin{theorem}
	\label{thm:NESS}
	Suppose that $\mathbf{Y}$ is a strictly $\alpha$-stable process in $\RR^d$, $\alpha \in (0, 2]$, with a transition density.
	Assume that $\mathbf{X}$ is obtained from $\mathbf{Y}$ by partial resetting with factor $c \in (0, 1)$.
	There exists $g \in D(\mathscr{A}_\rho)$ such that if $g \in D(\mathscr{A}_\rho^*)$ then
	$\mathscr{A}_\rho g \neq  \mathscr{A}_\rho^* g$.
\end{theorem}
\begin{proof}
	We are going to construct $g \in \calC_c^\infty(\RR^d)$.
	Since for $g \in \calC_c^\infty(\RR^d)$ both $\mathscr{A}_\rho g$ and $\mathscr{A}_\rho^* g$ are at least continuous,
	if suffices to show that $\mathscr{A}_\rho g(x) \neq  \mathscr{A}_\rho^* g(x)$ for some $x\in \RR^d$ such that
	$\rho_{\mathbf{Y}}(x)>0$.
	Let $B_r(x)$ denote the open Euclidean ball of radius $r > 0$ and centered at $x \in \RR^d$. Fix
	$x_0 \in \RR^d \setminus\{0\}$ such that $\rho_{\mathbf{Y}}(x_0) > 0$. Let us consider a function
	$g_r \in \calC_c^\infty(\RR^d)$ such that
	\[
		\ind{B_{r/2}(c x_0)} \leq g_r \leq \ind{B_r(c x_0)}.
	\]
	If $(1-c) \norm{x_0} > r > 0$, then
	\begin{align*}
		\mathscr{A}_{\rho} g_r(x_0)
		&=
		\mathscr{A} g_r(x_0) \\
		&=\mathscr{L} g_r(x_0)+g_r(cx_0)-g_r(x_0) \\
		&=\int_{\RR^d} g_r(x_0+z) \: \nu({\rm d}z) + 1.
	\end{align*}
	Moreover, by \eqref{eq:91}
	\begin{align*}
		\mathscr{A}_\rho^* g_r(x_0)
		&=\frac1{\rho_{\mathbf{Y}}(x_0)} \mathscr{A}^*(\rho_{\mathbf{Y}} g_r)(x_0)\\
		&=\frac1{\rho_{\mathbf{Y}}(x_0)}\bigg( \mathscr{L}^*(\rho_{\mathbf{Y}} g_r)(x_0)
		+\frac1{c}(\rho_{\mathbf{Y}} g_r)\Big(\frac{x_0}{c}\Big)-(\rho_{\mathbf{Y}} g_r)(x_0)\bigg)\\
		&=\frac1{\rho_{\mathbf{Y}}(x_0)} \int_{\RR^d} \rho_{\mathbf{Y}}(x_0+z)g_r(x_0+z) \: \nu^*({\rm d}z).
	\end{align*}
	Since $\nu$ and $\nu^*$ both are measures without atoms, see \eqref{eq:nu-str_st}, we obtain
	\[
		\lim_{r \to 0^+}
		\mathscr{A}_\rho g_r(x_0)
		=
		1, \qquad \text{ and }\qquad \lim_{r\to 0^+} \mathscr{A}_\rho^* g_r(x_0)=0.
	\]
	Therefore, there is $r > 0$ such that $\mathscr{A}_\rho g_r(x_0) \neq \mathscr{A}_\rho^* g_r(x_0)$.
\end{proof}

\section{Asymptotic and estimates of transition density}\label{sec:4}
In this section we study the asymptotic behavior of the transition densities processes obtained by partial resetting from
$\alpha$-stable processes such as: isotropic $\alpha$-stable processes (see Section \ref{sec:stable}), $\alpha$-stable
subordinators (see Section \ref{sec:sub}) and $d$-cylindrical $\alpha$-stable processes (see Section \ref{sec:cyl}).

\subsection{Isotropic $\alpha$-stable process with $\alpha\in(0,2)$ and resetting}
\label{sec:stable}
Let $\nu$ denote the L{\'e}vy measure of an isotropic $\alpha$-stable process with $\alpha\in (0,2)$. Namely,
\begin{align}
	\label{eq:lm}
	\nu(y)=\frac{2^{\alpha}\Gamma((d+\alpha)/2)}{\pi^{d/2}|\Gamma(-\alpha/2)|} |y|^{-d-\alpha},
\end{align}
see Example~\ref{ex:names:2}. In view of the construction of the resetting, the transition density of the process $\mathbf{X}$
depends on the parameter $c \in (0, 1)$. Since in the following theorem we also study the asymptotic behavior of $p(t; x, y)$
uniformly with respect to the parameter $c$, we write $p^{(c)}(t; x, y)$ instead of $p(t; x, y)$, to indicate
the dependence on $m = c^\alpha$.
\begin{theorem}
	\label{thm:ius}
	Suppose that $\mathbf{Y}$ is an isotropic $\alpha$-stable process in $\RR^d$, $\alpha \in (0, 2)$, with a transition
	density $p_0$. Assume that $\mathbf{X}^{(c)}$ is obtained from $\mathbf{Y}$ by partial resetting with factor $c\in(0,1)$.
	Then for each $\kappa_1, \kappa_2 \in (0, 1)$, the transition density $p^{(c)}$ of $\mathbf{X}^{(c)}$ satisfies
	\begin{align}
		\label{eq:ius-1}
		\lim_{\atop{|y| \to +\infty}{t \to +\infty}} \,
		\sup_{\substack{c\in(0,\kappa_1] \\ |x| \leq \kappa_2 |y|}}
		\left| (1-m)\frac{p^{(c)}(t; x, y)}{\nu(y)}-1 \right|= 0.
	\end{align}
	Furthermore, for any fixed $K>0$, we have
	\[
		\lim_{\substack{ |y| \to +\infty \\ t \to+\infty}}\,
		\sup_{\substack{c\in (0,1)\\ |x|\leq K |y|}}
		\left|  \frac{(1-m)t}{1 -e^{-(1-m)t}}\frac{p^{(c)}(t; x,  t^{1/\alpha} y)}{t\nu(t^{1/\alpha}y)} -1 \right| = 0,
	\]
	and
	\[
		\lim_{\substack{|y| \to +\infty \\ t \to+\infty} }
		\sup_{\substack{c\in (0,1)\\ |x|\leq K }}
		\left| \frac{(1-m)t}{1 -e^{-(1-m)t}}\frac{p^{(c)}(t; x, (1-m)^{-1/\alpha}y)}{t\nu((1-m)^{-1/\alpha}y)} - 1 \right| = 0.
	\]
\end{theorem}
\begin{proof}
	It is well known that (see \cite{BlumenthalGetoor}),
	\begin{align}
		\label{approx:is}
		p_0(s;0,w)\approx \min\bigg\{ s^{-d/\alpha}, \frac{s}{|w|^{d+\alpha}}\bigg\}
	\end{align}
	uniformly with respect to $s > 0$ and $w \in \RR^d$. In particular, there is $C > 0$ such that for all $s > 0$ and
	$w \in \RR^d$,
	\begin{align}
		\label{ineq:2}
		p_0(s;0,w)\leq C s \nu(w).
	\end{align}
	We are going to prove all three statements simultaneously. We set
	\[
		\tilde{y}=y\, \qquad\qquad \tilde{y}=t^{1/\alpha}y\,, \qquad\qquad \tilde{y}=(1-m)^{-1/\alpha}y,
	\]
	respectively. Fix $\kappa_1, \kappa_2 \in (0,1)$ and $K>0$. By \eqref{eq:rep-p-x}, we get
	\begin{equation}
		\label{eq:70}
		p^{(c)}(t; x, \tilde{y})
		= e^{-t} p_0(t; 0, \tilde{y}-x) +
		e^{-t} \sum_{j = 1}^\infty t^j \int_0^1   p_0(t u; 0, \tilde{y}-c^j x)  P_j(u) {\: \rm d} u.
	\end{equation}
	We emphasize that the splines $(P_j)$ depend on $c$. If $|y| \geq K/\kappa_2$ and $t \geq (K/\kappa_2)^{\alpha}$, then
	$|x| \leq \kappa_2 |\tilde{y}|$. Hence, for $j \geq \NN_0$,
	\begin{align}
		\label{ineq:1}
		\frac{\nu(\tilde{y}-c^jx)}{\nu(\tilde{y})} \leq (1-\kappa_2)^{-d-\alpha}.
	\end{align}
	Let us observe that for $u > 0$,
	\[
		\frac{u}{1-e^{-u}} = u + \frac{u}{e^u - 1} \leq u + 1.
	\]
	Therefore, by \eqref{ineq:2} and \eqref{ineq:1}, we get
	\[
		\frac{(1-m)t}{1 -e^{-(1-m)t}}
		e^{-t}\frac{p_0(t;0,\tilde{y}-x)}{t\nu(\tilde{y})}
		\leq
		C (1 - \kappa_2)^{-d-\alpha} (1+t) e^{-t},
	\]
	and so, the first term in \eqref{eq:70} uniformly tends to zero. Next, by \eqref{ineq:2}, \eqref{ineq:1} and
	Theorem \ref{thm:all-moments} we obtain
	\begin{align}
		\nonumber
		&\frac{(1-m) t}{1 -e^{-(1-m)t}} e^{-t} t^j
		\int_0^1 \bigg( \frac{p_0(t u; 0, \tilde{y}-c^j x)}{tu\nu(\tilde{y})}+ 1 \bigg) u P_j(u) {\: \rm d} u \\
		\nonumber
		&\qquad\qquad\leq
		(1+t) e^{-t} t^j \int_0^1 \Bigg(\frac{\nu(\tilde{y} - c^j x)}{\nu(\tilde{y})} + 1\bigg) u P_j(u) {\: \rm d} u \\
		\label{ineq:b-2}
		&\qquad\qquad\leq
		\big(C (1 - \kappa_2)^{-d-\alpha} + 1 \big) (1+t) e^{-t} \frac{t^j}{j!}.
	\end{align}
	Next, let us recall that by \cite{BlumenthalGetoor}
	\[
		\lim_{s |w|^{-\alpha} \to 0^+} \frac{p_0(s; 0, w)}{s \nu(w)} = 1.
	\]
	For $j \in \NN$, we can write
	\begin{align*}
		\bigg|\frac{p_0(s; 0,\tilde{y}-c^j x)}{s \nu(\tilde{y})} - 1 \bigg|
		\leq
		\bigg|\frac{p_0(s; 0, \tilde{y}-c^j x)}{s \nu(\tilde{y}-c^j x)} - 1 \bigg|\cdot
		\frac{\nu(\tilde{y}-c^j x)}{\nu(\tilde{y})}
		+\bigg|\frac{\nu(\tilde{y}-c^j x)}{\nu(\tilde{y})} -1\bigg|.
	\end{align*}
	Since $|x|\leq \kappa_2|\tilde{y}|$, we get
	\[
		\frac{s}{|\tilde{y}-c^jx|^{\alpha}} \leq  \frac{s |\tilde{y}|^{-\alpha}}{(1-\kappa_2)^{\alpha}}.
	\]
	Hence,
	\begin{itemize}
		\item if $|x|\leq \kappa_2|y|$, then
		\[
			\frac{\nu(\tilde{y})}{(1+\kappa_1^j\kappa_2)^{d+\alpha}}
			\leq \nu(\tilde{y}-c^j x) \leq \frac{\nu(\tilde{y})}{(1-\kappa_1^j\kappa_2)^{d+\alpha}}
			\quad\text{ for all } c \in (0,\kappa_1];
		\]
		\item if $|x|\leq K |y|$, then
		\[
			\frac{\nu(\tilde{y})}{(1+Kt^{-1/\alpha})^{d+\alpha}}
			\leq \nu(\tilde{y}-c^jx) \\
			\leq \frac{\nu(\tilde{y})}{(1-Kt^{-1/\alpha})^{d+\alpha}}
			\quad \text{ for all }t > K^\alpha;
		\]
		\item if $|x|\leq K$, then
		\[
			\frac{\nu(\tilde{y})}{(1+K/|y|)^{d+\alpha}} \leq \nu(\tilde{y}-c^jx) \leq \frac{\nu(\tilde{y})}{(1-K/|y|)^{d+\alpha}}
			\quad \text{ for all } |y|>K.
		\]
	\end{itemize}
	Therefore, for a given $\epsilon>0$ there are $\delta > 0$ and $n_0\in\NN$ depending only on $d$, $\alpha$, $\epsilon$,
	$\kappa_1$, $\kappa_2$ and $K$, such that
	\begin{align}
		\label{ineq:3}
		\bigg|\frac{p_0(s; 0, \tilde{y}-c^j x)}{s \nu(\tilde{y})} - 1 \bigg| \leq \epsilon
	\end{align}
	\begin{itemize}
		\item for all $j \geq n_0$, and $c \in (0, \kappa_1]$ provided that $s |\tilde{y}|^{-\alpha} \leq \delta$;
		\item for all $j \in \NN$, $c \in (0, 1)$ and $t \geq \delta^{-1}$ provided that $\norm{x} \leq K \norm{y}$;
		\item for all $j \in \NN$, $c \in (0, 1)$ and $\norm{y} \geq \delta^{-1}$ provided that $\norm{x} \leq K$.
	\end{itemize}
	In all three cases, we further investigate (see \eqref{def:Phi} for the definition of $\Phi$)
	\begin{equation}
		\label{eq:26}
		\begin{aligned}
		&
		\frac{(1-m)t}{1 -e^{-(1-m)t}} e^{-t}
		\sum_{j = 1}^\infty  t^j \int_0^1  \frac{p_0(t u; 0, \tilde{y}-c^j x)}{t\nu(\tilde{y})}  P_j(u) {\: \rm d} u
		- \frac{(1-m)t}{1 -e^{-(1-m)t}} e^{-t}\int_0^1 u \Phi(t, u) {\: \rm d} u\\
		&\qquad\qquad
		=\frac{(1-m)t}{1 -e^{-(1-m)t}} e^{-t} \sum_{j =1}^{n_0} t^j \int_0^1
		\bigg( \frac{p_0(t u; 0, \tilde{y}-c^j x)}{tu\nu(\tilde{y})}- 1 \bigg)  u P_j(u) {\: \rm d} u \\
		&\qquad\qquad\phantom{=}
		+ \frac{(1-m)t}{1 -e^{-(1-m)t}} e^{-t} \sum_{j =n_0+1}^{\infty} t^j
		\int_{\frac{\delta |\tilde{y}|^{\alpha}}{t} \land 1}^1
		\bigg( \frac{p_0(t u; 0, \tilde{y}-c^j x)}{tu\nu(\tilde{y})}- 1 \bigg)   u P_j(u) {\: \rm d} u\\
		&\qquad\qquad\phantom{=}
		+\frac{(1-m)t}{1 -e^{-(1-m)t}}e^{-t} \sum_{j =n_0+1}^{\infty} t^j \int_0^{\frac{\delta |\tilde{y}|^{\alpha}}{t} \land 1}
		\bigg( \frac{p_0(t u; 0, \tilde{y}-c^j x)}{tu\nu(\tilde{y})}- 1 \bigg)   u P_j(u) {\: \rm d} u.
		\end{aligned}
	\end{equation}
	The first term is small due to \eqref{ineq:b-2}. Furthermore, in all considered cases we can assure that
	$|y| \geq \delta$ which together with \eqref{ineq:3} and \eqref{eq:1-moment} (see \eqref{def:mu_t} for the definition
	of $\mu_t$) allows us to bound the absolute value of third term by
	\begin{align*}
		&\frac{(1-m)t}{1 -e^{-(1-m)t}}e^{-t}\sum_{j =n_0+1}^{\infty}
		\int_0^{\frac{\delta |\tilde{y}|^{\alpha}}{t} \land 1}  t^j
		\, \bigg| \frac{p_0(t u; 0, \tilde{y}-c^j x)}{tu\nu(\tilde{y})}- 1 \bigg| \,  u P_j(u) {\: \rm d} u \\
		&\qquad\qquad
		\leq
		\frac{(1-m) \epsilon}{1-e^{-(1-m)t}} \int_0^\infty u \: \mu_t({\rm d} u) = \epsilon.
	\end{align*}
	Next, the middle term in \eqref{eq:26} equals zero in the case when $\tilde{y}=t^{1/\alpha}y$. If $\tilde{y}=y$ or
	$\tilde{y}=(1-m)^{-1/\alpha}y$, by \eqref{ineq:2}, \eqref{ineq:1} and \eqref{ineq:2-moment} we obtain
	\begin{align*}
		&\frac{(1-m)t}{1 -e^{-(1-m)t}}e^{-t} \sum_{j =1}^\infty t^j \int_{\frac{\delta |\tilde{y}|^{\alpha}}{t} \land 1}^1
		\bigg( \frac{p_0(t u; 0, \tilde{y}-c^j x)}{tu\nu(\tilde{y})}+ 1 \bigg) u P_j(u) {\: \rm d} u \\
		&\qquad\qquad \leq \frac{1-m}{1 -e^{-(1-m)t}}
		\big(C(1-\kappa_2)^{-d-\alpha}+1\big) \int_{\delta |\tilde{y}|^{\alpha} \land t}^t u \: \mu_t({\rm d} u) \\
		&\qquad\qquad\leq  \frac{1-m}{1 -e^{-(1-m)t}}
		\Big(C(1-\kappa_2)^{-d-\alpha}+1\Big)  \frac1{\delta |\tilde{y}|^{\alpha}} \int_0^\infty u^2 \: \mu_t({\rm d} u)  \\
		&\qquad\qquad\leq  \frac{2 \big(C(1-\kappa_2)^{-d-\alpha}+1\big)}{\delta(1-m)|\tilde{y}|^{\alpha}}.
	\end{align*}
	Now we use either $(1-m)\geq (1-\kappa_1^\alpha)$ or $(1-m)|\tilde{y}|^{\alpha}=|y|$. To conclude the proof, it enough to
	notice that by \eqref{eq:1-moment} the expression
	\[
		 \frac{(1-m)t}{1 -e^{-(1-m)t}} e^{-t}\int_0^1 u \Phi(t, u) {\: \rm d} u
		 = 1 -  \frac{(1-m)t}{1 -e^{-(1-m)t}} e^{-t}
	\]
	converges to the stated limits.
\end{proof}

\begin{corollary}
	\label{cor:ius}
	Under the assumptions of Theorem~\ref{thm:ius} we have
	\begin{equation}
		\label{eq:27a}
		\lim_{|y|\to +\infty} \frac{\rho_{\mathbf{Y}}(y)}{\nu(y)}= \frac1{1-m},
	\end{equation}
	and
	\begin{equation}
		\label{eq:27b}
		\rho_{\mathbf{Y}} \approx 1 \land \nu, \qquad\text{on}\quad \RR^d.
	\end{equation}
	Furthermore, for all $\delta,r>0$
	\begin{align}
		\label{approx:p_rho}
		p \approx  \rho_{\mathbf{Y}}, \qquad\text{on}\quad [\delta,\infty)\times B_r \times \RR^d.
	\end{align}
\end{corollary}
\begin{proof}
	The equality \eqref{eq:27a} follows from \eqref{eq:ius-1} by passing with $t$ to infinity and using
	\eqref{representationrho}. The integral representation \eqref{representationrho} implies that $\rho_{\mathbf{Y}}$
	is bounded from above by $\rho_{\mathbf{Y}}(0)$. Moreover, it is bounded from below by a positive constant, on every
	compact subset of $\RR^d$. From \eqref{eq:27a} we easily deduce that $\rho_{\mathbf{Y}}(y) \approx 1 \land \nu(y)$ on
	$\RR^d$. Since $\inf_{y\in B_{2r}} \rho_{\mathbf{Y}}(y) > 0$, it follows from \eqref{eq:lim_p_t_infty-unif} that there
	is $T > 0$ such that $p \approx \rho_{\mathbf{Y}}$ on $[T,\infty)\times B_r \times B_{2r}$
	Now, by positivity and continuity, see Lemma~\ref{lem:p_reg}, $p \approx \rho_{\mathbf{Y}}$ on
	$[\delta,T]\times B_r \times B_{2r}$. It remains to investigate
	$(t, x, y) \in [\delta,\infty)\times B_r \times (\RR^d \setminus B_{2r})$. Using the representation \eqref{eq:rep-p-x}
	and \eqref{approx:is}, there is $C > 0$ such that
	\begin{equation}
		\label{eq:86}
		C^{-1} p(t; 0, y) \leq p(t;x,y) \leq C p(t;0,y)
	\end{equation}
	for all $(t, x, y) \in (0,\infty)\times B_r \times (\RR^d \setminus B_{2r})$.
	Now, by \eqref{eq:rep-p-0.1}, \eqref{ineq:2} and \eqref{eq:1-moment}, for all $t>0$, $y\in\RR^d$,
	\[
		\frac{p(t;0,y)}{\nu(y)} \leq C \int_0^\infty u \: \mu_t({\rm d} u) \leq \frac{C}{1-m}\,.
	\]
	Using again \eqref{eq:1-moment} and \eqref{ineq:2-moment}, we obtain that for all $t>0$,
	\begin{align}
		\label{ineq:cut-M}
		\int_{[M,\infty)} u \: \mu_t({\rm d} u)
		\leq \frac12 \int_0^\infty u \: \mu_t({\rm d} u)
	\end{align}
	where $M= 4/(1-m^2)$. Thus, by \eqref{approx:is} we have $p_0(s;0,w) \geq C^{-1} s \nu(w)$ provided that
	$|w| \geq  r (s/M)^{1/\alpha}$. Hence, by \eqref{eq:rep-p-0.1}, for all $t \geq \delta$, $|y|\geq 2r$,
	\begin{align*}
		\frac{p(t;0,y)}{\nu(y)} \geq  C^{-1} \int_{(0,M]} u \: \mu_t({\rm d} u)
		&\geq \frac{C^{-1}}2 \int_0^\infty u \: \mu_t({\rm d} u) \\
		&\geq \frac{C^{-1}}2  \frac{1-e^{-(1-m)\delta}}{1-m},
	\end{align*}
	which together with \eqref{eq:27b} and \eqref{eq:86} completes the proof.
\end{proof}

\subsection{$\alpha$-stable subordinator with $\alpha\in (0,1)$ and resetting}
\label{sec:sub}
Let $\nu$ be the density of the L{\'e}vy measure of an $\alpha$-stable subordinator, $\alpha \in (0, 1)$. Namely,
\[
	\nu(s)=\frac{\alpha}{\Gamma\left(1-\alpha\right)} \frac{\ind{(0,\infty)}(s)}{s^{1+\alpha}},
\]
see Example~\ref{ex:names:1}. A similar result to Theorem \ref{thm:ius} holds in the present setting as well.
\begin{theorem}
	\label{thm:s-s}
	Suppose that $\mathbf{Y}$ is an $\alpha$-stable subordinator with $\alpha \in (0, 1)$. Assume that $\mathbf{X}^{(c)}$
	is obtained from $\mathbf{Y}$ by partial resetting with factor $c\in(0,1)$. Then for each $\kappa_1, \kappa_2 \in (0, 1)$,
	the transition density $p^{(c)}$ of $\mathbf{X}^{(c)}$ satisfies
	\begin{align}
		\label{eq:s-s-1}
		\lim_{\atop{y \to +\infty}{t \to +\infty}} \, \sup_{\substack{ c \in(0,\kappa_1] \\ |x| \leq \kappa_2 y}}
		\left| (1-m)\frac{p^{(c)}(t; x, y)}{\nu(y)}-1 \right|= 0\,.
	\end{align}
	Furthermore, for all $K>0$,
	\[
		\lim_{\substack{ y \to +\infty \\ t \to+\infty}}\,
		\sup_{\substack{c \in (0,1)\\ |x|\leq K y}}
		\left|  \frac{(1-m)t}{1 -e^{-(1-m)t}}\frac{p^{(c)}(t; x,  t^{1/\alpha} y)}{t\nu(t^{1/\alpha}y)} - 1 \right|
		= 0,
	\]
	and
	\[
		\lim_{\substack{y \to +\infty \\ t \to+\infty} }
		\sup_{\substack{c \in (0,1)\\ |x|\leq K }}
		\left| \frac{(1-m)t}{1 -e^{-(1-m)t}}\frac{p^{(c)}(t; x, (1-m)^{-1/\alpha}y)}{t\nu((1-m)^{-1/\alpha}y)} - 1 \right|
		= 0.
	\]	
\end{theorem}
\begin{proof}
	First, let us recall that
	\begin{equation}
		\label{eq:34}
		\lim_{s w^{-\alpha} \to 0^+} \frac{p_0(s;0,w)}{s\nu(w)}=1,
	\end{equation}
	see e.g. \cite[Theorem 37.1]{MR0344810} or \cite[Theorem 2.5.6]{MR854867}. In particular, \eqref{eq:34} describes the
	asymptotic behavior of $p_0(1/\tilde{c}; 0, y)$ with $\tilde{c}=\cos(\pi\alpha/2)$, cf. \cite[Example 24.12]{MR1739520}.
	By the formula \cite[(9)]{MR2013738} we also have
	\begin{align}
		\label{ineq:s-s-2}
		p_0(s;0,w) \leq C s \nu(w).
	\end{align}
	Now the arguments follows by the same line of reasoning as the proof of Theorem~\ref{thm:ius}.
\end{proof}

\begin{corollary}
	\label{cor:s-s}
	Under the assumptions of Theorem~\ref{thm:s-s} we have
	\begin{equation}
		\label{eq:35a}
		\lim_{y\to +\infty}
		\frac{\rho_{\mathbf{Y}}(y)}{\nu(y)}
		=
		\frac1{1-m}.
	\end{equation}
	Moreover, for each $\delta > 0$,
	\begin{equation}
		\label{eq:35b}
		\rho_{\mathbf{Y}} \approx \nu \qquad\text{on } [\delta, \infty).
	\end{equation}
	Furthermore, for each $\delta, r>0$,
	\begin{align}
		\label{approx:p_rho-2}
		p \approx \rho_{\mathbf{Y}} \qquad\text{on } [\delta,\infty)\times (-r,r)\times [\delta,\infty).
	\end{align}
\end{corollary}
\begin{proof}
	The equality \eqref{eq:35a} follows from \eqref{eq:s-s-1} by passing with $t$ to infinity and using
	\eqref{representationrho}. Since both $\rho_{\mathbf{Y}}$ and $\nu$ are continuous and positive on $(0,\infty)$, by
	\eqref{eq:35a} we get $\rho_{\mathbf{Y}} \approx  \nu$ for $y \in [\delta, \infty)$. Because
	$\inf_{y\in (\delta,2r)} \rho_{\mathbf{Y}}(y)>0$, by \eqref{eq:lim_p_t_infty-unif} there is $T > 0$ such that
	$p \approx \rho_{\mathbf{Y}}$ on $[T,\infty)\times (-r,r) \times [\delta, 2r)$. We notice that $p(t;x,y)>0$ for all
	$t>0$, $x\in\RR$ and $y>0$, see \eqref{eq:rep-p-x}. Since $p$ is continuous, see Lemma \ref{lem:p_reg}, the comparability
	$p \approx \rho_{\mathbf{Y}}$ holds on $[\delta,T] \times (-r,r) \times [\delta,2r)$. Therefore, it remains to consider
	$(t, x, y) \in [\delta,\infty)\times (-r,r)\times [2r,\infty)$. Let us observe that for $|x|\leq r$, and $y \geq 2r$,
	\[
		\nu(y-c^jx)\leq 2^{1+\alpha}\nu(y).
	\]
	Hence, using the representation \eqref{eq:rep-p-x} and \eqref{ineq:s-s-2}, for all
	$(t, x, y) \in [\delta,\infty)\times (-r,r)\times [2r,\infty)$, we get
	\begin{align*}
		p(t;x,y)
		&\leq
		2^{1+\alpha} \nu(y)\, C \int_0^\infty u \: \mu_t({\rm d} u) \\
		&\leq 2^{1+\alpha}\frac{C}{1-m} \nu(y)
	\end{align*}
	where the last estimate is a consequence of \eqref{eq:1-moment}. In view of the formula \cite[(10)]{MR2013738}, for
	$w \geq (s/M)^{1/\alpha}$ with $M=4/(1-m^2)$, and $s > 0$,
	\[
		p_0(s;0,w) \geq C^{-1} s \nu(w).
	\]
	Moreover, for $u \in (0, 1)$,
	\[
		u -c^jx \geq r (tu/M)^{1/\alpha},
	\]
	thus, using \eqref{eq:rep-p-x} together with \eqref{ineq:cut-M} and \eqref{eq:1-moment}, we get
	\begin{align*}
		p(t;x,y)
		&\geq
		e^{-t} p_0(t;0,y-x) \ind{(0,M]}(t)+e^{-t}\sum_{j=1}^\infty t^j
		\int_0^1 p_0(tu;0,y-c^jx)  \ind{(0,M]}(tu) P_j(u) {\: \rm d} u \\
		&\geq 2^{-1-\alpha} \nu(y)\, C^{-1} \int_{(0,M]} u \: \mu_t({\rm d} u) \\
		&\geq 2^{-1-\alpha} \nu(y)\,\frac{C^{-1}}2 \int_0^\infty u \: \mu_t({\rm d} u)\\
		&\geq 2^{-1-\alpha} \frac{C^{-1}}2  \frac{1-e^{-(1-m)\delta}}{1-m} \nu(y),
	\end{align*}
	which together with \eqref{eq:35b} completes the proof.
\end{proof}

\subsection{$d$-Cylindrical $\alpha$-stable process with $\alpha\in(0,2)$ and resetting}
\label{sec:cyl}
In this section we study $d$-dimensional cylindrical $\alpha$-stable process, $\alpha \in (0, 2)$, namely a L\'evy
process with the density
\[
	p_0(t; x, y) = \prod_{j=1}^d q(t; x_j, y_j)
\]
where $q$ denotes the density of the symmetric one-dimensional $\alpha$-stable process. Let $\nu$ be
the L{\'e}vy measure of the one-dimensional process, that is
\[
	\nu(r)=\frac{2^{\alpha}\Gamma((1+\alpha)/2)}{\pi^{1/2}|\Gamma(-\alpha/2)|}
	|r|^{-1-\alpha}, \qquad r \in \RR.
\]
The function $\nu$ should not be confused with the L{\'e}vy measure of the $d$-cylindrical $\alpha$-stable process,
which is singular, see Example~\ref{ex:names:3} for details.

Given $\theta\in\RR^d$ let us define
\[
	S_1=\{i \in \{1, \ldots, d\} : \theta_i\neq 0\}\,,\qquad
	d_1=\#S_1\,,\qquad
	S_0=\{1, \ldots, d\} \setminus S_1.
\]
Recall that $p(t;x,y)$ depends on $c$ (or equivalently on $m$), see \eqref{def:m}.

\begin{theorem}
	\label{thm:cylindrical}
	Suppose that $\mathbf{Y}$ is a cylindrical $\alpha$-stable process in $\RR^d$, $\alpha \in (0, 2)$, with a transition
	density $p_0$. Assume that $\mathbf{X}$ is obtained from $\mathbf{Y}$ by partial resetting with factor $c\in(0,1)$.
	Then for each $\theta\in \RR^d$, $\kappa\in (0,1)$, $K\geq 0$ and $\delta>0$, the transition density of $\mathbf{X}$
	satisfies
	\begin{align}
		\label{eq:c-1}
		\lim_{\substack{ r \to +\infty \\ t \to+\infty}}
		\sup_{y \in \mathscr{Y}} \sup_{x\in \mathscr{X}(r,t)}
		\bigg| \frac{p(t;x,r\theta+y)}{\prod_{i\in S_1}\nu(r\theta_i)}- L \bigg|=0
	\end{align}
	where
	\begin{align*}
		\mathscr{X}(r,t)
		&=
		\left\{x\in\RR^d :
		\begin{aligned}
			\abs{x_i} &\leq \kappa \norm{r \theta_i}, \quad \text{for all}\quad i \in S_1, \\
			\log_{1/c} \abs{x_i} &\leq \frac{t}{1+\delta}, \quad \text{for all}\quad i \in S_0,
		\end{aligned}
		\right\} \\
		\mathscr{Y}&=\Big\{y\in\RR^d : |y_i|\leq K \quad \text{for all}\quad i \in S_1\Big\},
	\end{align*}
	and
	\[
		L= \int_0^{\infty} u^{d_1}  \prod_{i\in S_0} q(u;0,y_i) \: \mu({\rm d}u)\,.
	\]
\end{theorem}
\begin{proof}
	Let $d_0=\#S_0$. For $s > 0$ and $w \in \RR^d$,
	\begin{align*}
		p_0(s;0,w)=
		\bigg\{ \prod_{i\in S_1} q(s;0,w_i) \bigg\}
		\bigg\{\prod_{i\in S_0} q(s;0,w_i)\bigg\}.
	\end{align*}
	Recall that for all $s>0$, $v\in\RR$, (see \cite{BlumenthalGetoor})
	\begin{align}
		\label{approx:is-c}
		q(s;0,v) \approx \min\bigg\{ s^{-d/\alpha}, \frac{s}{|v|^{1+\alpha}}\bigg\}.
	\end{align}
	In particular,
	\begin{align}
		\label{ineq:2-c}
		q(s;0,v) & \leq C \min\big\{ s^{-1/\alpha}, s \nu(v)\big\}.
	\end{align}
	Using \eqref{eq:rep-p-x} we write
	\begin{align*}
		I_1 &= \frac{p(t; x, r\theta + y)}{\prod_{i\in S_1}\nu(r\theta_i)} \\
		&=
		e^{-t} t^{d_1} \frac{p_0(t;0,r\theta+y-x)}{\prod_{i\in S_1}t \nu(r\theta_i)}
		+ e^{-t}
		\sum_{j = 1}^\infty t^j \int_0^1 \frac{p_0(t u; 0,r\theta+ y-c^j x)}{\prod_{i\in S_1} (tu)
		\nu(r\theta_i)}(tu)^{d_1} P_j(u) {\: \rm d} u.
	\end{align*}
	First, we show that
	\[
		\lim_{\substack{ r \to +\infty \\ t \to+\infty}}
		\sup_{y \in \mathscr{Y}} \sup_{x\in \mathscr{X}(r,t)} \abs{I_1-I_2}=0
	\]
	where
	\[
		I_2=e^{-t}f(t,x)+e^{-t}\sum_{j=1}^\infty t^j \int_0^1 f(tu,c^jx) P_j(u) \: {\rm d}u,
	\]
	and
	\[
		f(u,x)= u^{d_1} \prod_{i\in S_0} q(u;0,y_i-x_i).
	\]
	Let
	\[
		r > 2 \frac{K}{1-\kappa}\min\big\{\abs{\theta_i}^{-1} : i \in S_1\big\}.
	\]
	Since for $i \in S_1$,
	\begin{equation}
		\label{eq:71}
		|x_i| \leq \kappa r|\theta_i|, \qquad\text{ and }\qquad r |\theta_i| \geq |y_i| \frac{2}{1-\kappa},
	\end{equation}
	for each $j \in \NN_0$, we have
	\begin{align}
		\label{ineq:1-c}
		\frac{\nu(r\theta_i+y_i-c^jx_i)}{\nu(r\theta_i)}
		\leq \left(\frac{1-\kappa}{2}\right)^{-1-\alpha}.
	\end{align}
	Now using \eqref{ineq:2-c} and \eqref{ineq:1-c} we obtain
	\[
		e^{-t} t^{d_1} \frac{p_0(t;0,r\theta+y-x)}{\prod_{i\in S_1} t \nu(r\theta_i)}
		\leq
		\tilde{C}^{d_1} C^{d_0} e^{-t} t^{d_1-d_0/\alpha},
	\]
	and
	\[
		e^{-t}f(t,x) \leq C^{d_0} e^{-t} t^{d_1-d_0/\alpha}
	\]
	where $\tilde{C} = 2^{1+\alpha} C (1-\kappa)^{-1-\alpha}$. Hence, the first terms in $I_1$ as well as in $I_2$ are
	negligible. Furthermore, by \eqref{ineq:2-c} and \eqref{ineq:1-c} for $j\in\NN$ we have
	\begin{align}
 		e^{-t} t^j
		\int_0^1
		\bigg( \frac{p_0(tu;0,r\theta+y-c^j x)}{\prod_{i\in S_1} (tu)\nu(r\theta_i)}(tu)^{d_1}+ f(tu,c^j x)\bigg)
		P_j(u) {\: \rm d} u \nonumber \\
		\leq
		\big( \tilde{C}^{d_1}+1\big) C^{d_0} e^{-t} t^j
		\int_0^1 (tu)^{d_1-d_0/\alpha} P_j(u) {\: \rm d} u. \label{ineq:b-2-c}
	\end{align}
	Let us recall that (see e.g. \cite{BlumenthalGetoor})
	\begin{equation}
		\label{eq:72}
		\lim_{s /|v|^{\alpha} \to 0^+} \frac{q(s; 0, v)}{s \nu(v)} = 1.
	\end{equation}
	For $j \in \NN$, we write
	\begin{equation}
		\label{eq:75}
		\begin{aligned}
		\bigg|\frac{q(s; 0,r\theta_i+ y_i-c^j x_i)}{s \nu(r\theta_i)} - 1 \bigg|
		&\leq
		\bigg|\frac{q(s; 0, r\theta_i+ y_i-c^j x_i)}{s \nu(r\theta_i+ y_i-c^j x_i)} - 1 \bigg|
		\frac{\nu(r\theta_i+ y_i-c^j x_i)}{\nu(r\theta_i)}\\
		&\phantom{\leq}+\bigg|\frac{\nu(r\theta_i+ y_i-c^j x_i)}{\nu(r\theta_i)} -1\bigg|.
		\end{aligned}
	\end{equation}
	By \eqref{eq:71}, for $i \in S_1$ and $j \in \NN$,
	\begin{equation}
		\label{eq:73}
		\frac{s}{|r\theta_i +y_i-c^jx_i|^{\alpha}}
		\leq
		2^{\alpha} (1 - \kappa)^{-\alpha}
		\frac{s}{\abs{r \theta_i}^\alpha},
	\end{equation}
	and
	\begin{equation}
		\label{eq:74}
		\frac{\nu(r\theta_i)}{\big(1+K/|r\theta_i|+c^j\kappa\big)^{1+\alpha}}
		\leq
		\nu(r\theta_i+ y_i-c^j x_i)
		\leq
		\frac{\nu(r\theta_i)}{\big(1-K/|r\theta_i|-c^j\kappa\big)^{1+\alpha}}.
	\end{equation}
	Let $\epsilon>0$. By \eqref{eq:75}, in view of \eqref{eq:72}, \eqref{eq:73} and \eqref{eq:74}, we deduce that
	there are $\tilde{\delta} > 0$ and $n_0 \in \NN$, such that if $s r^{-\alpha} \leq \tilde{\delta}$, then
	for all $j > n_0$, $r \geq (\tilde{\delta})^{-1}$ and $i \in S_1$,
	\begin{equation}
		\label{ineq:3-c}
		\bigg|\frac{q(s; 0, r\theta_i+ y_{0.i}-c^j xi)}{s \nu(r\theta_i)} - 1 \bigg| \leq \epsilon.
	\end{equation}
	Both $\tilde{\delta}$ and $n_0$ depend only on $\alpha$, $\epsilon$, $\kappa$, $K$ and $\theta$. Now, we write
	\begin{equation}
		\label{eq:77}
		\begin{aligned}
		&e^{-t} \sum_{j = 1}^\infty t^j
		\int_0^1 \frac{p_0(t u; 0,r\theta+ y-c^j x)}{\prod_{i\in S_1} (tu) \nu(r\theta_i)}(tu)^{d_1} P_j(u) {\: \rm d} u
		-e^{-t}\sum_{j=1}^\infty t^j \int_0^1 f(tu,c^jx) P_j(u) \: {\rm d}u\\
		&\qquad= e^{-t} \sum_{j=1}^{n_0} t^j \int_0^1
		\bigg( \frac{p_0(tu;0,r\theta+y-c^j x)}{\prod_{i\in S_1} (tu)\nu(r\theta_i)}(tu)^{d_1} - f(tu,c^j x)\bigg) P_j(u)
		{\: \rm d} u\\
		&\qquad\phantom{=}
		+e^{-t} \sum_{j=n_0+1}^\infty  t^j\int_{\frac{\tilde{\delta} r^\alpha}{t}\land 1}^1
		\bigg( \frac{p_0(tu;0,r\theta+y-c^j x)}{\prod_{i\in S_1} (tu)\nu(r\theta_i)}(tu)^{d_1} - f(tu,c^j x) \bigg)
		P_j(u) {\: \rm d} u \\
		&\qquad\phantom{=}
		+ e^{-t} \sum_{j=n_0+1}^\infty t^j
		\int_0^{\frac{\tilde{\delta} r^\alpha}{t}\land 1}
		\bigg( \frac{p_0(tu;0,r\theta+y-c^j x)}{\prod_{i\in S_1} (tu)\nu(r\theta_i)}(tu)^{d_1} - f(tu,c^j x)\bigg)
		P_j(u) {\: \rm d} u.
		\end{aligned}
	\end{equation}
	By \eqref{ineq:b-2-c}, the first term on the right-hand side of \eqref{eq:77} is small, see also Theorem
	\ref{thm:all-moments}. The absolute value of the third term can be bounded by
	\begin{align*}
		e^{-t}
		\sum_{j = n_0+1}^\infty t^j
		\int_0^{\frac{\tilde{\delta} r^\alpha}{t} \land 1}
		\bigg| \prod_{i\in S_1} \frac{q(tu;0,r\theta_i+y_i-c^jx_i)}{tu \nu(r\theta_i)} - 1\bigg|
		\bigg( \prod_{i\in S_0} q(tu;0,y_i-c^jx_i)\bigg) (tu)^{d_1} P_j(u) {\: \rm d} u.
	\end{align*}
	In view of \eqref{ineq:2-c}, \eqref{ineq:1-c} and \eqref{ineq:3-c} the last displayed formula is bounded by
	\begin{align*}
		\epsilon e^{-t}
		\sum_{j = 1}^\infty  t^j \int_0^1 \big(\tilde{C}\vee 1 \big)^{d_1-1} C^{d_0} (tu)^{d_1-d_0/\alpha} P_j(u) {\: \rm d} u
		\leq
		\epsilon \big( \tilde{C}\vee 1 \big)^{d_1-1} C^{d_0} \sup_{t\geq 1} \int_0^\infty u^{d_1-d_0/\alpha} \mu_t({\rm d}u)
	\end{align*}
	which is small by \eqref{ineq:sup-finite}. Next, using \eqref{ineq:2-c} and \eqref{ineq:1-c}, the absolute value of the
	second term in \eqref{eq:77} is controlled as follows
	\begin{align*}
		&e^{-t}
		\sum_{j=1}^\infty
		t^j\int_{\frac{\tilde{\delta} r^\alpha}{t}\land 1}^1
		\bigg( \frac{p_0(tu;0,r\theta+y-c^j x)}{\prod_{i\in S_1} (tu)\nu(r\theta_i)}(tu)^{d_1} + f(tu,c^j x)\bigg)
		P_j(u) {\: \rm d} u\\
		&\qquad
		\leq
		\big(\tilde{C}^{d_1}+1\big) C^{d_0}  \int_{\tilde{\delta} r^\alpha \land t}^t
		u^{d_1-d_0/\alpha}  \mu_t({\rm d}u)\\
		&\qquad
		\leq \big( \tilde{C}^{d_1}+1\big) C^{d_0}
		\frac1{\tilde{\delta} r^\alpha}
		\sup_{t\geq 1} \int_0^\infty u^{d_1-d_0/\alpha+1} \mu_t({\rm d}u).
	\end{align*}
	Consequently, $|I_1-I_2|$ converges to zero as desired.
	
	Therefore, it suffices to show that
	\[
		\lim_{\substack{r \to +\infty \\ t \to +\infty}} \sup_{y \in \mathscr{Y}} \sup_{x \in \mathscr{X}(r, t)}
		\big|I_2 - L \big| = 0.
	\]
	Without loss of generality we assume that $S_0=\{1,\ldots, d_0\}$. Since $\prod_{i\in S_0} q(s;0,w_i)$ may be regarded as
	a transition density of a strictly $\alpha$-stable process in $\RR^{d_0}$, it follows from Lemma~\ref{lem:A3} that
	the family
	\[
		\mathcal{F}=\Big\{ f : \RR_+ \times \RR^d \rightarrow \RR :
		f(u, x) =   u^{d_1} \prod_{i\in S_0} q(u;0,y_i-x_i), \, y\in \RR^{d_0} \Big\}
	\]
	satisfies \eqref{eq:78}. Notice that for each $f \in \mathcal{F}$ we have $\int_0^\infty f(u,0)\: \mu({\rm d} u)=L$.
	Applying now Lemma \ref{lem:unif_conv-2} with $\delta$ replaced by $\delta/2$, we get
	\[
		\lim_{t\to +\infty} \sup_{f\in \mathcal{F}} \sup_{x \in \mathscr{X}_{\delta/2}(t)} |I_2-L|=0.
	\]
	In particular, $I_2 \rightarrow L$ uniformly with respect to $y \in \RR^{d_0}$. Furthermore, since
	$x \in \mathscr{X} (r,t)$ implies that $x \in \mathscr{X}_{\delta/2}(t)$ provided that
	\[
		t \geq (1+\delta/2)(1+\delta^{-1}) \log_{1/c} d_0,
	\]
	the theorem follows.
\end{proof}

In addition to studying the asymptotic behavior of $p$, we also show its upper and lower estimates.
\begin{proposition}
	\label{prop:cylindrical}
	Under the assumptions of Theorem~\ref{thm:cylindrical}, for all $\theta\in\RR^d$ and $K\geq 0$,
	\begin{equation}
		\label{eq:79}
		\lim_{r\to\infty} \sup_{y\in\mathscr{Y}}
		\bigg| \frac{\rho_{\mathbf{Y}}(r\theta+y)}{\prod_{i\in S_1}\nu(r\theta_i)} - L\bigg|=0 \,.
	\end{equation}
	Furthermore, for all $\delta, r>0$,
	\begin{equation}
		\label{eq:80}
		p \approx \rho_{\mathbf{Y}} \qquad\text{on } [\delta,\infty) \times B_r \times \RR^d,
	\end{equation}
	and
	\[
		\rho_{\mathbf{Y}}(y) \approx \prod_{i=1}^d 1\land \nu(y_i) \qquad\text{for } y \in \RR^d.
	\]
\end{proposition}
\begin{proof}
	To get \eqref{eq:79}, it is enough to use \eqref{representationrho} and take $t$ tending to infinity in \eqref{eq:c-1}.
	Now, to prove \eqref{eq:80}, it is enough to show that
	\[
		\frac{p(t;x,y)}{\prod_{i=1}^d 1\land \nu(y_i)} \approx 1,
		\qquad\text{on } [\delta,\infty)\times B_r \times \RR^d.
	\]
	Without loss of generality we assume that $\delta < 1$. Given $y\in\RR^d$ and $r > 0$ we set
	\begin{align*}
		S_1(y, r)&=\{i \in \{1, \ldots, d\} : |y_i|> 2r\}, &\qquad
		d_1(y, r)&=\#S_1(y, r), \\
		S_0(y, r)&=\{i \in \{1, \ldots, d\} : |y_i|\leq 2r\}, &\qquad
		d_0(y, r)&=\#S_0(y, r).
	\end{align*}
	The reasoning is analogous to the proof of Theorem \ref{thm:cylindrical}. For $s > 0$ and $y, w \in \RR^d$, we write
	\[
		p_0(s;0,y-w)= \bigg\{ \prod_{i\in S_1(y, r)} q(s;0,y_i-w_i) \bigg\}
		\bigg\{\prod_{i\in S_0(y, r)} q(s;0,y_i-w_i)\bigg\}.
	\]
	Let $t \geq \delta$ and $r \geq \norm{x}$. For $y \in \RR^d$, by \eqref{ineq:2-c}, we have
	\begin{align*}
		\frac{q(s;0,y_i-c^jx_i)}{s \nu(y_i)} &\leq 2^{1+\alpha} C,
		\qquad\text{ for all } i \in S_1(y, r),
		\intertext{and}
		q(s;0,y_i-c^jx_i) &\leq C s^{-1/\alpha}, \qquad\text{ for all } i \in S_0(y, r),
	\end{align*}
	hence by \eqref{eq:rep-p-x}
	\begin{align*}
		p(t; x,y)
		&\leq
		2^{d_1(y, r) (1+\alpha)} C^d
		\bigg(\prod_{i\in S_1(y, r)} \nu(y_i)\bigg) \int_0^\infty u^{d_1(y,r)-d_0(y,r)/\alpha} \: \mu_t({\rm d} u).
	\end{align*}
	Since the latter integral is bounded by the sum of moments of $\mu_t$ of orders $-d/\alpha$ and $d$, using
	\eqref{ineq:sup-finite} we obtain the desired upper bound.

	Next, \eqref{approx:is-c} implies that
	\begin{align*}
		q(s;,0,v) &\geq C^{-1} s \nu(v), \qquad\text{if } |v|\geq r s^{1/\alpha},
		\intertext{and}
		q(s;,0,v)&\geq C^{-1}, \qquad \text{if } \norm{v} \leq 3r, s \in [\delta/2,1].
	\end{align*}
	Therefore, for $s \in [\delta/2,1]$,
	\begin{align*}
		\frac{q(s;0,y_i-c^jx_i)}{s\nu(y_i)} &\geq 2^{-1-\alpha}C^{-1},
		\qquad\text{for all } i \in S_1(y, r),
		\intertext{and}
		q(s;0,y_i-c^jx_i) &\geq C^{-1}, \qquad\text{for all } i \in S_0(y, r).
	\end{align*}
	Therefore, by \eqref{eq:rep-p-x} we get
	\begin{align*}
		p(t;x,y)
		&\geq
		e^{-t} p_0(t;0,y-x) \ind{[\delta/2,1]}(t)+
		e^{-t}\sum_{j=1}^\infty t^j \int_0^1 p_0(tu;0,y-c^jx)
		\ind{[\delta/2,1]}(tu) P_j(u) {\: \rm d} u\\
		&\geq 2^{-d_1(1+\alpha)} C^{-d} \bigg(\prod_{i\in S_1(y, r) } \nu(y_i)\bigg)
		\int_{\delta/2}^1 u^{d_1(y, r)} \mu_t({\rm d}u)\\
		&\geq 2^{-d_1(y, r) (1+\alpha)} C^{-d} (\delta/2)^d \bigg(\prod_{i\in S_1(y, r)}
		\nu(y_i)\bigg)
		\inf_{t \in [t_0,\infty)} \mu_t\big([\delta/2,1]\big).
	\end{align*}
	To conclude the proof, we invoke Lemma~\ref{lem:inf_mu_t}.
\end{proof}

\subsection{Brownian motion with resetting}
In this section we consider the case where $\mathbf{Y}$ is a Brownian motion with a transition density
\begin{align}
	\label{p0gauss}
	p_0(t;x,y)=(4\pi t)^{-\frac{d}{2}} e^{-\frac{|y-x|^2}{4t}}.
\end{align}
We are going to study the asymptotic behavior of $p(t; 0, y)$ when $t$ tends to infinity while $(t, y)$ stays in
certain space-time regions. By Corollary \ref{cor:rep-2} and \eqref{p0gauss}, we have
\begin{align}
	\label{eq:p_gauss}
	p(t; 0, y) = e^{-t} (4\pi t)^{-\frac{d}{2}} e^{-\frac{|y|^2}{4t}}
	+
	t e^{-t}(4\pi t)^{-\frac{d}{2}} \int_0^1 u^{-\frac{d}{2}} e^{-\frac{|y|^2}{4 tu}+\log \frac{\Phi(t, u)}{t}} {\: \rm d} u.
\end{align}
Let us define
\[
	\phi(t) = \sum_{j = 0}^\infty \frac{1}{j!} \frac{1}{(m; m)_{j+1}} t^j, \qquad t \geq 0.
\]
In view of \eqref{def:Phi} and Proposition \ref{prop:2}
\begin{equation}
	\label{eq:29}
	\Phi(t, u) \leq t \phi(t(1-u)), \qquad \text{for all } t >0,\, u \in [0, 1],
\end{equation}
and
\begin{equation}
	\label{eq:30}
	\Phi(t, u) = t \phi(t(1-u)), \qquad \text{for all } t > 0,\, u \in [m, 1].
\end{equation}
We first show a few key properties of the function $\phi(t)$. Note that for $k\in\NN_0$,
\begin{equation}
	\label{eq:31}
	\phi^{(k)}(t) = \sum_{j = 0}^\infty \frac{t^j}{j!} \frac{1}{(m; m)_{j+k+1}}, \qquad t \geq 0.
\end{equation}

\begin{proposition}
	\label{prop:7}
	For each $k \in \NN_0$, and $t > 0$,
	\begin{equation}
		\label{eq:10}
		\frac{e^t}{(m; m)_{k+1}} \leq \phi^{(k)}(t) \leq \frac{e^t}{(m; m)_\infty},
	\end{equation}
	and
	\begin{equation}
		\label{eq:11}
		\phi(t) \leq \phi^{(k)}(t) \leq \frac{1}{(m; m)_k} \phi(t).
	\end{equation}
\end{proposition}
\begin{proof}
	Since
	\[
		\frac{1}{(m; m)_{k+1}} \leq \frac{1}{(m; m)_{j+k+1}} \leq \frac{1}{(m; m)_{\infty}},
	\]
	we easily get \eqref{eq:10}. Moreover,
	\[
		\frac{1}{(m; m)_{j+1}}
		\leq
		\frac{1}{(m; m)_{j+k+1}}
		=
		\frac{1}{(m; m)_{j+1}} \prod_{\ell = 0}^{k-1} \frac{1}{1-m^{j+2+\ell}}
		\leq
		\frac{1}{(m; m)_{j+1}} \frac{1}{(m; m)_{k}}
	\]
	which leads to \eqref{eq:11}.
\end{proof}

\begin{proposition}
	\label{prop:8}
	For each $n \in \NN$ and $k \in \NN$,
	\begin{equation}
		\label{lim:phi-1}
		\lim_{t \to \infty} t^n \frac{\phi^{(k+1)}(t) - \phi^{(k)}(t)}{\phi^{(k)}(t)} = 0,
	\end{equation}
	and
	\begin{align}
		\label{lim:phi-2}
		\lim_{t \to \infty} t^n \Big(\frac{\phi(t)}{e^t} - \frac1{(m;m)_\infty}\Big) = 0.
	\end{align}
\end{proposition}
\begin{proof}
	Using \eqref{eq:31}, we get
	\[
		\phi^{(k+1)}(t) = \sum_{j = 0}^\infty \frac{t^j}{j!} \frac{1}{(m; m)_{j+k+1}} \frac{1}{1-m^{j+k+2}}.
	\]
	Thus
	\begin{align*}
		\phi^{(k+1)}(t)
		&= \sum_{j = 0}^\infty
		\frac{t^j}{j!} \frac{1}{(m; m)_{j+k+1}}
		\sum_{\ell = 0}^\infty m^{(j+k+2)\ell} \\
		&=\sum_{\ell=0}^\infty m^{(k+2)\ell} \phi^{(k)}\big(m^\ell t\big).
	\end{align*}
	By Proposition \ref{prop:7}, for $\ell \in \NN$,
	\[
		\frac{\phi^{(k)}(m^\ell t)}{\phi^{(k)}(t)} \leq \frac{(m; m)_{k+1}}{(m; m)_\infty} e^{-(1-m^\ell) t},
	\]
	hence
	\[
		\lim_{t \to +\infty} t^n \frac{\phi^{(k)}(m^\ell t)}{\phi^{(k)}(t)} = 0.
	\]
	Since for each $\ell\in\NN$,
	\begin{align*}
		t^n \frac{\phi^{(k)}(m^\ell t)}{\phi^{(k)}(t)}
		&\leq \frac{(m; m)_{k+1}}{(m; m)_\infty} n! \big(1-m^\ell\big)^{-n} \\
		&\leq \frac{(m; m)_{k+1}}{(m; m)_\infty} n! (1-m)^{-n},
	\end{align*}
	by the Lebesgue's dominated convergence theorem
	\[
		\lim_{t \to +\infty} \sum_{\ell = 1}^\infty m^{(k+2)\ell} t^n \frac{\phi^{(k)}(m^\ell t)}{\phi^{(k)}(t)} = 0,
	\]
	which completes the proof of \eqref{lim:phi-1}.	
	
	Next, we write
	\[
		\frac{1}{(m; m)_\infty} e^{t} - \phi(t) = \sum_{j = 0}^\infty \frac{t^j}{j!} \frac{1}{(m; m)_\infty}
		\Big(1 - \prod_{k = j+2}^\infty (1-m^k)\Big).
	\]
	By the generalized Bernoulli's inequality,
	\[
		1 - \prod_{k = j+2}^\infty (1 - m^k) \leq \sum_{k = j+2}^\infty m^k,
	\]
	hence
	\begin{align*}
		\frac{1}{(m; m)_\infty} e^{t} - \phi(t)
		&\leq
		\sum_{j = 0}^\infty \frac{t^j}{j!} \frac{1}{(m; m)_\infty} \sum_{k = j+2}^\infty m^k \\
		&\leq
		\frac{1}{(m; m)_\infty} \frac{m^2}{1-m} e^{m t}.
	\end{align*}
	Lastly,
	\begin{align*}
		0 \leq \Big(\frac{1}{(m; m)_\infty} e^{t} - \phi(t)\Big)e^{-t}
		\leq \frac{1}{(m; m)_\infty} \frac{m^2}{1-m}
		e^{-(1-m)t}
	\end{align*}
	which completes the proof of \eqref{lim:phi-2}.
\end{proof}

\begin{lemma}
	\label{lem:1}
	For all $t \geq 0$,
	\[
		\frac{\phi''(t)}{\phi(t)} - \bigg(\frac{\phi'(t)}{\phi(t)} \bigg)^2 < 0.
	\]
\end{lemma}
\begin{proof}
	Using \eqref{eq:31}, for $s,t\geq 0$, we obtain
	\[
		\phi'(s) \phi(t) - \phi'(t) \phi(s)
		=
		\sum_{k = 0}^\infty \sum_{\stackrel{j_1 + j_2 = k}{0 \leq j_1, j_2}}
		\frac{s^{j_1} t^{j_2}}{j_1! j_2!} \frac{1}{(m; m)_{j_1+1}}
		\frac{1}{(m; m)_{j_2+1}} \Big(\frac{1}{1-m^{j_1+2}} - \frac{1}{1-m^{j_2+2}}\Big).
	\]
	Observe that we can assume that $j_1 \neq j_2$. If $0 \leq j_1 < j_2$, then
	\[
		\frac{1}{1-m^{j_1+2}} - \frac{1}{1-m^{j_2+2}} = \frac{m^{j_1+2} - m^{j_2+2}}{(1-m^{j_1+2}) (1-m^{j_2+2})} > 0.
	\]
	Hence,
	\begin{align*}
		&
		\sum_{\stackrel{j_1 + j_2 = k}{0 \leq j_1, j_2}}
		\frac{s^{j_1} t^{j_2}}{j_1! j_2!} \frac{1}{(m; m)_{j_1+1}}
		\frac{1}{(m; m)_{j_2+1}} \Big(\frac{1}{1-m^{j_1+2}} - \frac{1}{1-m^{j_2+2}}\Big) \\
		&\qquad=
		\sum_{\stackrel{j_1 + j_2 = k}{0 \leq j_1 < j_2}}
		\frac1{j_1! j_2!} \frac{1}{(m; m)_{j_1+1}} \frac{1}{(m; m)_{j_2+1}}
		\Big(\frac{1}{1-m^{j_1+2}} - \frac{1}{1-m^{j_2+2}}\Big)
		\Big(s^{j_1} t^{j_2} - s^{j_2} t^{j_1}\Big).
	\end{align*}
	Therefore, $\phi''(s) \phi(t) - \phi'(t) \phi'(s)$ equals
	\begin{align*}
		\sum_{k = 1}^\infty
		\sum_{\stackrel{j_1 + j_2 = k}{0 \leq j_1 < j_2}}
		\frac1{j_1! j_2!} \frac{1}{(m; m)_{j_1+1}} \frac{1}{(m; m)_{j_2+1}}
		\Big(\frac{1}{1-m^{j_1+2}} - \frac{1}{1-m^{j_2+2}}\Big)
		\Big(j_1 s^{j_1-1} t^{j_2} - j_2 s^{j_2-1} t^{j_1}\Big).
	\end{align*}
	Now, to conclude the proof it suffices to put $s=t$ and use that $j_1<j_2$.
\end{proof}

\begin{proposition}
	\label{prop:log_phi}
	For each $k, n \in \NN$, $k \geq 2$, there is $C_{n, k} > 0$ such that for all $t \geq 0$,
		\[
			t^n \Big| \frac{{\rm d}^k}{{\rm d} t^k} \log \phi(t) \Big| \leq C_{n, k}.
		\]
\end{proposition}
\begin{proof}
	Let us recall that for a positive smooth real function $f$, by Fa\'a di Bruno's formula, there are positive constants
	$C_{\ell}$ such that
	\[
		\frac{{\rm d}^k }{{\rm d} t^k} \log f(t) = \sum_{j = 1}^k j! (-1)^{j+1} \frac{1}{(f(t))^j}
		\sum_{\ell} C_\ell \prod_{i = 1}^j \frac{{\rm d}^{\ell_i}}{{\rm d} t^{\ell_i}} f(t)
	\]
	where in the inner sum $\ell$ runs over all sequences $\ell = (\ell_1, \ldots, \ell_j)$, $\ell_i \in \NN$, with
	$\ell_1+\ldots+\ell_j = k$. Since the coefficients $C_\ell$ are independent of $f$, by taking $f(t) = e^t$ and $k \geq 2$,
	we get
	\[
		0 = \sum_{j = 1}^k j! (-1)^{j+1} \sum_{\ell} C_\ell.
	\]
	Hence, for $f \equiv \phi$, and $k \geq 2$, we obtain
	\begin{align*}
		\frac{{\rm d}^k}{{\rm d} t^k} \log \phi(t)
		&= \sum_{j = 1}^k j! (-1)^{j+1} \sum_{\ell} C_\ell \prod_{i = 1}^j \bigg(\frac{\phi^{(\ell_i)}(t)}{\phi(t)} - 1 + 1\bigg)
		-\sum_{j = 1}^k j! (-1)^{j+1} \sum_{\ell} C_\ell \\
		&= \sum_{j = 1}^k j! (-1)^{j+1} \sum_{\ell} C_\ell \sum_{\stackrel{I \subset \{1, \ldots, j\}}{I \neq \emptyset}}
		\prod_{i \in I} \bigg(\frac{\phi^{(\ell_i)}(t)}{\phi(t)} - 1\bigg).
	\end{align*}
	Now, the claim follows by \eqref{lim:phi-1} and \eqref{eq:11}.
\end{proof}

In view of \eqref{eq:p_gauss}, \eqref{eq:29}, and \eqref{eq:30} our aim is to understand the critical points of the function
\begin{equation}
	\label{psi}
	\psi(u) = -\frac{d}{2} \log u -\frac{L}{u} + \log \phi\big(t(1-u)\big), \qquad u \in (0, 1)
\end{equation}
where
\begin{equation}\label{A}
	L = \frac{|y|^2}{4 t}.
\end{equation}
We note that
\begin{equation}
	\label{eq:33}
	\psi'(u) = -\frac{d}{2} \frac{1}{u} + \frac{L}{u^2} - t \frac{\phi'\big(t(1-u)\big)}{\phi\big(t(1-u)\big)},
\end{equation}
and
\begin{equation}
	\label{eq:37}
	\psi''(u) = \frac{d}{2} \frac{1}{u^2} - 2 \frac{L}{u^3} +
	t^2\bigg\{
	\frac{\phi''\big(t(1-u)\big)}{\phi\big(t(1-u)\big)} -
	\bigg(\frac{\phi'\big(t(1-u)\big)}{\phi\big(t(1-u)\big)}\bigg)^2
	\bigg\}.
\end{equation}
The following lemma is instrumental in treating the integral \eqref{eq:p_gauss} restricted to the interval $(0, m)$.
\begin{lemma}
	\label{lem:psi_prim}
	If
	\[
		\frac{L}{t} \geq m^2 + \frac{C}{t}
	\]
	for certain $C > \frac{d}{2} m$, then $\psi'$ is decreasing and $\psi'(m)>0$. Moreover, there is $T_0 > 0$ such that
	for all $t \geq T_0$,
	\[
		\int_0^m u^{-\frac{d}{2}} e^{-\frac{|y|^2}{4 tu}+\log \frac{\Phi(t, u)}{t}} {\: \rm d} u
		\leq m e^{\psi(m)}.
	\]
\end{lemma}
\begin{proof}
	In view of \eqref{eq:37} and Lemma \ref{lem:1}, $\psi''(u) < 0$ for all $u \in (0, 1)$ provided that $L > d/4$.
	Next, by \eqref{eq:33} we have
	\[
		\psi'(m) = -\frac{d}{2} \frac{1}{m}
		+ t\bigg(1 - \frac{\phi'\big(t(1-m)\big)}{\phi\big(t(1-m)\big)}\bigg)
		+ \frac{t}{m^2} \bigg(\frac{L}{t} - m^2\bigg).
	\]
	Since
	\[
		-\frac{d}{2} \frac{1}{m}
		+ \frac{t}{m^2} \Big(\frac{L}{t} - m^2\Big)
		\geq -\frac{d}{2} \frac{1}{m}+\frac{C}{m^2}>0,
	\]
	by \eqref{lim:phi-1}, we get $\psi'(m) > 0$ provided that $t$ is sufficiently large. Lastly, by \eqref{eq:29} and
	\eqref{psi} we get
	\begin{align*}
		\int_0^m u^{-\frac{d}{2}} e^{-\frac{|y|^2}{4 tu}+\log \frac{\Phi(t, u)}{t}} {\: \rm d} u
 		&\leq \int_0^m e^{\psi(u)} {\: \rm d} u
	\end{align*}
	which is bounded by $m e^{\psi(m)}$.
\end{proof}

Before we formulate the next result let us define
\[
	\vphi(r) = \int_0^\infty e^{-s} \phi(r s) {\: \rm d} s, \qquad\text{ for } r \in [0, 1).
\]
Using the definition of the function $\phi$ we can write
\begin{align*}
	\vphi(r)
	&=
	\sum_{j = 0}^\infty \frac{1}{j!} \frac{r^j}{(m; m)_{j+1}} \int_0^\infty s^j e^{-s} {\: \rm d} s \\
	&=
	\sum_{j = 0}^\infty \frac{r^j}{(m; m)_{j+1}}.
\end{align*}

\begin{theorem}
	\label{thm:5}
	Suppose that $\mathbf{Y}$ is Brownian motion in $\RR^d$. Assume that $\mathbf{X}$ is obtained from $\mathbf{Y}$
	by partial resetting with factor $c \in (0,1)$. Then for each $\delta > 0$, the transition density of $\mathbf{X}$
	satisfies
	\[
		p(t; 0, y)
		= e^{-t}
		(4\pi t)^{-\frac{d}{2}} e^{-\frac{|y|^2}{4t}} \bigg\{
		1 + \bigg(\frac{4t^2}{|y|^2}\bigg)  \vphi\bigg(\frac{4t^2}{\norm{y}^2}\bigg)+
		\calO\bigg(\frac{t}{\norm{y}^2}\bigg)
		\bigg\}
	\]
	as $t$ tends to infinity, uniformly in the region
	\begin{equation}
		\label{rg1}
		\Big\{(t, y) \in \RR_+ \times \RR^d : \frac{|y|^2}{4t^2} \geq 1 +\delta \Big\}.
	\end{equation}
\end{theorem}
\begin{proof}
	Thank to \eqref{eq:p_gauss}, we can write
	\begin{equation}
		\label{eq:58}
		e^t (4 \pi t)^{\frac{d}{2}}
		p(t; 0, y) = e^{-\frac{\norm{y}^2}{4 t}} + t
		\int_0^m u^{-\frac{d}{2}}
		u^{-\frac{d}{2}} e^{-\frac{\norm{y}^2}{4tu} + \log \frac{\Phi(t, u)}{t}} {\: \rm d} u
		+t I
	\end{equation}
	where we have set
	\[
		I = \int_m^1 u^{-\frac{d}{2}} e^{-\frac{|y|^2}{4 tu}+\log \frac{\Phi(t, u)}{t}} {\: \rm d} u.
	\]
	Using \eqref{eq:30}, we get
	\[
		I=
		\int_m^1
		u^{-\frac{d}{2}} e^{-\frac{L}{u}}
		\phi\big(t(1-u)\big) {\: \rm d} u.
	\]
	Let us observe that for $(t, y)$ in the region \eqref{rg1} we have $L \geq (1+\delta) t$.
	By the change of variable $v = L(u^{-1}-1)$, we obtain
	\begin{align*}
		\int_m^1 u^{-\frac{d}{2}} e^{-\frac{L}{u}} \phi\big(t(1-u)\big)
		{\: \rm d} u
		=
		\frac{1}{L} e^{-L}
		\int_0^{L (m^{-1}-1)} e^{-v}
		\Big(1 + \frac{v}{L} \Big)^{\frac{d}{2}-2}
		\phi\Big(\frac{t}{L} \frac{L v}{v+L} \Big)
		{\: \rm d} v.
	\end{align*}
	By \eqref{eq:10}, for $v > 0$ we have
	\begin{align*}
		e^{-v} \phi\Big(\frac{t}{L} \frac{L v}{v+L} \Big)
		&\leq
		\frac{1}{(m; m)_\infty}
		\exp\Big\{-v + v \frac{t}{L} \frac{L}{v+L} \Big\} \\
		&\leq
		\frac{1}{(m; m)_\infty}
		\exp\Big\{-v + v \frac{1}{1+\delta} \Big\} \\
		&=
		\frac{1}{(m; m)_\infty} e^{-\frac{\delta}{1+\delta} v}.
	\end{align*}
	Moreover, for $v \in [0, L(m^{-1} - 1) ]$,
	\begin{align*}
		\Big|\Big(1 + \frac{v}{L} \Big)^{\frac{d}{2}-2} - 1 \Big|
		\leq C \frac{v}{L}.
	\end{align*}
	Hence,
	\begin{align*}
		&\bigg|
		\int_0^{L(m^{-1}-1)}
		e^{-v}
		\Big(1 + \frac{v}{L} \Big)^{\frac{d}{2}-2}
		\phi\Big(\frac{t}{L} \frac{L v}{v+L} \Big)
		{\: \rm d} v
		-
		\int_0^{L(m^{-1}-1)} e^{-v}
		\phi\Big(\frac{t}{L} \frac{L v}{v+L} \Big)
		{\: \rm d} v
		\bigg| \\
		&\qquad\qquad \leq
		\int_0^{L(m^{-1}-1)} e^{-v} \phi\Big(\frac{t}{L} \frac{L v}{v+L} \Big)
		\Big|\Big(1 + \frac{v}{L} \Big)^{\frac{d}{2}-2} - 1 \Big|
		{\: \rm d} v \\
		&\qquad\qquad\leq
		\frac{C}{(m;m)_\infty}
		\frac1{L} \int_0^\infty e^{-\frac{\delta}{1+\delta} v} v {\: \rm d} v.
	\end{align*}
	Since $\phi'$ is increasing, by the mean value theorem and \eqref{eq:10},
	\begin{align*}
		\bigg|
		\phi\Big(\frac{t}{L} \frac{L v}{v+L} \Big) - \phi\Big(\frac{t}{L} v \Big)
		\bigg|
		&\leq
		\phi'\Big(\frac{t}{L} v \Big) \Big|\frac{t}{L} v \frac{L}{v+L} - \frac{t}{L} v \Big| \\
		&\leq
		\frac{v}{(m;m)_\infty} \exp\Big\{\frac{t}{L} v\Big\} \Big|\frac{L}{v+L} - 1\Big| \\
		&\leq
		\frac{v^2}{(m;m)_\infty} \frac1{L} \exp\Big\{\frac{t}{L} v\Big\} \\
		&\leq
		\frac1{(m;m)_\infty} \frac1{L} v^2 e^{\frac{1}{1+\delta} v}.
	\end{align*}
	Therefore,
	\begin{align*}
		&\bigg|
		\int_0^{L(m^{-1}-1)}
		e^{-v}
		\phi\Big(\frac{t}{L} \frac{L v}{v+L} \Big)
		{\: \rm d} v
		-
		\int_0^{L(m^{-1}-1)} e^{-v} \phi\Big(\frac{t}{L}v\Big)
		{\: \rm d} v
		\bigg| \\
		&\qquad\qquad\leq
		\int_0^{L(m^{-1}-1)}
		e^{-v}
		\bigg|\phi\Big(\frac{t}{L} \frac{L v}{v+L} \Big) - \phi\Big(\frac{t}{L}v\Big) \bigg|
		{\: \rm d} v \\
		&\qquad\qquad\leq
		\frac1{(m;m)_\infty} \frac1{L} \int_0^\infty e^{-\frac{\delta}{1+\delta} v} v^2 {\: \rm d} v.
	\end{align*}
	Lastly, we can write
	\begin{align*}
		\int_0^{L(m^{-1}-1)} e^{-v} \phi\Big(\frac{t}{L}v\Big) {\: \rm d} v
		&=
		\int_0^\infty  e^{-v} \phi\Big(\frac{t}{L}v\Big) {\: \rm d} v
		-
		\int_{L(m^{-1}-1)}^\infty
		e^{-v} \phi\Big(\frac{t}{L}v\Big) {\: \rm d} v \\
		&=
		\vphi\Big(\frac{t}{L}\Big)
		-
		\int_{L(m^{-1}-1)}^\infty
		e^{-v} \phi\Big(\frac{t}{L}v\Big) {\: \rm d} v
	\end{align*}
	and so, by \eqref{eq:10},
	\begin{align*}
		\int_{L(m^{-1}-1)}^\infty
		e^{-v} \phi\Big(\frac{t}{L}v\Big) {\: \rm d} v
		&\leq
		\frac1{(m;m)_\infty}
		\int_{L(m^{-1}-1)}^\infty \exp\Big\{-\Big(1-\frac{t}{L}\Big) v\Big\} {\: \rm d} v \\
		&\leq
		\frac1{(m;m)_\infty}
		\int_{L(m^{-1}-1)}^\infty e^{-\frac{\delta}{1+\delta} v} {\: \rm d} v \\
		&\leq
		 \frac1{(m;m)_\infty} e^{-L (1-m) \frac{\delta}{1+\delta}}
		 \int_0^\infty e^{-(1-m) \frac{\delta}{1+\delta} v} {\: \rm d} v.
	\end{align*}
	Summarizing, we obtain
	\begin{align*}
		I=
		e^{-L} \Big( \frac{t}{L}\Big)
		\Big\{ \vphi\Big(\frac{t}{L}\Big) + \calO\big(L^{-1}\big)\Big\}.
	\end{align*}
	Hence, by \eqref{eq:58} and Lemma \ref{lem:psi_prim}
	\begin{align*}
		\Big|
		e^t (4 \pi t)^{\frac{d}{2}} p(t; 0, y) - e^{-L} - t I
		\Big|
		&\leq
		t \int_0^m u^{-\frac{d}{2}} e^{-\frac{L}{u} + \log \frac{\Phi(t, u)}{t}} {\: \rm d} u \\
		&\leq
		m t e^{\psi(m)},
	\end{align*}
	provided that $t$ is sufficiently large. To complete the proof we need to show that
	\[
		t e^{\psi(m)+L}=\calO\big(L^{-1}\big).
	\]
	To see this, we apply \eqref{eq:10} to get
	\begin{align*}
		t e^{\psi(m)+L}
		&\leq C_1 \exp\Big\{-\frac{L}{m} +t(1-m)+L\Big\}\\
		&=C_1 \exp\Big\{ -L (1-m) \Big(\frac{1}{m} - \frac{t}{L}\Big) \Big\}\\
		&\leq C_1 \exp\Big\{ -L (1-m) \Big(\frac{1}{m} -1\Big) \Big\}.
	\end{align*}
	and the theorem follows.
\end{proof}

\begin{corollary}
	Under the assumptions of Theorem \ref{thm:5},
	\[
		\lim_{t \to +\infty}
		\frac{p(t;0,y)}{\abs{y}^{-\frac{d-1}{2}} e^{-\abs{y}}} = 0
	\]
	uniformly in the region \eqref{rg1}.
\end{corollary}
\begin{proof}
	Since $2\sqrt{L t} = |y|$, by Theorem~\ref{thm:5} we get
	\[
		\lim_{t \to +\infty}
		\frac{p(t;0,y)}{\abs{y}^{-\frac{d-1}{2}} e^{-\abs{y}}}
		=
		\lim_{t \to +\infty}
		t^{-\frac12} \exp\Big\{-t-L -\frac{d-1}{2}\log t +2\sqrt{L t} +\frac{d-1}{2}\log |y|\Big\}.
	\]
 	Now, it is sufficient to observe that in the region \eqref{rg1},
	\[
		-t \Big(\sqrt{\frac{L}{t}}-1\Big)^2 +\frac{d}{2} \log \sqrt{\frac{L}{t}}
	\]
	is uniformly bounded from above.
\end{proof}

\begin{theorem}
	\label{thm:6}
	Suppose that $\mathbf{Y}$ is a Brownian motion in $\RR^d$. Assume that $\mathbf{X}$ is obtained from $\mathbf{Y}$
	by partial resetting with factor $c\in(0,1)$. Then for each $\delta > 0$, the transition density of $\mathbf{X}$
	satisfies
	\[
		p(t; 0, y)
		=
		\frac{1}{2} \frac1{(m;m)_\infty} (2\pi)^{-\frac{d-1}{2}} \abs{y}^{-\frac{d-1}{2}} e^{-\abs{y}}
		\Big(1 + \calO\big(t^{-1}\big)\Big)
	\]
	as $t$ tends to infinity, uniformly in the region
	\begin{equation}	
		\label{eq:38}
		\Big\{(t, y) \in \RR_+ \times \RR^d : m^2 +\delta \leq \frac{\norm{y}^2}{4t^2} \leq 1 - \delta \Big\}.
	\end{equation}
\end{theorem}
\begin{proof}
	Notice that for $(t, y)$ in the region \eqref{eq:38}, $L$ is comparable to $t$. Now, let us observe that
	\[
		\psi'(1) = -\frac{d}{2} + t \bigg(\frac{L}{t} - \frac{\phi'(0)}{\phi(0)}\bigg),
		\quad\text{ and }\quad
		\frac{\phi'(0)}{\phi(0)}=\frac1{1-m^2},
	\]
	thus $\psi'(1)<0$. Using now Lemma~\ref{lem:psi_prim} we conclude that $\psi$ has the unique critical point $u_0$. Moreover,
	$u_0$ belongs to $(m, 1)$. In fact, we claim that $u_0 \in (m, 1-\delta')$ for certain $\delta' > 0$. Indeed, suppose that
	$u_0$ tends to $1$. We can assume that $t(1-u_0)$ converges to $g \in [0, \infty]$. Since $u_0$ satisfies \eqref{eq:33},
	we have
	\[
		\frac{1}{u_0^2} - 1 = \frac{d}{2 L} \frac{1}{u_0} +
		\frac{t}{L} \bigg( \frac{\phi'\big(t(1-u_0) \big)}{\phi\big(t(1-u_0) \big)}
		-
		\frac{\phi'(g)}{\phi(g)}\bigg) + \frac{t}{L}\bigg(\frac{\phi'(g)}{\phi(g)} - \frac{L}{t}\bigg).
	\]
	Hence
	\[
		\frac{t}{L}\bigg(\frac{\phi'(g)}{\phi(g)} - \frac{L}{t}\bigg) = o(1)
	\]
	which leads to a contradiction because according to \eqref{eq:11},
	\[
		1 \leq \frac{\phi'(t)}{\phi(t)} \leq \frac{1}{1-m^2}
	\]
	for all $t \geq 0$. This proves the claim.

	Next, we determine the asymptotic behavior of $u_0$. Since $u_0$ solves \eqref{eq:33}, we obtain
	\[
		u_0 = \frac{2L}{\dfrac{d}{2} +
		\sqrt{\Big(\dfrac{d}{2}\Big)^2 + 4 L t \dfrac{\phi'(t(1-u_0))}{\phi(t(1-u_0))}}}.
	\]
	Because $t(1-u_0)$ tends to infinity, by \eqref{lim:phi-1},
	\[
		\frac{\phi'\big(t(1-u_0) \big)}{\phi\big(t(1-u_0)\big)}
		=
		1 + o\big(t^{-2}\big),
	\]
	and consequently
	\begin{align}
		\nonumber
		\frac{1}{u_0}
		&=
		\frac{d}{4 L}
		+\sqrt{\Big(\frac{d}{2}\Big)^2 \frac{1}{4 L^2} + \frac{t}{L} + \frac{t}{L}
		\Big(\frac{\phi'(t(1-u_0))}{\phi(t(1-u_0))} - 1\Big)} \\
		\nonumber
		&=
		\frac{d}{4 L}
		+\sqrt{\frac{t}{L}}
		\sqrt{
		1+ \Big(\frac{d}{2}\Big)^2 \frac{1}{4L t} + \Big(\frac{\phi'(t(1-u_0))}{\phi(t(1-u_0))} - 1\Big)} \\
		\nonumber
		&=
	 	\frac{d}{4 L}+\sqrt{\frac{t}{L}}
		\Big(1 + \calO\big(t^{-2}\big)\Big) \\
		\label{eq:48}
		&= \frac{d}{4 L} + \sqrt{\frac{t}{L}} + \calO\big(t^{-2}\big).
	\end{align}
	Hence
	\begin{align*}
		u_0
		&= \frac{1}{ \frac{d}{4L } + \sqrt{\frac{t}{L}} + \calO\big(t^{-2}\big)} \\
		&= \sqrt{\frac{L}{t}} \frac{1}{1+\frac{d}{4 \sqrt{t L}} + \calO\big(t^{-2}\big)} \\
		&= \sqrt{\frac{L}{t}} \Big(1 - \frac{d}{4 \sqrt{t L}} + \calO\big(t^{-2}\big)\Big),
	\end{align*}
	and so
	\begin{equation}
		\label{eq:47}
		u_0 = \sqrt{\frac{L}{t}} - \frac{d}{4 t} + \calO\big(t^{-2}\big).
	\end{equation}
	We now establish the asymptotic behavior of $\psi(u_0)$ and $\psi''(u_0)$. By \eqref{lim:phi-2},
	\[
		\frac{\phi\big(t(1-u_0)\big)}{e^{t(1-u_0)}} = \frac1{(m;m)_\infty} + o\big(t^{-1}\big),
	\]
	and
	\[
		\log \phi\big(t(1-u_0)\big) = \log \frac1{(m;m)_\infty} + t(1-u_0) + o\big(t^{-1}\big).
	\]
	By \eqref{eq:48} and \eqref{eq:47} we obtain
	\begin{align*}
		\frac{L}{u_0} &= \frac{d}{4} + \sqrt{L t} +  \calO\big(t^{-1}\big), \\
		t u_0 &= \sqrt{L t} - \frac{d}{4} + \calO\big(t^{-1}\big), \\
		\log u_0 &=\log \sqrt{\frac{L}{t}}+\calO\big(t^{-1}\big).
	\end{align*}
	Therefore,
	\[
		\log \phi\big(t(1-u_0)\big) = \log \frac1{(m;m)_\infty}+ t - \sqrt{L t} + \frac{d}{4} + \calO\big(t^{-1}\big),
	\]
	and
	\begin{align}
		\nonumber
		\psi(u_0)
		&= -\frac{d}{2} \log u_0 - \frac{L}{u_0} + \log \phi\big(t(1-u_0)\big) \\
		\label{eq:39}
		&= \log \frac1{(m;m)_\infty} -\frac{d}{2} \log \sqrt{\frac{L}{t}} - 2 \sqrt{L t}  + t + \calO\big(t^{-1}\big).
	\end{align}
	Using \eqref{eq:48} and \eqref{lim:phi-1} supported by \eqref{eq:11} we can write
	\begin{align}
		\psi''(u_0)
		&= \frac{d}{2} \frac{1}{u_0^2} - 2 \frac{L}{u_0^3} + t^2
		\bigg\{\frac{\phi''\big(t(1-u_0)\big)}{\phi\big(t(1-u_0)\big)} -
		\bigg(\frac{\phi'\big(t(1-u_0)\big)}{\phi\big(t(1-u_0)\big)} \bigg)^2\bigg\} \nonumber  \\
		\nonumber
		&= \frac{d}{2} \frac{t}{L} - 2 t \sqrt{\frac{t}{L}} + \calO\big(t^{-1}\big) \\
		&= -2t \sqrt{\frac{t}{L}} \Big( 1 + \calO\big(t^{-1}\big)\Big).\label{eq:psi_bis_asymp}
	\end{align}
	We are now ready to study the asymptotic behavior of the transition density. By \eqref{eq:p_gauss},
	\eqref{eq:30} and \eqref{psi}, we can write
	\begin{equation}
		\label{eq:61}
		e^t (4 \pi t)^{\frac{d}{2}} p(t; 0, y)
		=
		e^{-L} + t \int_0^m u^{-\frac{d}{2}} e^{-\frac{|y|^2}{4 tu}+\log \frac{\Phi(t, u)}{t}} {\: \rm d} u
		+ t I
	\end{equation}
	where
	\[
		I = \int_m^1 e^{\psi(u)} {\: \rm d} u.
	\]
	Since
	\[
		m^2 + \delta \leq \frac{L}{t} \leq 1 - \delta,
	\]
	and $L > d/2$, by Lemma \ref{lem:1} for all $u \in (m, 1)$,
	\begin{align}
		\label{eq:53}
		\psi''(u)
		&< \frac{d}{2} \frac{1}{u^2} - \frac{L}{u^2}- L\leq - L
		\leq -m^2 t.
	\end{align}
	Our aim is to find the asymptotic behavior of $I$. Let us first focus on the integral over
	$(m, u_0 - \eta) \cup (u_0 + \eta, 1)$, for certain $\eta > 0$. By the Taylor's theorem and \eqref{eq:53}, for
	$u \in (m, 1)$ we get
	\begin{align*}
		\psi(u)
		&= \psi(u_0) + \frac{1}{2} \psi''\big(\lambda u + (1-\lambda) u_0 \big) (u - u_0)^2 \\
		&\leq \psi(u_0) - \frac{m^2}{2} t (u-u_0)^2.
	\end{align*}
	Therefore, using \eqref{eq:psi_bis_asymp},
	\begin{align}
		\nonumber
		\bigg(\int_m^{u_0-\eta } + \int_{u_0+\eta}^1 \bigg) e^{\psi(u)} {\: \rm d} u
		&\leq e^{\psi(u_0)}
		\bigg(\int_m^{u_0-\eta } + \int_{u_0+\eta}^1 \bigg) e^{-\frac{m^2}{2}t  (u-u_0)^2} {\: \rm d} u \\
		\nonumber
		& \leq
		e^{\psi(u_0)} e^{-\frac{m^2 \eta^2}{2} t}\\
		&\leq C t^{-1} e^{\psi(u_0)} (-\psi''(u_0))^{-\frac12}.
		\label{eq:51}
	\end{align}
	Next, we tread the integral over $(u_0-\eta,u_0+\eta)$. We write
	\begin{equation}
		\label{eq:45}
		\int_{\abs{u-u_0} \leq \eta} e^{\psi(u)} {\: \rm d} u
		=
		e^{\psi(u_0)} \int_{\abs{u - u_0} \leq \eta} e^{\frac{1}{2} \psi''(u_0) (u-u_0)^2} e^{\theta(u)} {\: \rm d} u
	\end{equation}
	where
	\[
		\theta(u) = \psi(u) - \psi(u_0) - \frac{1}{2} \psi''(u_0) (u-u_0)^2.
	\]
	For $k \geq 2$, we compute
	\[
		\psi^{(k)}(u) = (-1)^{k} k! \frac{d}{2} \frac{1}{u^k} + (-1)^{k+1} k! \frac{L}{u^{k+1}}
		+ (-1)^k t^k (\log \phi)^{(k)}\big(t(1-u)\big).
	\]
	Hence, by Proposition~\ref{prop:log_phi}, there is $C_k$ such that for all $u \in (m, u_0+\eta)$ and $t \geq 1$,
	\begin{equation}
		\label{eq:41}
		\big| \psi^{(k)}(u) \big| \leq C_k t,
	\end{equation}
	thus by the Taylor's theorem
	\begin{align}
		\label{eq:54}
		\abs{\theta(u)}
		\leq \frac{C_3}{3!} t \abs{u-u_0}^3.
	\end{align}
	If $\abs{u - u_0} \leq \eta$, by \eqref{eq:53}, we obtain
	\begin{equation}
		\label{eq:46}
		\abs{\theta(u)} \leq -\frac{1}{4} \psi''(u_0) \abs{u - u_0}^2
	\end{equation}
	provided that
	\[
		\eta \leq m^2 \frac{3}{2 C_3}.
	\]
	Another application of the Taylor's theorem together with \eqref{eq:41} gives
	\begin{equation}
		\label{eq:44}
		\Big|
		\theta(u) - \frac{1}{3!} \psi'''(u_0) (u-u_0)^3
		\Big|
		\leq
		\frac{C_4}{4!} t \abs{u-u_0}^4.
	\end{equation}
	Now, we write
	\begin{align*}
		e^{\theta(u)}
		&= \Big(e^{\theta(u)} - 1 - \theta(u)\Big) \\
		&\phantom{=}+
		\Big(\theta(u) - \frac{1}{3!} \psi'''(u_0)(u-u_0)^3\Big) \\
		&\phantom{=}+
		\frac{1}{3!} \psi'''(u_0)(u-u_0)^3 + 1,
	\end{align*}
	and we split the integral \eqref{eq:45} into four corresponding integrals. Observe that for $x \in \RR$,
	\begin{equation}
		\label{eq:49}
		\big| e^x - 1 - x\big| \leq \frac{1}{2} x^2 e^{|x|}.
	\end{equation}
	Hence, by \eqref{eq:54} and \eqref{eq:46}, the first integrand can be bounded as follows
	\begin{align*}
		\Big|e^{\theta(u)} - 1 - \theta(u) \Big|
		\leq
		C t^2 \abs{u-u_0}^6 e^{-\frac{1}{4} \psi''(u_0)(u-u_0)^2}.
	\end{align*}
	Therefore, by \eqref{eq:53},
	\begin{align*}
		&\bigg|
		\int_{\abs{u-u_0} \leq \eta}
		e^{\frac{1}{2}\psi''(u_0) (u-u_0)^2} \big(e^{\theta(u)} - 1 - \theta(u) \big) {\: \rm d} u
		\bigg| \\
		&\qquad\qquad\leq
		C t^2 \int_{\abs{u - u_0} \leq \eta} e^{\frac{1}{4} \psi''(u_0)(u-u_0)^2} \abs{u-u_0}^6 {\: \rm d} u \\
		&\qquad\qquad =
		C t^2 (-\psi''(u_0))^{-3-\frac{1}{2}}
		\int_{\abs{u} \leq \eta \sqrt{-\psi''(u_0)}} e^{-\frac{1}{4} \abs{u}^2} \abs{u}^6 {\: \rm d} u
		\\
		&\qquad\qquad\leq C t^{-1} (-\psi''(u_0))^{-\frac{1}{2}}.
	\end{align*}
	For the second integral, by \eqref{eq:44} and \eqref{eq:53} we obtain
	\begin{align*}
		&\bigg|
		\int_{\abs{u - u_0} \leq \eta} e^{\frac{1}{2}\psi''(u_0) (u-u_0)^2}
		\Big(\theta(u) - \frac{1}{3!} \psi'''(u_0)(u-u_0)^3 \Big) {\: \rm d} u
		\bigg|\\
		&\qquad\qquad\leq
		C t \int_{\abs{u-u_0} \leq \eta}  e^{\frac{1}{2}\psi''(u_0) (u-u_0)^2} \abs{u-u_0}^4 {\: \rm d} u\\
		&\qquad\qquad\leq
		C t (-\psi''(u_0))^{-2 - \frac{1}{2}} \int_{\abs{u} \leq \eta \sqrt{-\psi''(u_0)}}
		e^{-\frac{1}{2} \abs{u}^2} \abs{u}^4 {\: \rm d} u \\
		&\qquad\qquad\leq
		C t^{-1} (\psi''(u_0))^{-\frac{1}{2}}.
	\end{align*}
	The third integral equals zero. Lastly, by \eqref{eq:53}, we have
	\begin{align*}
		\int_{\abs{u-u_0} \leq \eta}e^{\frac{1}{2}\psi''(u_0) (u-u_0)^2} {\: \rm d}u
		&=(-\psi''(u_0))^{-\frac{1}{2}}
		\int_{\abs{u} \leq \eta \sqrt{-\psi''(u_0)}} e^{-\frac{1}{2} \abs{u}^2} {\: \rm d} u\\
		&=(-\psi''(u_0))^{-\frac{1}{2}} \left(\sqrt{2\pi} + \int_{\abs{u}
		\geq \eta \sqrt{-\psi''(u_0)}} e^{-\frac{1}{2} \abs{u}^2} {\: \rm d} u \right)\\
		&= (-\psi''(u_0))^{-\frac{1}{2}} \sqrt{2\pi} \Big(1 + \calO(t^{-1})\Big).
	\end{align*}
	Summarizing, we get
	\begin{align*}
		I = e^{\psi(u_0)} (-\psi''(u_0))^{-\frac{1}{2}} \sqrt{2\pi} \Big(1 + \calO(t^{-1})\Big).
	\end{align*}
	Hence, by \eqref{eq:39} and \eqref{eq:psi_bis_asymp},
	\begin{align*}
		I
		&=
		\frac1{(m;m)_\infty} \Big(\frac{L}{t}\Big)^{-\frac{d}{4}}
		e^{- 2 \sqrt{L t} + t} \Big( 2t \sqrt{\frac{t}{L}} \Big)^{-\frac12} \sqrt{2\pi}
		\Big(1 + \calO(t^{-1})\Big)\\
		&= \frac1{(m;m)_\infty} (2 \pi)^{-\frac{d-1}{2}} |y|^{-\frac{d-1}{2}}e^{-|y|} \Big(1 + \calO(t^{-1})\Big).
	\end{align*}
	Moreover, by \eqref{eq:61} and Lemma \ref{lem:1} we get
	\begin{align*}
		\Big|
		p(t; 0, y) - e^{-t} (4 \pi t)^{-\frac{d}{2}}e^{-L} -
		t e^{-t} (4 \pi t)^{-\frac{d}{2}} I
		\Big|
		\leq
		m t^{1-\frac{d}{2}} e^{-t+\psi(m)}.
	\end{align*}
	Therefore, to finish the proof we need to show that
	\begin{equation}
		\label{eq:62}
		t^{-\frac{d}{2}}e^{-t-L}
		=
		\abs{y}^{-\frac{d-1}{2}} e^{-\abs{y}} \calO(t^{-1}),
	\end{equation}
	and
	\begin{equation}
		\label{eq:63}
		t^{1-\frac{d}{2}}e^{-t+\psi(m)}
		=
		\abs{y}^{-\frac{d-1}{2}} e^{-\abs{y}} \calO(t^{-1}).
	\end{equation}
	Notice that $t^{-\frac{d}{2}}$ is comparable to $t^{-\frac12}|y|^{-\frac{d-1}{2}}$. Since $L \leq t(1-\delta)$,
	we have
	\[
		\Big(\sqrt{\frac{L}{t}} - 1 \Big)^2 \geq (1-\sqrt{1-\delta})^2
	\]
	which implies that
	\[
		-t - L \leq -\norm{y}- (1-\sqrt{1-\delta})^2 t
	\]
	proving \eqref{eq:62}. Since $t^{1-\frac{d}{2}}$ is comparable to $t^{\frac12}|y|^{-\frac{d-1}{2}}$, and
	$L \geq t(m^2 + \delta)$, we get
	\[
		\Big(\sqrt{\frac{L}{t}} - m \Big)^2 \geq (\sqrt{m^2+\delta}-m)^2=\delta^*
	\]
	which implies that
	\[
		- \frac{L}{m}- m t \leq -\norm{y}- \frac{(\sqrt{m^2+\delta}-m)^2}{m} t.
	\]
	In view of \eqref{psi} and \eqref{eq:10}, we have
	\[
		-t+\psi(m)\leq -\frac{L}{m} -mt  +C,
	\]
	which proves \eqref{eq:63}, and the theorem follows.
\end{proof}

\begin{remark}
	We note that the asymptotic behavior of $p(t; 0, y)$ in Theorem~\ref{thm:6} remain true if $m^2+\delta \leq L/t$
	is replaced by
	\[
		\Big(m+\sqrt{\frac{\log f(t)}{t}} \Big)^2
		\leq \frac{L}{t}
	\]
	where $f$ is a fixed positive function such that
	\[
		\liminf_{t\to\infty} t^{-\frac{3m}{2}} f(t) >0.
	\]
	Indeed, Lemma~\ref{lem:psi_prim} still applies, and the only change required in the proof is in showing that
	\[
		t^{\frac12} e^{-t+\psi(m)}=e^{-|y|}\calO(t^{-1}),
	\]
	or equivalently
	\[
		t^{\frac12} \exp\Big\{-\frac{t}{m}\Big( \sqrt{\frac{L}{t}} -m\Big)^2\Big\} =\calO(t^{-1}),
	\]
	which is guaranteed by the proposed condition.
\end{remark}

\begin{remark}
	\label{nessgeneralgaussian}
	Theorems~\ref{thm:5} and~\ref{thm:6} can be easily generalized to the case of
	$Y_t=\Upsilon B_t$ for $\Upsilon \in \GL(\mathbb{R}, d)$. In such a case, the counterpart of \eqref{eq:p_gauss} is given by
	\[
		p(t; 0, y) = e^{-t} (\det \Upsilon)^{-1} (2\pi t)^{-\frac{d}{2}}  e^{-\frac{|\Upsilon^{-1} y|^2}{4t}}
		+t e^{-t}(4\pi t)^{-\frac{d}{2}} \int_0^1 u^{-\frac{d}{2}} e^{-\frac{|\Upsilon^{-1} y|^2}{4 tu}+
		\log \frac{\Phi(t, u)}{t}} {\: \rm d} u.
	\]
	Then $\eqref{psi}$ takes the form
	\[
		\psi(u) = -\frac{d}{2} \log u -\frac{L}{u} + \log \phi\big(t(1-u)\big)
	\]
	where
	\[
		L = \frac{|\Upsilon^{-1}y|^2}{2 t},
	\]
	cf. \eqref{A}. The statements of Theorems~\ref{thm:5} and~\ref{thm:6} remain true if we replace $\frac{|y|^2}{4 t}$
	by $\frac{|\Upsilon^{-1}y|^2}{2 t}$. All the proofs are the same.
\end{remark}

\appendix

\section{Strictly stable processes}
\label{appendix:A}
We begin with a general observation on L{\'e}vy process in $\mathbb{R}^d$. By \cite[Theorem 8.1]{MR1739520} there is
a one-to-one correspondence between L{\'e}vy processes and generating triplets $(A, \nu, \gamma)$ where $A$ is nonnegative
definite real $d \times d$ matrix, $\nu$ a Borel measure on $\RR^d$ such that
\[
	\int_{\RR^d} \big(1 \land \norm{x}^2\big) \nu({\rm d} x) < \infty,
\]
and $\gamma \in \RR^d$. The triplet are helpful in identification of the generator of the process, see
\cite[Theorem 31.5]{MR1739520}, which for sufficiently smooth functions $f$ can be given by the following expression
\begin{align*}
	\mathscr{L} f(x)
	&= \sum_{j, k = 1}^d a_{jk} \partial_{x_j} \partial_{x_k} f(x) + \sprod{\nabla f(x)}{\gamma}
	+
	\int_{\RR^d} \Big(f(x+z) - f(x) - \ind{\{\abs{z} < 1\}} \sprod{\nabla f(x)}{z}\Big) \nu({\rm d} z).
\end{align*}
We are going to describe the infinitesimal generator of the semigroup associated with the process considered
on $\calC_0(\RR^d)$ as well as $L^r(\RR^d)$, $r \in [1, \infty)$. Let us recall that
\[
	\mathcal{B}_r =
	\begin{cases}
		L^r(\RR^d) &\text{if } r \in [1, \infty), \\
		\calC_0(\RR^d) &\text{if } r = \infty.
	\end{cases}
\]
\begin{proposition}
	\label{prop:9}
	Let $\mathbf{Y}$ be a L{\'e}vy process in $\mathbb{R}^d$ and $r \in [1, \infty]$. Then the family of operators defined
	on $\mathcal{B}_r$ as
	\[
		\Big(f \mapsto \EE[f(Y_t + \cdot)]: t > 0\Big)
	\]
	forms a strongly continuous contraction semigroup. Let $\mathscr{L}_r$ be its infinitesimal generator with domain
	$D(\mathscr{L}_r)$. Moreover, if $f \in \calC_0^2(\RR^d)$ is such that $\partial_x^{\mathbf{a}} f \in \mathcal{B}_r$
	for each $\mathbf{a} \in \NN_0^d$, $\abs{\mathbf{a}} \leq 2$, then $f \in D(\mathscr{L}_r)$ and
	\[
		\mathscr{L}_r f = \mathscr{L} f.
	\]
\end{proposition}
\begin{proof}
	It is well-known that the contractivity and strong continuity are consequences of Minkowski's integral inequality and
	the continuity of the shift operator, while the semigroup property follows by the Markov property of the process.
	Since for $f \in \calC_0^2(\RR^d)$,
	\[
		\big\|\mathscr{L}f \big\|_{\mathcal{B}_r} \leq c
		\sum_{|\mathbf{a}|\leq 2} \left\|\partial^{\mathbf{a}} f \right \|_{ \mathcal{B}_r}
	\]
	we have $\mathscr{L}f \in \mathcal{B}_r$, because $\partial^{\mathbf{a}} f \in \mathcal{B}_r$ for each
	$\mathbf{a} \in \NN_0^d$, $\abs{\mathbf{a}} \leq 2$. Next, by \cite[Theorem 31.5]{MR1739520}, we get
	\[
		\bigg\|
		\frac{\EE\big[f(Y_t + x)\big]-f(x)}{t} - \mathscr{L} f(x)
		\bigg\|_{\mathcal{B}_r(x)}
		\leq
		\frac1{t} \int_0^t
		\Big\|
		\EE\big[\mathscr{L}f(Y_s + x)- \mathscr{L} f(x) \big]
		\Big\|_{\mathcal{B}_r(x)} {\rm d}s.
	\]
	By the strong continuity in $\mathcal{B}_r(x)$, the right hand-side in the last inequality
	converges to zero as $t$ tends to zero.
\end{proof}

L{\'e}vy processes in $\RR^d$ are  in one-to-one correspondence with their characteristic exponents,
see e.g. \cite[Theorems 8.1 and 7.10, and Corollaries 8.3 and 11.6]{MR1739520}, that is
\[
	\EE e^{i \sprod{x}{Y_t}}=e^{-t\Psi(x)}.
\] Non-zero strictly stable processes in $\RR^d$ form a subclass of L{\'e}vy processes in $\RR^d$ indexed
by $\alpha\in(0,2]$. For every non-zero strictly $\alpha$-stable process $\mathbf{Y}$ in $\RR^d$, $\alpha \in (0, 2]$,
there are unique $A$, $\lambda$ and $\gamma$ such that the characteristic exponent $\Psi$ has the form:
\begin{enumerate}[label=\rm (A.\roman*), start=1, ref=A.\roman*]
	\item
	\label{en:4:1}
	if $\alpha=2$, then
	\[
		\Psi(x)=\sprod{x}{Ax}
	\]
	where $A$ is a non-zero $d\times d$ symmetric non-negative definite real matrix;
	\item
	\label{en:4:2}
	if $\alpha \in (1, 2)$, then
	\[
		\Psi(x)=\int_{\mathbb{S}}
		\lambda({\rm d}\xi)
		\int_0^\infty \big(1-e^{i\sprod{x}{r\xi}} +i \sprod{x}{r\xi} \big) \frac{{\rm d}r}{r^{1+\alpha}}
	\]
	where $\lambda$ is a non-zero finite measure on the unit sphere $\mathbb{S}$;
	\item
	\label{en:4:3}
	if $\alpha=1$,
	\[
		\Psi(x)=\int_{\mathbb{S}}\lambda({\rm d}\xi)
		\int_0^\infty \big(1-e^{i\sprod{x}{r\xi}} + i \sprod{x}{r\xi} \ind{(0,1]}(r)\big)
		\frac{{\rm d}r}{r^{1+\alpha}}-i\sprod{\gamma}{x}
	\]
	where either $\gamma\in\RR^d$ and $\lambda$ is a non-zero finite measure on the unit sphere $\mathbb{S}$ satisfying
	\[
		\int_{\mathbb{S}} \xi\, \lambda({\rm d}\xi)=0
	\]
	or $\gamma \neq 0$ and $\lambda\equiv 0$;
	\item
	\label{en:4:4}
	if $\alpha \in (0, 1)$,
	\[
		\Psi(x)=\int_{\mathbb{S}}\lambda({\rm d}\xi)
		\int_0^\infty \big(1-e^{i\sprod{x}{r\xi}} \big) \frac{{\rm d}r}{r^{1+\alpha}}
	\]
	where $\lambda$ is a non-zero finite measure on the unit sphere $\mathbb{S}$.
\end{enumerate}
Conversely, for every function $\Psi$ of the form as in \eqref{en:4:1}--\eqref{en:4:4} there is a unique non-zero strictly
$\alpha$-stable process $\mathbf{Y}$ in $\RR^d$ with $\alpha\in(0,2]$.

Let us observe that the only non-zero trivial L{\'e}vy process is a non-zero constant drift, that is
$\alpha=1$ with $\gamma\neq 0$ and $\lambda \equiv 0$. Thus, an alternative representation of characteristic exponents of
non-zero strictly stable processes is a consequence of \cite[Theorem 14.10 and Example 18.8]{MR1739520}. The classification
\eqref{en:4:1}--\eqref{en:4:4} is a summary of the following results: \cite[Theorem 14.2 and Example 2.10]{MR1739520} for
$\alpha=2$ and \cite[Theorem 14.7, Remarks 14.6 and 14.4]{MR1739520} for $\alpha \in (0, 2)$. For precise definitions
of a strictly stable process see \cite[Definitions 13.2 and 13.16]{MR1739520}, and a  non-zero and nontrivial
process see \cite[Definition 13.6]{MR1739520}. Let us recall that for a give non-zero strictly $\alpha$-stable
process, $\alpha \in (0, 2]$, the generating triplet is identified as follows:
\begin{enumerate}[label=\rm (A.\roman*), start=1, ref=A.\roman*]
	\item
	if $\alpha = 2$, then take $(A, 0, 0)$;
	\item
	if $\alpha \in (1, 2)$, then take $(0, \nu, \gamma)$ where
	\[
		\gamma = -\int_{\{\norm{z} > 1\}} z \nu({\rm d} z);
	\]
	\item
	if $\alpha = 1$, then we take $(0, \nu, \gamma)$;
	\item
	if $\alpha \in (0, 1)$, we take $(0, \nu, \gamma)$ where
	\[
		\gamma = \int_{\{ |z| \leq 1 \}}z \nu({\rm d}z);
	\]
\end{enumerate}
In view of \cite[Theorem 14.3]{MR1739520},
\begin{equation}
	\label{eq:nu-str_st}
	\nu(B)=\int_{\mathbb{S}}\lambda({\rm d}\xi)
	\int_0^\infty \ind{B} (r\xi) \frac{{\rm d}r}{r^{1+\alpha}}, \qquad B \in \mathcal{B}(\RR^d).
\end{equation}
In our studies, we shall assume that for each $t > 0$, the distribution of $Y_t$ is absolutely continuous with respect to
the Lebesgue measure. In fact, for strictly stable processes, it is equivalent to the absolute continuity of $Y_1$.

The proof of the following lemma is inspired by \cite[Proposition 24.20 and Example 37.19]{MR1739520}.
For the definition of non-degenerate process see \cite[Definitions 24.16 and 24.18]{MR1739520}.
\begin{lemma}
	\label{lem:density}
	Let $\mathbf{Y}$ be a strictly $\alpha$-stable process in $\RR^d$, $\alpha\in(0,2]$. The following statements are equivalent:
	\begin{enumerate}[label=\rm (\roman*), start=1, ref=\roman*]
		\item
		\label{en:1:1}
		the distribution of $Y_1$ is absolutely continuous with respect to the Lebesgue measure;
		\item
		\label{en:1:2}
		$\mathbf{Y}$ is non-degenerate;
		\item
		\label{en:1:3}
		if $\alpha = 2$, then $\det A \neq 0$; otherwise for each $x \in \RR^d \setminus \{0\}$,
		\begin{equation}
			\label{eq:60}
			\int_{\mathbb{S}} |\sprod{x}{\xi}| \: \lambda({\rm d}\xi) \neq 0;
		\end{equation}
		\item
		\label{en:1:4}
		$\Re ( \Psi(x)) \approx |x|^{\alpha}$.
	\end{enumerate}
\end{lemma}
\begin{proof}
	Each of the conditions \eqref{en:1:1}--\eqref{en:1:4} implies that $\mathbf{Y}$ is non-zero. Clearly,
	\eqref{en:1:1} $\implies$ \eqref{en:1:2} follows from the definition of the non-degeneracy.
	Next, let us assume that \eqref{en:1:2} holds true. If $\alpha=2$, then by \cite[Proposition 24.17]{MR1739520} we must
	have $\det(A)\neq 0$. If $\alpha\in(0,2)$, by \cite[Theorem 14.10 and Example 18.8]{MR1739520}, we have that
	\begin{equation}
		\label{eq:59}
		\Re (\Psi(x)) = \int_{\mathbb{S}} |\sprod{x}{\xi}|^{\alpha} \lambda_1({\rm d}\xi)
	\end{equation}
	where $\lambda_1$ is a positive measure proportional to $\lambda$. Suppose, contrary to our claim, that there is
	$x_0 \in \RR^d \setminus\{0\}$, such that $\Re (\Psi(x_0)) = 0$. Hence, the measure $\lambda_1$ as well as the
	L{\'e}vy measure $\nu$, are supported on $x_0^\perp$. This contradicts the
	non-degeneracy of $\mathbf{Y}$, see \cite[Proposition 24.17]{MR1739520}. This proves \eqref{en:1:2} $\implies$
	\eqref{en:1:3}.
	
	Let us observe that for $x \in \RR^d \setminus \{0\}$,
	\[
		\Re \Phi(x) = \norm{x}^\alpha \Re \Psi\Big(\tfrac{x}{\norm{x}} \Big).
	\]
	Indeed, it follows immediately from \eqref{en:4:1} for $\alpha = 2$, and from \eqref{eq:59} for $\alpha \in (0, 2)$.
	By the continuity and compactness of the unit sphere, $\Re \Phi$ attains its extremes. Therefore, to show that \eqref{en:1:3}
	$\implies$ \eqref{en:1:4}, it is enough to prove that the minimum on $\mathbb{S}$ is positive. For $\alpha \in (0, 2)$, it
	is an easy consequence of \eqref{eq:60}. If $\alpha = 2$, since the matrix $A$ is symmetric and non-negative definite, the
	condition $\det A \neq 0$ implies that it is in fact positive definite, proving \eqref{en:1:4}.
	
	Finally, \eqref{en:1:4} $\implies$ \eqref{en:1:1} easily follows by the Fourier inversion formula.
\end{proof}

\begin{remark}
	\label{rem:density}
	If $d=1$, the only non-zero strictly $\alpha$-stable process which is \emph{degenerate} is a non-zero drift, that is,
	$\alpha=1$, $\gamma \neq 0$ and $\lambda \equiv 0$. Hence, by Lemma~\ref{lem:density}, for every non-zero strictly
	$\alpha$-stable process other than drift, $Y_1$ is absolutely continuous with respect to the Lebesgue measure.
\end{remark}
If any of the conditions \eqref{en:1:1}--\eqref{en:1:4} in Lemma \ref{lem:density} is satisfied, then for each $t > 0$
the random variable $Y_t$ has absolutely continuous density $p_0(t; 0, \cdot)$. Moreover, $p_0(u; 0, \cdot)$ belongs to
$\calC_0^\infty(\RR^d)$, and
\begin{align}
	\label{eq:p_0}
	p_0(t;x,y)= \bigg(\frac{1}{2\pi}\bigg)^d
	\int_{\RR^d} e^{-i\sprod{y-x}{z}} e^{-t \Psi(z)} {\: \rm d}z,
\end{align}
see \cite{Hartman, MR3010850, MR4140542} and \cite[Section 6]{MR3996792} for a broader context.
\begin{lemma}
	\label{lem:A3}
	Let $\mathbf{Y}$ be a strictly $\alpha$-stable process in $\RR^d$, $\alpha\in(0,2]$, with a density $p_0$.
	\begin{enumerate}[label=\rm (\roman*), start=1, ref=\roman*]
		\item
		\label{en:3:1}
		For all $x,y\in\RR^d$ and $t,u>0$,
		\[
			p_0(tu; x, y)=t^{-d/\alpha}p_0(u;t^{-1/\alpha}x,t^{-1/\alpha}y).
		\]
		\item
		\label{en:3:2}
		For every $\ell \in \NN_0$, $\mathbf{a} \in \NN_0^d$ there is a constant $C > 0$ such that for all $y \in \RR^d$
		and $u>0$,
		\[
			\partial_u^\ell \,\partial_y^\mathbf{a} p(u;0,\cdot) \in \calC_0^\infty\big(\RR^d\big),
		\]
		and
		\[
			\big| \partial_u^\ell \partial_y^\mathbf{a} p(u;0,y)\big|
			\leq C u^{-\ell-(d+|\mathbf{a}|)/\alpha}.
		\]
		\item
		\label{en:3:3}
		There is a constant $C > 0$ such that for all $x,y,z\in\RR^d$ and $u>0$,
		\begin{align*}
			p_0(u;0,y-x)
			\leq C u^{-d/\alpha},
		\end{align*}
		and
		\begin{align*}
			\big|p_0(u;0,y+z)-p_0(u;0,y)\big|
			\leq C |z| u^{-(d+1)/\alpha}.
		\end{align*}
	\end{enumerate}
\end{lemma}
\begin{proof}
	The proof of \eqref{en:3:1} follows from \eqref{eq:p_0} and the scaling property of $\Psi$, that is
	$t\Psi(z)=\Psi(t^{1/\alpha}z)$. By Lemma~\ref{lem:density} and \cite[Proposition~2.17]{MR3156646} we obtain that
	$\Re\Psi(z) \approx |z|^{\alpha}$ and $|\Psi(z)|\leq c (1+|z|^2)$, respectively. Thus, we can differentiate under
	the integral sign in \eqref{eq:p_0}, and the integrands are absolutely integrable. This proves the regularity in
	\eqref{en:3:2}. Now using the scaling of $\Psi$ we get
	\[
		\partial_u^{\ell} \partial_y^\mathbf{a} p_0(u;0,y)=
		u^{-\ell-(d+|\mathbf{a}|)/\alpha}\big(\partial_u^\ell \,\partial_y^\mathbf{a} \, p_0\big)\big(1;0,u^{-1/\alpha}y\big)
	\]
	which together with the above mentioned estimates on $\Psi$ leads to the upper bound in \eqref{en:3:2}.
	Finally, \eqref{en:3:3} easily follows from \eqref{en:3:2}.
\end{proof}

\begin{example}
	\label{ex:names:1}
	Let $\mathbf{Y}$ be an $\alpha$-stable subordinator with $\alpha\in (0,1)$, and let $d = 1$. If
	$\lambda({\rm d}\xi)= \frac{\alpha}{\Gamma\left(1-\alpha\right)} \delta_{1}({\rm d}\xi)$, then
	\[
		\nu({\rm d}s)=\frac{\alpha}{\Gamma\left(1-\alpha\right)} \ind{(0, \infty]}(s) \frac{{\rm d} s}{s^{1+\alpha}}.
	\]
	Moreover, the Laplace transform is $\EE e^{- u Y_t}= e^{-t u^{\alpha}}$, $u \geq 0$. See \cite[Example 24.12]{MR1739520}.
\end{example}

\begin{example}
	\label{ex:names:2}
	Let $\mathbf{Y}$ be an isotropic $\alpha$-stable process with $\alpha\in (0,2)$. If
	$\lambda({\rm d}\xi)=c \lambda_0({\rm d}\xi)$ (and $\gamma=0$ for $\alpha=1$) where $\lambda_0$ is the uniform probability
	measure on the unit sphere $\mathbb{S}$, and
	\[
		c^{-1}=c_0(d,\alpha)=2^{-\alpha}\frac{\Gamma(d/2)\Gamma((2-\alpha)/2)}{\alpha \Gamma((\alpha+d)/2)}.
	\]
	Then
	\[
		\nu({\rm d}z)= \mathscr{A}_{d,\alpha} |z|^{-d-\alpha} {\: \rm d}z,
		\quad\text{ and }\quad
		\mathscr{A}_{d,\alpha} =\frac{2^\alpha \Gamma((d+\alpha)/2)}{\pi^{d/2} |\Gamma(-\alpha/2)|},
	\]
	and $\EE e^{i\sprod{x}{Y_t}}=e^{-t|x|^\alpha}$. See  \cite[Example 18.9]{MR1739520}.
\end{example}

\begin{example}
	\label{ex:names:3}
	Let $\mathbf{Y}$ be a cylindrical $\alpha$-stable process with $\alpha\in (0,2)$. If $d \geq 2$, then
	\[
		\lambda ({\rm d\xi})=\sum_{k=1}^d \Big( \frac{c}{2} \big(\delta_{-1}({\rm d}\xi_k)
		+\delta_{1}({\rm d}\xi_k)\big) \prod_{\substack{j=1 \\ j\neq k}}^d \delta_0 ({\rm d}\xi_j)\Big)
	\]
	where $c^{-1} = c_0(1, \alpha)$. If $\alpha = 1$, we also have $\gamma = 0$. Then
	\[
		\nu({\rm d}z) = \sum_{k=1}^d \Big( \mathscr{A}_{1,\alpha} |z_k|^{-1-\alpha} {\: \rm d}z_k
		\prod_{\substack{j=1 \\ j \neq k}}^d \delta_0 ({\rm d} z_j) \Big),
	\]
	and $\EE e^{i\sprod{x}{Y_t}}=\exp\{-t (|x_1|^\alpha+\ldots +|x_d|^\alpha)\}$.
\end{example}

\begin{example}
	Let $\mathbf{Y}$ be Brownian motion. If $A=I$ (identity matrix), then
	\[
		p_0(t;x,y)=(4\pi t)^{-d/2}e^{-\frac{|y-x|^2}{4t}},
	\]
	and $\EE e^{i\sprod{x}{Y_t}} = \exp\{-t \norm{x}^2\}$.
\end{example}

There is also a formula for the transition density of $\frac{1}{2}$-stable subordinator, see e.g.
\cite[Example 2.13]{MR1739520}. Namely,
\[
	p_0(t;x,y)= (4\pi)^{-1/2} t (y-x)^{-3/2} e^{- \frac{t^2}{4(y-x)}} \ind{\{x<y\}}.
\]
Another example covered in the article with an explicit formula for the transition density
is the Cauchy process, i.e., a L{\'e}vy process in $\RR^d$ with the characteristic function
\[
	\EE e^{i\sprod{x}{Y_t}}=e^{-t(|x|- i \sprod{\gamma}{x})}.
\]
For $\gamma=0$, the formula reduces to the isotropic $1$-stable process, called the \emph{symmetric Cauchy process}. For general
$\gamma\in\RR^d$ it is a strictly $1$-stable process with the density, see e.g. \cite[Example 2.12]{MR1739520},
\[
	p_0(t;x,y)= \frac{\Gamma\big(\tfrac{d+1}{2}\big)}{\pi^{\frac{d+1}{2}}}
	\frac{t}{\big(|y-x-\gamma t|^2+t^2\big)^\frac{d+1}{2} }.
\]

\section{Further comments}
\label{appendix:B}
\subsection{Literature overview}
Resetting describes an evolutionary pattern encountered in a variety of physical phenomena (see \cite{Gupta, MR4093464} for
survey). Apart from the topics already described in Introduction, see Section \ref{sec:Intro}, it models reset algorithms
(see e.g. \cite{Avrach, MR2023017}), and appears as the fluid limit for some queuing models with binomial catastrophe
rates (see e.g. \cite{Artalejo}). Such processes can be viewed as particular examples of the so-called shot noise model.
They are used in modeling earthquakes or layers of sedimentary material accumulated in depositional environments,
but not subject to subsequent erosion (see e.g. \cite{18}). It can also model snow avalanches or neuron firings (see e.g.
\cite{MR1780992, MR2213972, MR2044928,MR1102872}). Applications of resetting procedure have also been discussed
in the context of backtrack recovery in RNA polymerase (see e.g. \cite{4, 9}), enzymatic velocity (see e.g. \cite{16}),
pollination strategies (see e.g. \cite{17}), enzymatic inhibition (see e.g. \cite{2}), stochastic thermodynamics
(see e.g. \cite{21}), and quantum mechanics (see e.g. \cite{22}). Realization of resetting has also been confirmed
in switching holographic optical tweezers (see \cite{8} for details). Partial resetting also models disasters and
catastrophes (see e.g. \cite{MR4189343}), for example, in the context of population dynamics
(see e.g. \cite{Gerber, Assaf, 33, 34}). 
The risk process with partial resetting found to be an indispensable part of the so-called microinsurance policies
(see e.g. \cite{Corina}). In \cite{Rev} it has been pointed out that a process with resetting can be used to model
certain chemical reactions. Additionally, the partial resetting mechanism has also been observed in granular gas where collisions
convert the relative speed of two particles quasi-instantaneously to its fractions (see e.g. \cite{27}), and in a soil
contaminant where rainfall events (or other events) leach the contaminant to its fraction
(see e.g. \cite{28, 29,30, 18}). Other possible applications of the partial resetting procedure
include vegetation biomass with fires (see e.g. \cite{31, 32}), modeling of financial crashes (see, e.g.
\cite{35, 36, Mendoza, Sor}), or queueing with random clearing events (see e.g. \cite{38}).

Much attention has been devoted to the fact that resetting can produce non-equilibrium steady states (see e.g.
\cite{Eule, Malakar, Gupta, Evans, Evans1, MR4093464, MR3476293, Pal, White, frenchNESS}). In most of the papers, the density
of the stationary measure of the process with resetting is described either numerically or in explicit form (e.g. in the case
of Brownian motion with total resetting). Then there are arguments given to prove that  so-called balance conditions are
violated. Roughly, the conditions check if the reversed in time process with resetting seen from the point of view of
invariant law is the same as original process starting at the invariant law.
In \cite{frenchNESS}, in the context of NESS, the authors discussed the duality of the generator for symmetric exclusion process
with reservoirs.

\subsection{Fourier transforms}
Let $Y$ be the isotropic strictly stable process. Let $\mathbf{X}$ be obtained from $\mathbf{Y}$ by partial resetting
with factor $c \in (0, 1)$. Let $p$ denote a transition density of the process $\mathbf{X}$. Following \cite{jaifizycy},
we can formally compute the Laplace--Fourier transform of $p$ as well as the Fourier transform of the density
$\rho_{\mathbf{Y}}$ of the stationary measure for $\mathbf{X}$.

Indeed, applying the Laplace--Fourier transform to both sides of the second equation of Theorem \ref{thm:H+F-P} with
$\mathscr{L}^*f(x)=\mathscr{L}f(x)=\Delta^{\alpha/2}f(x)$ (with usual understanding that for $\alpha=2$ we have
$\Delta^{\alpha/2}=\Delta$) we obtain
\begin{equation*}
	\label{Fp}
	r\mathcal{F}(p)(\theta, r)-1= -\mathcal{F}(p)(\theta, r)\big(|\theta|^\alpha+1\big)+ \mathcal{F}(p)(c\theta, r),
	\quad\text{for } r\geq 0, \theta \in \RR
\end{equation*}
where
\[
	\mathcal{F}(p)(\theta, r)=\int_0^\infty e^{-rt}\mathcal{F}(p(t; 0,\cdot ))(\theta) {\: \rm d} t.
\]
Hence, it is expected that
\begin{equation*}
	\label{Fpa}
	\mathcal{F}(p)(\theta, r)=\sum_{k=0}^\infty \prod_{j=0}^k \frac{1}{1+|\theta|^\alpha m^{j}+r}.
\end{equation*}
Similarly, using the first equation in Theorem \ref{thm:H+F-P} we obtain
\begin{equation*}
	\label{Fpi}
	-\mathcal{F}(\rho_{\mathbf{Y}})(\theta)\big(|\theta|^\alpha+1\big)+ \mathcal{F}(\rho_{\mathbf{Y}})(c\theta)=0,
\end{equation*}
and thus
\begin{equation*}
	\label{Fpb}
	\mathcal{F}(\rho_{\mathbf{Y}})(\theta)=\prod_{j=0}^\infty \frac{1}{1+|\theta|^\alpha m^{j}}.
\end{equation*}
Note that from Theorem~\ref{thm:H+F-P} it follows that
$\rho_{\mathbf{Y}}$ satisfies the pantograph functional-differential equation, which has reach structure
(see e.g. \cite{iserles_1993, WISNIEWOLSKI2024600}).

\subsection{L\'evy process with partial resetting and autoregressive process of order $1$}
Let $\mathbf{Y}$ be a one-dimensional L\'evy process with a transition density $p_0$. Let $\mathbf{X}$ be obtained from
$\mathbf{Y}$ by partial resetting with factor $c \in (0, 1)$. Then the density of the stationary distribution of $\mathbf{X}$
exists and it equals the stationary distribution of the process $\mathbf{X}$ observed at Poisson epochs only.
Indeed, according to \cite[Definition 1]{MR1157423}, $\mathbf{X}$ admits an embedded recursive chain
$\mathbf{Z}=(Z_n =X_{T_n}: n \in \NN_0)$ being an autoregressive process of order $1$ defined via
\[
	\begin{cases}
		Z_0 = 0, \\
		Z_{n+1} =c Z_n + Y_{T_{n+1}-T_n}, &\quad n \in \NN.
	\end{cases}
\]
By \cite[Theorem 1(3), Theorem 8 and Theorem 2]{MR1157423}, it is enough to show that $\mathbf{Z}$ convergence in total
variation towards its stationary law. The latter follows from \cite[Theorem 2.1 and Remark B]{MR1956829} and
\cite[Theorem 2.8]{MR1894253}, since $\mathbf{Z}$ is an aperiodic, positive Harris recurrent Markov chain that has an
absolutely continuous distribution (see also \cite[Theorem 13.0.1]{MR1287609}).

Furthermore, denoting by $X_\infty$ the random variable with distribution being the invariant law of $\mathbf{X}$,
we have the following affine equation satisfied
\begin{equation}
	\label{affine}
	X_\infty \eqdistr cX_\infty+ Y_{T_1}.
\end{equation}
In particular, if $\mathbf{Y}$ is a one-dimensional strictly isotropic $\alpha$-stable process, then by \cite{Grinc} and
\cite[Lemma A.3]{MikoschSamrodnitsky}, as $y\rightarrow +\infty$, we get
\begin{enumerate}
	\item if $\alpha \in (0, 2)$, then
	\begin{align*}
		\mathbb{P} (\pm X_\infty >y)
		&\sim (1-c^{\alpha})^{-1}\mathbb{P}( \pm Y_{T_1}>y) \\
		&\sim \frac{1}{2\alpha \Gamma(-\alpha)\cos \left(\frac{(2-\alpha)\pi}{2}\right)}y^{-\alpha},
		\qquad \text{as } y \rightarrow + \infty;
	\end{align*}
	\item if $\alpha = 2$, then
	\[
		\mathbb{P} (X_\infty >y) \sim \frac{1}{(m,m)_\infty}\mathbb{P} (Y_{T_1} >y)
		\qquad\text{as } y \rightarrow +\infty.
	\]
\end{enumerate}

\subsection{Generalized Ornstein--Uhlenbeck process and exponential functional}
Observe that the one-dimensional process $\mathbf{X}$ obtained from the L\'evy process $\mathbf{Y}$
by partial resetting is a solution of the stochastic differential equation
\begin{equation}
	\label{sde1}
	\mathrm{d}X_t=(c-1)X_{t-}\;\mathrm{d}N_t+{\rm d} Y_t,
\end{equation}
hence it is a special case of the so-called Generalized Ornstein--Uhlenbeck process. Therefore,
$X_\infty$ being the solution of \eqref{affine} has the same law as the law of the following exponential functional
\[
	D=
	\int_0^\infty e^{-\log (1+c)N_{s-}}\mathrm{d} Y_s,
\]
see e.g. \cite{Behme}.

\section*{Acknowledgement}
The second author acknowledges that the research is partially supported by the Polish National Science Centre
Grant No.~2021/41/B/HS4/00599.
\section*{Data availability}
We do not analyse or generate any datasets, because our work proceeds within a theoretical and mathematical approach.
\section*{Conflict of interest}
On behalf of all authors, the corresponding author states that there is no conflict of interest. 

\end{document}